\documentclass[english,reqno,11pt]{amsart}
\usepackage{amsmath, amsthm, amsfonts}
\usepackage{hyperref} 
\usepackage{hyperref}
\hypersetup{
	colorlinks=true,
		citecolor=blue!60!black,
		linkcolor=red!60!black,
		urlcolor=green!40!black,
		filecolor=yellow!50!black,
	breaklinks=true,
	pdfpagemode=UseNone,
	bookmarksopen=false,
}
\usepackage{tcolorbox}
\usepackage{amsmath,amsthm,amssymb}
\usepackage{amscd,indentfirst,epsfig}
\usepackage{latexsym}
\usepackage{times}
\usepackage{enumerate}
\usepackage{mathrsfs}
\usepackage{stmaryrd}
\usepackage{amsopn}
\usepackage{amsmath}
\usepackage{amssymb,dsfont,mathtools}
\usepackage{amsfonts,bm}
\usepackage{amsbsy,amsmath}
\usepackage{amscd}
\usepackage{xcolor}
\usepackage{mathtools}

\DeclarePairedDelimiter\floor{\lfloor}{\rfloor}
\usepackage[mono=false]{libertine}
\usepackage[T1]{fontenc}
\usepackage{amsthm}
\usepackage[normalem]{ulem} 
\usepackage[cal=euler, scr=boondoxo]{}
\usepackage{microtype}

\usepackage{numprint}

\newcommand{\verti}[1]{{\left\vert\kern-0.25ex\left\vert\kern-0.25ex\left\vert #1 
    \right\vert\kern-0.25ex\right\vert\kern-0.25ex\right\vert}}

\newtheorem{theo}{Theorem}[section]
\newtheorem{question}{Question}

\newtheorem{lemme}[theo]{Lemma}
\newtheorem{propo}[theo]{Proposition}
\newtheorem{cor}[theo]{Corollary}
\newtheorem{hyp}[theo]{Assumption}
\newtheorem{defi}[theo]{Definition}
\newtheorem{exe}[theo]{Example}

\newtheorem{nb}[theo]{Remark}

\def \bq {\begin{equation}}
\def \eq {\end{equation}}
\def \leq {\leqslant}
\def \geq {\geqslant}

\def \N {\mathbb{N}}
\def \ind {\mathbf{1}}

\def \S {\mathbb{S}}
\numberwithin{equation}{section}

\def\lp {L^1_+}
\def\lm {L^1_-}

\def \d {\mathrm{d}}
\def \D {\mathscr{D}}

\def \C {\mathbb{C}}
\def \Rs {\mathcal{R}}

\def \ml {M_{\lambda}}

\def \R {\mathbb{R}}

\def \M {\mathcal{M}}
\def \G {\bm{G}}

\def \X {\mathbb{X}}
\def \Y {\mathbb{Y}}
\def \l+ {L^1_+}
\def \l- {L^1_-}

\renewcommand{\epsilon}{\varepsilon}

\def \ds {\displaystyle}
\def \l {\lambda}
\def \T {\mathsf{T}}
\def \B {\mathsf{B}}
\def \e {\varepsilon}
\def \H {\mathsf{H}}

\begin{document}
\title[Diffuse boundary conditions]{Convergence Rate To Equilibrium For Collisionless Transport Equations With Diffuse Boundary Operators: A New Tauberian Approach}

 \author{B. Lods}

 \address{Universit\`{a} degli
 Studi di Torino \& Collegio Carlo Alberto, Department of Economics and Statistics, Corso Unione Sovietica, 218/bis, 10134 Torino, Italy.}\email{bertrand.lods@unito.it}

 \author{M. Mokhtar-Kharroubi}

 \address{Universit\'e de Bourgogne-Franche-Comt\'e, Equipe de Math\'ematiques, CNRS UMR 6623, 16, route de Gray, 25030 Besan\c con Cedex, France
}
\email{mustapha.mokhtar-kharroubi@univ-fcomte.fr}

\maketitle
\begin{abstract}
This paper provides a new tauberian approach to the study of quantitative time asymptotics of collisionless transport semigroups with general diffuse boundary operators. We obtain an (almost) optimal algebraic rate of convergence to equilibrium under very general assumptions on the initial datum and the boundary operator. The rate is prescribed by the maximal gain of integrability that the boundary operator is able to induce. The 
  proof relies on a representation of the collisionless transport semigroups by a (kind of) Dyson-Phillips series and on a fine analysis of the trace on the imaginary axis of Laplace transform of remainders (of large order) of this series. Our construction is systematic and is based on various preliminary results of independent interest.



\end{abstract}

\tableofcontents
\section{Introduction}

We consider here the time asymptotics for collisionless kinetic  equations of the form
\begin{subequations}\label{1}
\begin{equation}\label{1a}
\partial_{t}f(x,v,t) + v \cdot \nabla_{x}f(x,v,t)=0, \qquad (x,v) \in \Omega \times V, \qquad t \geq 0
\end{equation} 
with initial data
\begin{equation}\label{1c}f(x,v,0)=f_0(x,v), \qquad \qquad (x,v) \in \Omega \times V,\end{equation}
under \emph{diffuse  boundary}
\begin{equation}\label{1b}
f_{|\Gamma_-}=\mathsf{H}(f_{|\Gamma_+}),
\end{equation}\end{subequations}
where $\Omega$ is a bounded open subset of $\R^{d}$ and $V$ is a  given closed subset of $\R^{d}$ (see Assumptions \ref{hypO} for major details), 
$$\Gamma _{\pm }=\left\{ (x,v)\in \partial \Omega \times V;\ \pm
v \cdot n(x)>0\right\}$$
($n(x)$ {being} the outward unit normal at $x\in \partial \Omega$)  and $\mathsf{H}$  {is
a linear 
boundary operator relating the
outgoing and incoming fluxes $f_{\mid \Gamma _{+}}$ and $f_{\mid
\Gamma _{-}}$ and is  bounded   on
the trace spaces
$$L^{1}_{\pm}=L^{1}(\Gamma_{\pm}\,;\,|v\cdot n(x)| \pi(\d  x)\otimes \bm{m}(\d v))=L^{1}(\Gamma_{\pm},\d\mu_{\pm}(x,v))$$
where $\pi$ denotes the Lebesgue surface measure on $\partial \Omega$ and $\bm{m}$ is a Borel measure on the set of velocities (see Assumptions \ref{hypO} hereafter). The
boundary operator }%
$$\H\::\:\lp \rightarrow \lm$$
is nonnegative and stochastic, i.e.
\begin{equation}\label{eq:H1m}
\int_{\Gamma_{-}}\H\psi\,\d\mu_{-}=\int_{\Gamma_{+}}\psi\,\d\mu_{+}, \qquad \forall \psi \in L^{1}(\Gamma_{+},\d\mu_{+})\end{equation}
 so that \eqref{1} is 
governed by a stochastic  $C_{0}$-semigroup $\left(
U_{\H}(t)\right) _{t\geq 0\text{ }}$ on $L^{1}(\Omega \times V\,,\,\d x\otimes %
\bm{m}(\d v))$ with generator $\T_{\H}.$ 

In a previous contribution \cite{LMR}, a systematic study of \eqref{1} for 
 general \emph{partly diffuse boundary
operators} $\H$ have been performed providing  a general theory on the existence of an invariant density and its asymptotic
stability (i.e. convergence to equilibrium), see also earlier one-dimensional
results \cite{MKR}. However, the question of the \emph{rate of convergence to equilibrium} has been left open by our contribution \cite{LMR} and is the main concern of the present paper.

\subsection{Our contribution in a nutshell}\label{sec:nut}

The main question addressed in this paper is then the following:
\begin{question}\label{ques}
Determine  a general class $\mathcal{C} \subset L^1(\Omega\times V)$ of initial datum $f$ and a general rate function $r\::\:\R^{+}\to\R^{+}$ such that
\begin{equation}\label{eq:rate}
\left\|U_{\H}(t)f-\varrho_{f}\Psi_{\H}\right\|_{L^{1}(\Omega\times V)}=\mathbf{O}(r(t)) \qquad \text{ as } \quad t \to 0^{+} \quad \text{ for any } f\in \mathcal{C}\end{equation}
where  $\Psi_{\H}$ is the unique invariant density of $\T_{\H}$ with unit mass (see the subsequent Theorem \ref{theo:LMR}), 
$$\varrho_{f}=\int_{\Omega\times V}f(x,v)\d x \otimes \bm{m}(\d v), \qquad f \in L^1(\Omega\times V)$$ 
and $\lim_{t\to0}r(t)=0.$
\end{question}
We answer this question here by considering only \textit{diffuse} boundary operators for which, typically,
\begin{equation}\label{eq:Hhkernel}
\H\psi(x,v)=\int_{v'\cdot n(x) > 0}\bm{k}(x,v,v')\psi(x,v')\,|v'\cdot n(x)|\bm{m}(\d v'), \qquad (x,v) \in \Gamma_{-}\end{equation}
where,
\begin{equation}\label{eq:normalise}\int_{v \cdot n(x) <0}\bm{k}(x,v,v^{\prime })|v\cdot n(x)|\bm{m}(\d v)=1,\ \quad (x,v')\in \Gamma
_{+}\,.\end{equation}
We do not consider the case where the velocities are bounded away from
zero which deserves a separate analysis, mainly because in this case $\left( U_{\H}(t)\right) _{t\geq 0\text{ }}$ exhibits a
spectral gap and the convergence to equilibrium is exponential \cite{LM-iso}. Let us describe more precisely our mathematical framework and the set of assumptions we adopt throughout the paper. First, the general assumptions on the phase space are  the following

\begin{hyp}\label{hypO} The phase space $\Omega \times V$ is such that
\begin{enumerate} 
\item $\Omega
\subset \R^{d}$ $(d\geq 2)$ is an open and \emph{bounded} subset with $\mathcal{C}^{1}$ boundary $\partial \Omega 
$.
\item  $V \subset \R^{d}$ is the support of a nonnegative Borel measure $\bm{m}$  which is orthogonally invariant (i.e. invariant under the action of the orthogonal group of matrices in $\R^{d}$).
\item $0 \in V$, $\bm{m}(\{0\})=0$ and $\bm{m}\left(V \cap B(0,r)\right) >0$ for any $r >0$ where $B(0,r)=\{v \in \R^{d}\,,\,|v| < r\}.$
\end{enumerate}
We denote by 
$$\X_{0}:=L^{1}(\Omega \times V\,,\,\d x\otimes \bm{m}(\d v))$$
endowed with its usual norm $\|\cdot\|_{\X_{0}}.$ More generally, for any $s \geq0$, we set
$$\X_{s}:=L^{1}(\Omega \times V\,,\,\max(1,|v|^{-s})\d x \otimes \bm{m}(\d v))$$
with norm $\|\cdot\|_{\X_{s}}.$
\end{hyp}
Notice that the above Assumption \textit{(3)} is necessary to ensure that the transport operator $T_{\H}$ has at least the whole imaginary axis in its spectrum (see Theorem \ref{theo:spectTH} for a precise statement).\medskip

With respect to our previous contribution \cite{LMR}, as already mentioned, we do not consider abstract and general boundary operators here but focus our attention on the specific case of a diffuse boundary operator satisfying the following assumption where we define, for $\lambda \in \C$, $\mathrm{Re}\lambda \geq0$ the following bounded operator
\begin{equation*}
\begin{cases}
\mathsf{M}_{\l} \::\:&L^1_- \longrightarrow L^1_+\\
&u \longmapsto
\mathsf{M}_{\l}u(x,v)=e^{-\l\,\tau_{-}(x,v)}u(x-\tau_{-}(x,v)v,v),\:\:\:(x,v) \in \Gamma_+\;;
\end{cases}
\end{equation*}
where $\tau_{-}(x,v):=\inf\{\,s > 0\,;\,x + sv \notin \Omega\}$ for any $(x,v) \in \Gamma_{+}$ (see Section \ref{sec:trav} for more details on the travel time).
\begin{hyp}\label{hypH}
The boundary operator $\H\::\lp \to \lm$ is a bounded and stochastic operator of the form \eqref{eq:Hhkernel} which satisfies the following
\begin{enumerate}[1)]
\item There exists some $n \in \N$  (where $\N$ is the set of nonnegative integers $\N=\{0,1,\ldots\}$) such that
$$\H \in \mathscr{B}(\lp,\Y_{n+1}^{-})$$
where, for any $s \geq 0$, we define
$$\Y_{s}^{\pm}:=\{g \in L^{1}_{\pm}\;;\;\int_{\Gamma_{\pm}}\max(1,|v|^{-s})|g(x,v)|\d\mu_{+}(x,v) < \infty\}.$$
We will set
\begin{equation}\label{eq:HYn}
N_{\H}:=\sup\{k \in \N\;;\;\H \in \mathscr{B}(\lp,\Y_{k+1}^{-})\}.\end{equation}
and assume  that $N_{\H} < \infty$ in all the paper. 
\item The operator $\H\mathsf{M}_{0}\H \in \mathscr{B}(\lp,\lm)$
is weakly compact. 
\item $\mathsf{M}_{0}\H$ is irreducible.
\item There exist $\mathsf{p} \in \N$  and   $C >0$ such that
\begin{equation}\label{eq:power}
\int_{\R}\left\|\left(\mathsf{M}_{\varepsilon+i\eta}\H\right)^{\mathsf{p}}\right\|_{\mathscr{B}(\lp)}\d \eta \leq C \qquad \forall \e \geq0.\end{equation}
\end{enumerate}
\end{hyp}
We will give later in Section \ref{sec:REGU} practical criteria ensuring this set of assumptions to be met resorting notably to our previous contribution \cite{LMR} (see also Section \ref{sec:exain} in this Introduction for some earlier considerations about Assumptions \ref{hypH}). \medskip

By means of a new and robust tauberian approach, we can answer  Question \ref{ques}. The main contribution of this work is summarized in the following 

\begin{theo}\label{theo:maindec}
Under Assumptions \ref{hypO}--\ref{hypH}, for any $f \in \X_{N_{\H}+1}$ there exist a constant $C_{f} > 0$ and 
$$\mathsf{\Theta}_{f} \in \mathscr{C}_{0}(\R,\X_{0}) \cap L^{1}(\R,\X_{0})$$
such that
\begin{equation}\label{eq:rat}
\left\|U_{\H}(t)f-\varrho_{f}\Psi_{H}\right\|_{\X_{0}} \leq \frac{C_{f}}{(1+t)^{N_{\H}}} \bm{\epsilon}(t) \qquad \forall t\geq 0
\end{equation}
where
$$\bm{\epsilon}(t)=\frac{1}{1+t}+\left\|\int_{-\infty}^{\infty}\exp\left(i\eta\,t\right)\mathsf{\Theta}_{f}(\eta)\d \eta\right\|_{\X_{0}} \qquad \forall t\geq0$$
is such that $\lim_{t\to \infty}\bm{\epsilon}(t)=0.$ If we assume moreover that there is $C(\mathsf{p}) > 0$ and $\beta >0$ such that
\begin{equation}\label{eq:decay-power}
\int_{|\eta| >R}\left\|\left(\mathsf{M}_{i\eta}\H\right)^{\mathsf{p}}\right\|_{\mathscr{B}(\lp)} \d \eta \leq \frac{C(\mathsf{p})}{R^{\beta}}, \qquad \forall R >0\end{equation}
then,  there is some positive constant $K >0$ depending only on $\beta$ and $C(\mathsf{p})$ such that 
\begin{equation}\label{eq:estMod}
\left\|\int_{-\infty}^{\infty}\exp\left(i\eta\,t\right)\mathsf{\Theta}_{f}(\eta)\d \eta\right\|_{\X_{0}}\leq K\,\left(\omega_{f}\left(\frac{\pi}{t}\right)\right)^{\frac{\beta}{\beta+1}} \qquad \forall t \geq 1\end{equation}
where $\omega_{f}\::\:\R^{+} \to \R^{+}$ denotes the \emph{minimal modulus of continuity} of the uniformly continuous mapping $\mathsf{\Theta}_{f}$.\end{theo}

Our main Theorem provides therefore an \emph{explicit rate of convergence} of the type
\begin{equation}\label{eq:OT}
\left\|U_{\H}(t)f-\varrho_{f}\Psi_{H}\right\|_{\X_{0}} =\mathbf{O}(t^{-N_{\H}}).\end{equation}
It is important to point out that this rate of convergence is therefore prescribed by the maximal gain of integrability $\H$ is able to provide (corresponding to the parameter $N_{\H}$) and therefore is governed by the action of the boundary operator $\H$ on small velocities. This important feature of collisionless transport equation is fully exploited in the companion paper \cite{LM-iso} where, in the case of velocity bounded away from zero, the rate of convergence turns out to be exponential.

The above explicit rate \eqref{eq:OT} can actually be strengthened into the \emph{semi-explicit}
$$\left\|U_{\H}(t)f-\varrho_{f}\Psi_{H}\right\|_{\X_{0}} =\mathbf{o}(t^{-N_{\H}}).$$
We wish to insist here on the fact that the rate of convergence given in \eqref{eq:rat} is \emph{optimal} in the sense that the error function $\bm{\epsilon}(t)$ is the \emph{exact} correction to the rate $\left(1+t\right)^{-N_{\H}}$. Of course, this correction is only \emph{semi-explicit} and to derive a more explicit rate, one needs to precisely determinate the rate of convergence to zero (granted by Riemann-Lebesgue Theorem) of the Fourier transform
$$t \geq 0 \longmapsto \int_{-\infty}^{\infty}\exp\left(i\eta\,t\right)\mathsf{\Theta}_{f}(\eta)\d \eta \in \X_{0}$$
where $\mathsf{\Theta}_{f}(\cdot)$ is defined as the $N_{\H}$-derivative of some suitable boundary function (see \eqref{eq:defTheta} for a precise definition).
The second part of the Theorem, stated as \eqref{eq:estMod} is a first step towards this direction. The additional assumption \eqref{eq:decay-power} is easy to check in practice (see Section \ref{sec:REGU} where actually Assumption \ref{hypH} \textit{4)} is deduced from \eqref{eq:decay-power}). We point out also the following:

\begin{enumerate}[a)]
\item We assumed for simplicity that $N_{\H}$ is finite but, of course, if $N_{\H}=\infty$, then the above result remains valid and the rate of convergence we obtain then is of the type
$$\|U_{\H}(t)f-\varrho_{f}\Psi_{\H}\|_{\X_{0}}=\mathbf{O}\left(\left(1+t\right)^{-k}\right) \qquad \forall f \in \X_{k+1}\,,$$
for any $k \geq 0.$ It is an interesting open problem to determine whether this convergence can be upgraded to some (stretched) exponential convergence if $f \in \bigcap_{k}\X_{k}$ (for instance, if the support of $f$ is away from $0$). 
\item If the decay at infinity of the mapping $\l \mapsto \left(\mathsf{M}_{\l}\H\right)^{\mathsf{p}} \in \mathscr{B}(\lp)$ is such that \eqref{eq:decay-power} is valid for \emph{any} $\beta >0$ large enough (up to change $\mathsf{p}$), then the overall decay of $\bm{\epsilon}(t)$ is as close as desired from the one of 
$$t \mapsto \omega_{f}\left(\frac{\pi}{t}\right).$$ 
This is the case for instance when $\H$ is associated to the Maxwellian boundary condition of Example \ref{exe:maxw} hereafter.
\item Providing explicit estimates of the modulus of continuity $\omega_{f}$ is an open problem. Some reasonable conjecture about this is given at the end of the paper (see Section \ref{sec:conj}).  
\end{enumerate}

The existence and uniqueness of the equilibrium density $\Psi_{\H}$ as well as some qualitative convergence result has been obtained in a systematic way in \cite{LMR}. Namely, under Assumption \ref{hypH}, one can deduce directly the following from  \cite[Theorem 6.5 and Section 7]{LMR} and \cite{LMK-arxiv}:
\begin{theo}\label{theo:LMR} Under Assumption \ref{hypH}, the operator $(\T_{\H},\D(\T_{\H}))$ defined by
\begin{equation*}\begin{split}
\D(\T_{\H})&=\left\{f \in \X_{0}\;;\;\,v
\cdot \nabla_x \psi \in \X_{0}\;;\;f_{\vert\Gamma_{\pm}} \in L^{1}_{\pm}\,\;\;\,\H\,f_{\vert \Gamma_{+}}=f_{\vert\Gamma_{-}}\right\},\\ 
\T_{\H}f&=-v\cdot \nabla_{x}f, \qquad f \in \D(\T_{\H})\end{split}\end{equation*}
 is the generator of a \emph{stochastic} $C_{0}$-semigroup  $(U_{\mathsf{H}}(t))_{t\geq 0}$. Moreover, $(U_{\mathsf{H}}(t))_{t\geq 0}$ is irreducible and has a unique invariant density ${\Psi}_{\mathsf{H}} \in \D(\mathsf{T}_{\mathsf{H}})$ with 
$${\Psi}_{\mathsf{H}}(x,v) >0 \qquad \text{ for a. e. } (x,v) \in \Omega \times \R^{d}, \qquad \|{\Psi}_{\mathsf{H}}\|_{\X_{0}}=1$$
and $\mathrm{Ker}(\mathsf{T}_{\mathsf{H}})=\mathrm{Span}({\Psi}_{\mathsf{H}}).$ Moreover, 
\begin{equation}\label{eq:ergodic}
\lim_{t \to \infty}\left\|U_{\mathsf{H}}(t)f-\mathbb{P}f\right\|_{\X_{0}}=0, \qquad  \forall f \in \X_{0}\end{equation}
 where $\mathbb{P}$ denotes the ergodic projection
$$\mathbb{P}f=\varrho_{f}\,{\Psi}_{\mathsf{H}}, \qquad \text{ with } \quad\varrho_{f}=\displaystyle\int_{\Omega\times\R^{d}}f(x,v)\d x \otimes \bm{m}(\d v), \qquad f \in \X_{0}.$$ 
\end{theo}

Besides partial results in \cite[Section 7]{LMR}, the strong convergence \eqref{eq:ergodic} has been obtained  recently in our previous (unpublished) contribution \cite{LMK-arxiv}. In that paper, we first proved a general qualitative (without rate) convergence to equilibrium for \eqref{1} under Assumptions \ref{hypO}--\ref{hypH} via a Tauberian argument using Ingham's theorem. Moreover, we also addressed Question \ref{ques} and notably derived suboptimal rates of convergence to
equilibrium for solutions to \eqref{1} under mild assumptions on the initial datum $f$ thanks to recent quantified versions of Ingham's theorem \cite{chill}. Typically, with respect to Theorem \ref{theo:maindec}, the rate obtained in \cite{LMK-arxiv} were of the form
$$\left\|U_\H(t)f-\varrho_f\Psi_\H\right\|_{\X_0}=\mathbf{O}\left(t^{-\frac{N_{\H}}{2}}\right)$$
for $f \in \X_{N_{\H}+1}.$ 
The present paper is a significant improvement of the results of \cite{LMK-arxiv} which do not longer use quantified versions of Ingham's theorem and strenghten in an almost optimal way the rate of convergence. We anticipate already that the tool which allows us to get rid of Ingham's theorem is the use of a suitable representation of the solutions to \eqref{1} combined with a tauberian approach. As far as we know, our construction is new and appears here for the first time.

\subsection{Related literature} Besides its fundamental role in the study of the Boltzmann equation with boundary conditions \cite{EGKM,guo03,briant}, the mathematical interest towards relaxation to equilibrium for collisionless equations is
relatively recent in kinetic theory starting maybe with numerical evidences obtained in \cite{aoki1}. A precise description of the relevance of the question as well as very interesting results have been obtained then in \cite{aoki}.  We mention also the important contributions \cite{kuo,liu} which obtain seemingly optimal rate of convergence when the spatial domain is a ball. For transport equations in a slab geometry, a Tauberian approach based upon Ingham's theorem has been introduced in \cite{mmkseifert}. Such an approach was then generalized to more general geometry (in higher dimension) and improved in our aforementioned unpublished manuscript \cite{LMK-arxiv}. The two very recent works \cite{bernou1,bernou2} provide (nearly optimal) convergence rate for general domains $\Omega$ in a $L^{1}$-setting. All these works are dealing with partially diffuse boundary operator of Maxwell-type for which
\begin{multline}\label{eq:maxweBC}
\H\varphi(x,v)=\alpha (x)\varphi (x,v-(v\cdot n(x))n(x)) \\
+\frac{(1-\alpha
(x))}{\gamma(x)}\M_{\theta(x)}(v)\int_{v'\cdot n(x) >0}\varphi (x,v^{\prime })|v'\cdot n(x)|
\bm{m}(\d v^{\prime })\end{multline}
where, as above $\M_{\theta(x)}$ is a Maxwellian distribution given (see example ) for which the temperature $\theta(x)$ depends (continuously) on $x \in \partial\Omega$ and $\gamma(x)$ is a normalization factor ensuring $\H$ to be stochastic. A nearly optimal rate of convergence for the boundary condition \eqref{eq:maxweBC} in dimension $d=2,3$ has been obtained recently in \cite{bernou1} thanks to a clever use of Harris's subgeometrical convergence theorem for Markov processes. A related probabilistic approach, based on coupling, has been addressed in \cite{bernou2} in dimension $d\geq2$ whenever $\theta(x)=\theta$ is constant and in both these works, the rate of convergence is nearly optimal and given by
$$\mathbf{O}\left(\frac{\left(\log(1+t)\right)^{d+1}}{t^{d}}\right) \quad \text{ as } t \to\infty.$$
For such a model, with the notations of our Theorem \ref{theo:maindec}, 
$$N_{\H}=d-1$$
which suggests that our optimal correction $\bm{\epsilon}(t)$ is at least of the order $\bm{\epsilon}(t)= \frac{\left(\log(1+t)\right)^{d+1}}{t} $. Let us finally mention the  very recent contribution \cite{kim} which closely follows the approach of \cite{bernou1} and provides a $L^{1}-L^{\infty}$ framework for solutions to \eqref{1} with exponential moments and obtain a rate of convergence similar to that of \cite{bernou1} in the case of \emph{diffuse boundary operator} of Maxwell-type as considered here.

Let us point out here that, even if the rate obtained in \cite{bernou1,bernou2,kim} are slightly better than the one obtained here, our contribution is not really comparable to those 
\begin{enumerate}
\item First, we deal here with different kind of boundary conditions (in any dimension $d \geq 1$) and, even if we restrict ourselves to diffuse boundary condition, the structure of the kernel $\bm{k}(x,v,v')$ is much more general than the Maxwellian case \eqref{eq:maxweBC}. Notice in particular that our assumptions on the boundary operator $\H$ are relatively easy to check and quite general. On this aspect, our result can be seen as a \emph{systematic treatment of Question \ref{ques} for general diffusive boundary conditions}. Let us mention here that, even though the approach of \cite{bernou1,bernou2} is robust enough to be applied to more general boundary conditions than the Maxwellian one, such an approach requires the construction of some pointwise lower bounds for the solution to \eqref{1} which would differ from  one boundary operator to another and, as such, some specific work has to be done for each given boundary condition.
\item Second,  the mathematical tools used in the present paper  are completely new and different from those of the associated literature. In particular, since the analysis of \cite{bernou1,bernou2} is based upon a clever modification of Harris's convergence theorem, it is specifically tailored to deal with the $L^1$ (or measure) functional framework \footnote{or more generally, to abstract state spaces where the norm is additive on the positive cone}. Our approach on the contrary, though it uses in some places some specific features of the $L^1$ setting, is definitely robust enough to be applied to a more general functional framework. In particular, it can be suitably modified to be applied  (under suitable \emph{ad hoc} assumptions on the boundary operator $\H$) to some (weighted) $L^2$ or $L^{\infty}$ settings which are particularly relevant for the study of the linearized Boltzmann equation with Maxwell-like boundary condition. 
\item Our theory fully exploits the structure of equation (1.1), but we point out that our construction can be adapted to the study of collisional  equations (see \cite{LM-torus} for neutron transport equation on the torus or \cite{LM-L1} for the spatially homogeneous linear Boltzmann equation with soft potentials)  and we believe that it is virtually adaptable to other perturbative contexts using  the representation of solution to linear collisional kinetic equation as a Dyson-Phillips series.
\end{enumerate}

We insist here again on the fact that the rate derived in \eqref{eq:rat} is \emph{exact}. Even if we derived, up to now, only suboptimal explicit convergence rate, the only restriction to get an optimal rate lies in the difficulty we encountered in estimating accurately the decay to zero of the mapping $t \geq0 \mapsto \int_{-\infty}^{\infty}\exp\left(i\eta\,t\right)\mathsf{\Theta}_{f}(\eta)\d\eta \in\X_{0}$.

\subsection{Practical examples}\label{sec:exain}

A few remarks are in order about our set of Assumptions:
\begin{itemize}
\item First, we gave in our previous contribution \cite[Theorem 5.1]{LMR} a precise definition of a general class of boundary operator for which $\H\mathsf{M}_{0}\H$ is weakly-compact. This class of operators was defined in \cite{LMR} as the class of \emph{regular diffuse boundary} operators and we will simply say here that $\H$ is diffuse. \item Moreover, a practical criterion ensuring the above property \textit{3)} to occur is also given in \cite{LMR}. In practice, as observed already, the typical operator we have in mind is given by \eqref{eq:Hhkernel}.
Under some  \emph{strong positivity} assumption on $\bm{k}(\cdot,\cdot,\cdot)$, one can show that $\mathsf{M}_{0}\H$ is irreducible (see \cite[Section 4]{LMR}).
\item We believe that Assumption \textit{4)} is met for \emph{any} regular diffuse boundary operators. We have been able to prove the result with $\ell = 2$ for a slightly more restrictive class of boundary operators (see Proposition \ref{lem:norm2}) whenever $\bm{m}(\d v)$ is absolutely continuous with respect to the Lebesgue measure on $\R^{d}$.
\end{itemize}

 We refer to Section \ref{sec:REGU}  for more details on this set of assumptions and only provide here a brief list of examples covered by those assumptions and which are particularly relevant as models of boundary interactions in the kinetic theory of gas (see \cite{CIP}).
\begin{exe}[\textit{\textbf{Generalized Maxwell-type}}]\label{exe:gener} The most typical example corresponds to a generalized Maxwell-type diffuse operator for which
$$\bm{k}(x,v,v')=\gamma^{-1}(x)\bm{G}(x,v)$$
where $\G\::\:\partial \Omega \times V \to \R^{+}$ is a measurable and nonnegative mapping such that 
\begin{enumerate}[($i$)]
\item $\G(x,\cdot)$ is radially symmetric for $\pi$-almost every $x \in \partial \Omega$; 
\item $\G(\cdot,v) \in L^{\infty}(\partial \Omega)$ for almost every $v \in V$;
\item The mapping $x \in \partial \Omega \mapsto \gamma(x)$ is continuous and \emph{bounded away from zero} where
\begin{equation}\label{eq:gamm}
\gamma(x):=\int_{\Gamma_{-}(x)}\G(x,v)|v \cdot n(x)| \bm{m}(\d v) \qquad \forall x \in \partial\Omega,\end{equation}
i.e. there exist $\gamma_{0} >0$ such that $\gamma(x) \geq \gamma_{0}$ for $\pi$-almost every $x \in \partial\Omega.$
\end{enumerate}
\end{exe}
\begin{exe}[\textit{\textbf{Maxwell diffuse boundary condition}}]\label{exe:maxw} A particularly relevant example is a special case of the previous one for which, $\bm{m}(\d v)=\d v$ and $\G$ is a given Maxwellian with temperature $\theta(x)$, i.e.
$$\G(x,v)=\mathcal{M}_{\theta(x)}(v), \qquad \mathcal{M}_{\theta}(v)=(2\pi\theta)^{-d/2}\exp\left(-\frac{|v|^{2}}{2\theta}\right), \qquad x \in \partial \Omega, \:\:v \in \R^{d}.$$ 
Then, 
$$\gamma(x)=\bm{\kappa}_{d}\sqrt{\theta(x)}\int_{\R^{d}}|w|\M_{1}(w)\d w, \qquad x \in \partial\Omega$$ 
for some positive constant $\bm{\kappa}_{d}$ depending only on the dimension. The above assumption \textit{$(iii)$} asserts that the temperature mapping $x \in \partial\Omega \mapsto \theta(x)$ is bounded away from zero and continuous.
\end{exe}

 \subsection{Method of proof} \label{sec:method} For the sake of clarity and in
order to help the reading of the paper, we give here an idea of the main
steps of the proof of Theorem \ref{theo:maindec}. The main ideas behind the proof can be summarized in the next three steps:
 \begin{enumerate}[Step 1)]
 \item We exploit an explicit representation of the semigroup $\left(U_{\H}(t)\right)_{t\geq0}$  obtained recently in \cite{luisa}. With this representation, similar to the Dyson-Phillips series for additive perturbative semigroup  theory, the semigroup is expressed as a suitable strongly convergence series
 \begin{equation}\label{eq:series}
 U_{\H}(t)f=\sum_{k=0}^{\infty}\bm{U}_{k}(t)f, \qquad f \in \X_{0}, \quad t \geq 0\end{equation} where the family of operators $\{\bm{U}_{k}(t)\}_{k\geq0}$ is defined inductively. Typically, $\left(\bm{U}_{0}(t)\right)_{t\geq0}$ denotes the semigroup generated by $\T_{0}$ (corresponding to absorbing boundary condition $\H \equiv 0$) whereas, for any $k \in \N$, $\bm{U}_{k}(t)f$ denotes the solution to \eqref{1} after having experienced $k$ rebounds with the boundary. 
 \item Precise estimates of the decay of each terms $\bm{U}_{k}(t)f$ (for given $k \in \N$) are obtained by suitably investigating the influence of the boundary operator for small and large velocities. We in particular show that, if $f \in \X_{N_{\H}+1}$ then 
 $$\left\|\sum_{k=0}^{n}\bm{U}_{k}(t)f\right\|_{\X_{0}}=\mathbf{O}(t^{-N_{\H}-1}) \qquad \text{ as } t \to \infty$$
 for any $n \in \N$.
 See Lemma \ref{lem:decayUn} and Proposition \ref{prop:Snt} for more precise statements.
 \item We these two first points, to investigate the decay of $U_{\H}(t)(\mathbb{I-P})$, we only need to understand that of some suitable remainder of the series \eqref{eq:series}, say
 $$\bm{S}_{n}(t)f=\sum_{k=n+1}^{\infty}\bm{U}_{k}(t)f.$$
 Notice that investigating $\bm{S}_{n}(t)\left(\mathbb{I-P}\right)f$ amounts to study carefully $\bm{S}_{n}(t)f$ for some function $f$ with zero mean, i.e. such that
 $$\varrho_{f}:=\int_{\Omega\times V}f(x,v)\d x \otimes \bm{m}(\d v)=0.$$
This is the most technical part of the paper. It requires a careful study of the spectral properties of $\T_{\H}$ and $\mathsf{M}_{0}\H$ and some tools from Fourier-Laplace analysis. More precisely, while the two previous points can be seen as a \emph{semigroup approach} to Question \ref{ques}, this third step is rather a \emph{resolvent approach} since we deduce the properties of the remainder
 $$\sum_{k=n+1}^{\infty}\bm{U}_{k}(t)f$$
 from the careful study of its Laplace transform which is related to the resolvent of $\T_{\H}.$ This is 
 \end{enumerate}
Let us describe with more details this third step. The resolvent of $\T_{\H}$ can be written, for $\l \in \C_{+}$ as
$$\Rs(\l,\T_{\H})f=\Rs(\l,\T_{0})f+\sum_{k=0}^{\infty}\mathsf{\Xi}_{\l}\H\left(\mathsf{M}_{\l}\H\right)^{k}\mathsf{G}_{\l}f$$
for some suitable operators $\mathsf{\Xi}_{\l},\mathsf{M}_{\l}$ and $\mathsf{G}_{\l}$ described in Section \ref{sec:oper}. Then, for a given $n \in \N$, we can show that the remainder $\bm{S}_{n}(t)$  admits, as a Laplace transform,
$$\int_{0}^{\infty}\exp\left(-\l t\right)\bm{S}_{n}(t)f\d t=\sum_{k=n}^{\infty}\mathsf{\Xi}_{\l}\H\left(\mathsf{M}_{\l}\H\right)^{k}\mathsf{G}_{\l}f=\Upsilon(\l)f, \qquad \mathrm{Re}\l >0.$$
Using properties of the inverse Laplace transform \cite{arendt}, we can describe then entirely $\bm{S}_{n}(t)f$ in terms of $\Upsilon_{n}(\l)f$, namely
$$\bm{S}_{n}(t)f=\frac{\exp(\e t)}{2\pi}\lim_{\ell\to\infty} \int_{-\ell}^{\ell}\exp\left(i\eta t\right)\Upsilon_{n}(\e+i\eta)f\d\eta, \qquad \forall f \in \X_{0}$$
for any $t >0$, $\e >0.$ Of course, to hope deducing a decay of $\bm{S}_{n}(t)f$ for large $t$, the positive exponential is a dramatic obstacle. This enforces to deduce a second representation formula for $\bm{S}_{n}(t)f$ where the inverse Laplace transform is derived \emph{on the imaginary axis}, i.e. for $\l=i\eta,$ $\eta \in\R.$ 

A first mathematical difficulty occurs here since $\Upsilon_{n}(\e+i\eta)f$ is not even defined for $\e=0.$ We need therefore to build, for suitable class of functions $f$, the \emph{boundary trace} of $\Upsilon_{n}(\l)f$ along the imaginary axis. This is the most technical part of the present work which will result in the following
\begin{theo}\label{theo:introtr} Let $f \in \X_{N_{\H}+1}$ be such that
\begin{equation}\label{eq:0mean}
\varrho_{f}=\int_{\Omega\times V}f(x,v)\d x \otimes \bm{m}(\d v)=0.\end{equation}
Then, for any $n\geq0$ the limit
$$\lim_{\e\to0^{+}}\Upsilon_{n}(\e+i\eta)f,$$
exists in $\mathscr{C}_{0}^{N_{\H}}(\R,\X_{0})$. Its limit is denoted $\mathsf{\Psi}_{n}(\eta)f$. Moreover, for $n$ large enough (explicit), the mappings
$$\eta \in \R \longmapsto \dfrac{\d^{k}}{\d\eta^{k}}\mathsf{\Psi}_{n}(\eta)f \in \X_{0}, \qquad 0 \leq k \leq N_{\H}$$
are integrable.
\medskip
\end{theo}
Here above, for any Banach space $(X,\|\cdot\|_{X})$ and any $k \in \N$, we set
\begin{multline*}
\mathscr{C}_{0}^k(\R,X)=\left\{h\::\:\R \mapsto X\;;\;\text{of class $\mathscr{C}^{k}$  over $\R$ }\right.\\
\left.\text{and such that } \lim_{|\eta|\to\infty}\left\|\frac{\d^j}{\d\eta^j}h(\eta)\right\|_{X}=0 \qquad \forall j \leq k\right\}\end{multline*}
and we endow $\mathscr{C}_{0}^k(\R,\X)$ with the norm $\|h\|_{L^{\infty}(X)}:=\sup_{\eta}\max_{0\leq j \leq k}\|\frac{\d^j}{\d\eta^j}h(\eta)\|_{X}$ which makes it a Banach space.
\begin{nb}
We point out here that we cannot expect a \emph{higher regularity} for the mapping $\mathsf{\Psi}_{n}(\eta)f$; the mapping $\Upsilon_{n}(\e+i\eta)f$ is not differentiable $N_{\H}+1$ times because $\mathsf{M}_{\e+i\eta}\H$ is differentiable $k$-times if and only if $\H \in \mathscr{B}(\lp,\Y_{k+1}^{-})$. Therefore, by definition of $N_{\H}$, $\eta \mapsto \mathsf{M}_{\e+i\eta}\H \in \mathscr{B}(\lp,\X_{0})$ is differentiable exactly $N_{\H}$ times and so is $\eta \mapsto  \Upsilon_{n}(\e+i\eta)f \in \X_{0}.$
\end{nb}

With this boundary function, one can prove the second representation formula (see Theorem \ref{theo:represent2})
$$\bm{S}_{n}(t)f=\frac{1}{2\pi}\int_{-\infty}^{\infty}\exp\left(i\eta\,t\right)\mathsf{\Psi}_{n}(\eta)f\,\d\eta\,, \qquad \forall t\geq0$$
for any $f \in \X_{N_{\H}+1}$ satisfying \eqref{eq:0mean}.

Recall that the condition $\varrho_{f}=0$ is equivalent to the condition $\mathbb{P}f=0$, i.e. the above is somehow a representation formula for $\bm{S}_{n}(t)(\mathbb{I-P})f$. With this last representation formula, one easily get convinced that, if one can prove that the mapping
$$\eta \in \R \longmapsto \mathsf{\Psi}_{n}(\eta)f \in \X_{0}$$
belongs to $\mathscr{C}^{N_{\H}}_{0}(\R)$, then by integration by parts, we can expect
$$\bm{S}_{n}(t)f=\left(\frac{i}{t}\right)^{N_{\H}}\int_{-\infty}^{\infty}\exp\left(i\eta\,t\right)\dfrac{\d^{N_{\H}}}{\d \eta^{N_{\H}}}\mathsf{\Psi}_{n}(\eta)f\dfrac{\d\eta}{2\pi}, \qquad t \geq0.$$
This provides at least a decay like 
$\bm{S}_{n}(t)f=\mathbf{O}(t^{-N_{\H}}).$ This decay is actually strengthened into 
$$\bm{S}_{n}(t)f=\mathbf{o}(t^{-N_{\H}})$$ by a simple use of Riemann-Lebesgue Theorem using that the mapping
$$\eta \in \R \longmapsto \dfrac{\d^{N_{\H}}}{\d \eta^{N_{\H}}}\mathsf{\Psi}_{n}(\eta)f$$
is \emph{integrable} over $\R$. One sees here that this is exactly the contents of Theorem \ref{theo:maindec} where clearly 
\begin{equation}\label{eq:defTheta}\mathsf{\Theta}_{f}(\eta)=\dfrac{\d^{N_{\H}}}{\d \eta^{N_{\H}}}\mathsf{\Psi}_{n}(\eta)f \end{equation} for a suitably large choice of the parameter $n$ (related to $\mathsf{p}$ in \eqref{eq:power}). We point out here that we cannot expect a \emph{higher regularity} for the mapping $\mathsf{\Psi}_{n}(\eta)f$.

Let us comment very briefly on the technical points behind the proof of Theorem \ref{theo:introtr}. To study $\Upsilon_{n}(\l)$ along the imaginary axis, we observe first that
$$r_{\sigma }\left(\mathsf{M}_{\lambda}\H\right) <1 \qquad \text{Re}\l >0 \qquad \text{ but $1$  is a (simple) eigenvalue of } \mathsf{M}_{0}\H.$$
It turns out that $\lambda \mapsto \mathsf{M_{\lambda }H} \in \mathscr{B}(\lp)$ extends to the
imaginary axis with
$$r_{\sigma }\left(\mathsf{M_{i\eta }H}\right) <1 \qquad (\eta \neq 0);$$
(see Proposition \ref{prop:Meps}). It follows that 
$$
\lim_{\varepsilon \rightarrow 0_{+}}\Upsilon_{n}(\e+i\eta)f\text{ \ exists}\ \
(\eta \neq 0)
$$
and the convergence is locally uniform in $\eta \neq 0$.
Because $r_{\sigma }\left(\mathsf{M_{0}H}\right) =1$, the treatment of the case $%
\eta =0$  is very involved and the various technical results of Section \ref{sec:near0}  are
devoted to this delicate point. In particular, by exploiting the fact that
near $\lambda =0$, the eigenvalue of $\mathsf{M_{\lambda }H}$ of maximum modulus is
algebraically simple (converging to $1$ as $\lambda \rightarrow 0$, see Proposition \ref{prop:eigenMLH}), and
analyzing the corresponding spectral projection, we can  show the \emph{existence of the boundary function} for all the various operators involved in $\Upsilon_{n}(\e+i\eta)f$ under the assumption \eqref{eq:0mean} (see Lemma \ref{lem:convZk}) where the convergence is uniform with respect to $\eta \in\R$. 

To prove then the smoothness of the boundary function $\mathsf{\Psi}_{n}(\cdot)f$, we need an alternative representation of $\Upsilon_{n}(\l)$ since it appears out of reach to compute the derivatives of $\sum_{k=n}^{\infty}\left(\mathsf{M}_{\l}\H\right)^{k}\mathsf{G}_{\l}f$ at $\l=0$ still because of the fact that $r_{\sigma}(\mathsf{M}_{0}\H)=1.$ It is easy to see that
$$\Upsilon_{n}(\l)=\Rs(\l,T_{\H})-\Rs(\l,\T_{0})-\sum_{k=0}^{n}\mathsf{\Xi}_{\l}\left(\mathsf{M}_{\l}\H\right)^{k}\mathsf{G}_{\l}$$
and the smoothness is then obtained  through the well-known identity for derivatives of the resolvent
$$\dfrac{\d^{k}}{\d\l^{k}}\Rs(\l,\T_{\H})=(-1)^{k}\,k!\,\Rs(\l,\T_{H})^{k+1}, \qquad \mathrm{Re}\l >0.$$ 
Finally, the last technical point is to show that the $N_\H $-th derivative of the boundary
function is integrable. This point is the crucial one where Assumption \ref{hypH} \textit{(4)} is fully exploited. See Theorem \ref{theo:represent2}  for details. All these results yield then to the decay rate \eqref{eq:rat}. Then, the more precise investigation of the correction $\bm{\varepsilon}(t)$ and its link with the modulus of continuity of $\mathsf{\Theta}_f$ is deduced by a simple adaptation of considerations linking the fractional regularity of functions  with the decay of its Fourier coefficients \cite{grafakos}. This gives then \eqref{eq:estMod} under the additional assumption \eqref{eq:decay-power}.
\medskip

As mentioned earlier, proving that Assumption \ref{hypH} \textit{4)} is met for a large class of diffuse boundary operators is a highly technical task and we devote Section \ref{sec:REGU}  to this.  The results of Section \ref{sec:REGU} are also related to a general change of variable formula transferring integrals in velocities into integrals over $\partial\Omega$ which has its own interest. It clarifies several computations scattered in the literature \cite{EGKM,guo03} and will be a fundamental tool for the analysis in the companion paper \cite{LM-iso}.

\subsection{Organization of the paper} In Section \ref{sec:func}, we introduce the functional setting and notations used in the rest of the paper and recall several known results mainly from our previous contribution \cite{LMR}. The representation of the semigroup $\left(U_{\H}(t)\right)_{t\geq0}$ is discussed in Section \ref{sec:DP} together with the decay rate of the iterates $\bm{U}_k(t)$, settling the above Steps 1) and 2) of our approach. Section \ref{sec:fineR} is devoted to the fine analysis of the resolvent $\Rs(1,\mathsf{M}_{\l}\H)$ which is well-defined for $\mathrm{Re}\l >0$ but needs to be carefully extended to the imaginary axis $\l=i\eta$, $\eta \in\R.$ This section is the most technical one of the paper and we have to deal separately with the case $\eta \neq 0$ and $\eta=0.$S uch an extension is a cornerstone in the construction of the boundary function $\lim_{\e \to 0^{+}}\Upsilon_n(\e+i\eta,\T_{\H})f$ (for suitable $f$) which is performed in Section \ref{sec:near0}. In this Section, we also establish the suitable representation formula for $\bm{S}_{n}(t)$ thanks to the inverse Laplace transform (Section \ref{sec:repr}) and we provide the full proof of Theorem \ref{theo:maindec}. In Section \ref{sec:REGU}, we provide practical criteria ensuring Assumptions \ref{hypH} to be met. It contains several results we believe to be of independent interest. We also illustrate how the examples described here above in Section \ref{sec:exain} are covered by our results.

\subsection*{Acknowledgements} BL gratefully acknowledges the financial
support from the Italian Ministry of Education, University and
Research (MIUR), ``Dipartimenti di Eccellenza'' grant 2018-2022. Part of this research was performed while the second author was visiting the ``Laboratoire de Math\'ematiques CNRS UMR 6623'' at Universit\'e de  Franche-Comt\'e in February 2020. He wishes to express his gratitude for the financial support and warm hospitality offered by this Institution.

  \section{Reminders of known results}\label{sec:func}

\subsection{Functional setting}\label{sec:functional}
We introduce the partial Sobolev space 
$$W_1=\{\psi \in \X_{0}\,;\,v
\cdot \nabla_x \psi \in \X_{0}\}.$$ It is known \cite{ces1,ces2}
that any $\psi \in W_1$ admits traces $\psi_{|\Gamma_{\pm}}$ on
$\Gamma_{\pm}$ such that 
$$\psi_{|\Gamma_{\pm}} \in
L^1_{\mathrm{loc}}(\Gamma_{\pm}\,;\,\d \mu_{\pm}(x,v))$$ where
$$\d \mu_{\pm}(x,v)=|v \cdot n(x)|\pi(\d x) \otimes \bm{m}(\d v),$$
denotes the "natural" measure on $\Gamma_{\pm}.$ Notice that, since $\d\mu_{+}$ and $\d\mu_{-}$ share the same expression, we will often simply denote it by 
$$\d \mu(x,v)=|v \cdot n(x)|\pi(\d x) \otimes \bm{m}(\d v),$$
the fact that it acts on $\Gamma_{-}$ or $\Gamma_{+}$ being clear from the context.  Note that
$$\partial
\Omega \times V:=\Gamma_- \cup \Gamma_+ \cup \Gamma_0,$$ where
$$\Gamma_0:=\{(x,v) \in \partial \Omega \times V\,;\,v \cdot
n(x)=0\}.$$
We introduce the set
$$W=\left\{\psi \in W_1\,;\,\psi_{|\Gamma_{\pm}} \in L^1_{{\pm}}\right\}.$$
One can show \cite{ces1,ces2} that 
$$W=\left\{\psi \in
W_1\,;\,\psi_{|\Gamma_+} \in \lp\right\} =\left\{\psi \in
W_1\,;\,\psi_{|\Gamma_-} \in \lm\right\}.$$ Then, the \textit{trace
operators} $\mathsf{B}^{\pm}$:
\begin{equation*}\begin{cases}
\mathsf{B}^{\pm}: \:W_1 \subset &\X_{0} \to L^1_{\mathrm{loc}}(\Gamma_{\pm}\,;\,\d \mu_{\pm})\\
&\psi \longmapsto \mathsf{B}^{\pm}\psi=\psi_{|\Gamma_{\pm}},
\end{cases}\end{equation*}
are such that $\mathsf{B}^{\pm}(W)\subseteq L^1_{\pm}$. Let us define the
{\it maximal transport operator } $\mathsf{T}_{\mathrm{max}}$  as follows:
\begin{equation*}\begin{cases} \mathsf{T}_{\mathrm{max}} :\:\D(\mathsf{T}_{\mathrm{max}}) \subset &\X_{0} \to \X_{0}\\
&\psi \longmapsto \mathsf{T}_{\mathrm{max}}\psi(x,v)=-v \cdot \nabla_x
\psi(x,v),
\end{cases}\end{equation*}
with domain $\D(\mathsf{T}_{\mathrm{max}})=W_1.$
Now, for any \textit{
bounded boundary operator} $\mathsf{H} \in\mathscr{B}(L^1_+,L^1_-)$, define
$\mathsf{T}_{\mathsf{H}}$ as
$$\mathsf{T}_{\mathsf{H}}\varphi=\mathsf{T}_{\mathrm{max}}\varphi \qquad \text{ for any }
\varphi \in \D(\mathsf{T}_{\mathsf{H}}),$$ where 
$$\D(\mathsf{T}_{\mathsf{H}})=\{\psi \in
W\,;\,\psi_{|\Gamma_-}=\mathsf{H}(\psi_{|\Gamma_+})\}.$$ In particular, the
transport operator with absorbing conditions (i.e. corresponding
to $\mathsf{H}=0$) will be denoted by $\T_0$. 

\begin{defi} For any $s \geq 0$, we define the function spaces
$$\Y^{\pm}_{s}=L^{1}(\Gamma_{\pm}\,,\max(1,|v|^{-s}) \d\mu_{\pm})$$
with the norm
$$\|u\|_{\Y^{\pm}_{s}}=\int_{\Gamma_{\pm}}|u(x,v)|\,\max(1,|v|^{-s}) \d\mu_{\pm}(x,v).$$
In the same way, for any $s \geq 0$, we introduce 
$$\X_{s}=L^{1}(\Omega \times V\,,\max(1,|v|^{-s})\d x\otimes \bm{m}(\d v))$$
with norm $\|f\|_{\X_{s}}:=\|\,\max(1,|v|^{-s})f\|_{\X_{0}},$ $f \in \X_{s}.$
\end{defi} 
\begin{nb}\label{nb:varpi} Of course, for any $s\geq0$, $\Y^{\pm}_{s}$ is continuously and densely embedded in $L^{1}_{\pm}.$ In the same way, $\X_{s}$ is continuously and densely embedded in $\X_{0}$. Introduce, for any $s \in \N$, the function
$$\varpi_{s}(v)=\max(1,|v|^{-s}), \qquad v \in V.$$
One will identify, without ambiguity, $\varpi_{s}$ with the multiplication operator acting on $L^{1}_{\pm}$ or on $\X_{0}$, e.g.
\begin{equation*}
\begin{cases}
\varpi_{s} \::\:&\X_{0} \longrightarrow \X_{0}\\
&f \longmapsto 
\varpi_{s}f(x,v)=\varpi_{s}(v)f,\:\:\:(x,v) \in \Omega \times V.
\end{cases}
\end{equation*}
Then, one sees that 
$$\Y^{\pm}_{s}=\{f \in L^{1}_{\pm}\;;\;\varpi_{s}f \in L^{1}_{\pm}\}, \qquad \X_{s}=\{f \in \X_{0}\;;\;\varpi_{s}f \in \X_{0}\}.$$
\end{nb}

\subsection{Travel time and integration formula} \label{sec:trav}
Let us now introduce the \textit{travel time} of particles in $\Omega$ (with the notations of \cite{mjm1}),
defined as:
\begin{defi}\label{tempsdevol}
For any $(x,v) \in \overline{\Omega} \times V,$ define
\begin{equation*}
t_{\pm}(x,v)=\inf\{\,s > 0\,;\,x\pm sv \notin \Omega\}.
\end{equation*}
To avoid confusion, we will set $\tau_{\pm}(x,v):=t_{\pm}(x,v)$  if $(x,v) \in
\partial \Omega \times V.$
\end{defi}
 
With the notations of \cite{guo03}, $t_{-}$ is the \emph{backward exit time} $t_{\mathbf{b}}$.
From a heuristic perspective, $t_{-}(x,v)$ is the time needed by a
particle having the position $x \in \Omega$ and the velocity $-v
\in V$ to reach the boundary $\partial\Omega$. One can prove
\cite[Lemma 1.5]{voigt} that $t_{\pm}(\cdot,\cdot)$ is measurable on
$\Omega \times V$. Moreover $\tau_{\pm}(x,v)=0 \text{ for any } (x,v)
\in \Gamma_{\pm}$ whereas $\tau_{\mp}(x,v)> 0$ on $\Gamma_{\pm}.$ It holds
$$(x,v) \in \Gamma_{\pm} \Longleftrightarrow \exists y \in \Omega \quad \text{ with } \quad t_{\pm}(y,v) < \infty \quad \text{ and }\quad x=y\pm t_{\pm}(y,v)v.$$
In that case, $\tau_{\mp}(x,v)=t_{+}(y,v)+t_{-}(y,v).$  Notice also that,
\begin{equation}\label{eq:scale}
t_{\pm}(x,v)|v|=t_{\pm}\left(x,\omega\right), \qquad \forall (x,v) \in \overline{\Omega} \times V, \:v \neq 0, \:\omega=|v|^{-1}\,v \in \mathbb{S}^{d-1}.\end{equation}

We have the following integration formulae from \cite{mjm1}.
\begin{propo} For any $h \in \X_{0}$, it holds
\begin{equation}\label{10.47}
\int_{\Omega \times V}h(x,v)\d x \otimes \bm{m}(\d v)
=\int_{\Gamma_\pm}\d\mu_{\pm}(z,v)\int_0^{\tau_{\mp}(z,v)}h\left(z\mp\,sv,v\right)\d s,
\end{equation}
and for any $\psi \in L^1(\Gamma_-,\d\mu_{-})$,
\begin{equation}\label{10.51}
\int_{\Gamma_-}\psi(z,v)\d\mu_{-}(z,v)=\int_{\Gamma_+}\psi(x-\tau_{-}(x,v)v,v)\d\mu_{+}(x,v).\end{equation}
\end{propo}
\begin{nb} Notice that, because $\mu_{-}(\Gamma_{0})=\mu_{+}(\Gamma_{0})=0$, we can extend the above identity \eqref{10.51} as follows: for any $\psi \in L^{1}(\Gamma_{-} \cup \Gamma_{0},\d\mu_{-})$ it holds
\begin{equation}\label{10.52}
\int_{\Gamma_{-}\cup \Gamma_{0}}\psi(z,v)\d\mu_{-}(z,v)=\int_{\Gamma_{+}\cup\Gamma_{0} }\psi(x-\tau_{-}(x,v)v,v)\d\mu_{+}(x,v).\end{equation}
\end{nb}

\subsection{Decay of the semigroup assiociated to  absorbing boundary conditions}\label{sec:absor}

We end this section with a decay property of the semigroup $\left(U_{0}(t)\right)_{t\geq0}$ on the hierarchy of spaces $\X_{k}$ $(k \in \N)$:
\begin{lemme}\label{lem:decayU0} 
Given $k \in \N$ and $f \in \X_{k}$, one has
$$\left\|U_{0}(t)f\right\|_{\X_{0}} \leq \frac{D^{k}}{t^{k}}\|f\|_{\X_{k}}, \qquad \forall t >0.$$
\end{lemme}
\begin{proof} Let $f \in \X_{k}$ and $t >0$ be fixed. For simplicity, we introduce $g(x,v)=|v|^{-k}|f(x,v)|$, $(x,v) \in \Omega\times V$ and denote by $\widetilde{g}$ the extension by zero of $g$ to $\R^{d}\times V.$ One has then
$$\|U_{0}(t)f\|_{\X_{0}}=\int_{\Omega\times V}|v|^{k}g(x-tv,v)\ind_{t < t_{-}(x,v)}\d x \bm{m}(\d v)$$ 
For a given $x \in \Omega$, one has
\begin{equation*} 
\int_{V}|v|^{k}g(x-tv,v)\ind_{t <t_{-}(x,v)}\bm{m}(\d v)\leq \int_{V}|v|^{k}\widetilde{g}(x-tv,v)\bm{m}(\d v).\end{equation*}
Recalling that $t >0,x \in \Omega$ are fixed, we denote by $\bm{m}_{\#,(t,x)}(\d y)$ the image measure of $\bm{m}(\d v)$ through the transform  $v \in V \mapsto y=x-tv \in \R^{d}$ and deduce
\begin{equation*}\begin{split}
\int_{V}|v|^{k}g(x-tv,v)\ind_{t < t_{-}(x,v)}\bm{m}(\d v)&\leq
\int_{\R^{d}}\left(\frac{|x-y|}{t}\right)^{k}\widetilde{g}\left(y,\frac{x-y}{t}\right)\bm{m}_{\#,(t,x)}(\d y)\\
&\leq \frac{D^{k}}{t^{k}}\int_{\R^{d}}\widetilde{g}\left(y,\frac{x-y}{t}\right)\bm{m}_{\#,(t,x)}(\d y).\end{split}\end{equation*}
Therefore,
$$\|U_{0}(t)f\|_{\X_{0}}\leq \frac{D^{k}}{t^{k}}\int_{\R^{d}}\d x \int_{\R^{d}}\widetilde{g}\left(y,\frac{x-y}{t}\right) \bm{m}_{\#,(t,x)}(\d y).$$
By definition of $\bm{m}_{\#,(t,x)}(\d y)$, we can performe back the change of variable $y\mapsto v=\frac{x-y}{t}$ to get
$$\|U_{0}(t)f\|_{\X_{0}} \leq \frac{D^{k}}{t^{k}}\int_{V}\bm{m}(\d v)\int_{\R^{d}} \widetilde{g}(x-tv,v)\d x.$$
Now, given $v \in V$, we perform the change of variable $x \in \R^{d}\mapsto z=x-tv \in \R^{d}$ to get 
$$\|U_{0}(t)f\|_{\X_{0}}\leq \frac{D^{k}}{t^{k}}\int_{\R^{d}\times V}g(z,v)\d z \bm{m}(\d v)=\frac{D^{k}}{t^{k}}\|g\|_{\X_{0}}$$
which gives the result.
\end{proof}
We complement the above with the following technical property
\begin{lemme}\label{lem:UotInt} For any $k \geq 0$ and $f \in \X_{k+1}$, 
$$\int_{0}^{\infty}\left\|U_{0}(t)f\right\|_{\X_{k}}\d t \leq D\|f\|_{\X_{k+1}} \qquad \text{ and } \qquad \int_{0}^{\infty}t^{k}\left\|U_{0}(t)f\right\|_{\X_{0}}\d t \leq \frac{(2D)^{k+1}}{k+1}\|f\|_{\X_{k+1}}.$$
\end{lemme}
\begin{proof} We prove the result for $k=0$. Assume $f \in \X_{1}$. Using \eqref{10.47} one computes for any $t\geq	0$,
\begin{equation*}\begin{split}
\|U_{0}(t)f\|_{\X_{0}}&=\int_{\Gamma_{-}}\d\mu_{-}(z,v)\int_{0}^{\tau_{+}(z,v)}\left|\left[U_{0}(t)f\right](z+sv,v)\right|\d s\\
&=\int_{\Gamma_{-}}\d\mu_{-}(z,v)\int_{0}^{\tau_{+}(z,v)}\ind_{[0,t_{-}(z+sv,v)]}(t)\left|f(z+(s-t)v,v)\right|\d s
\end{split}\end{equation*}
Since $t_{-}(z+sv,v)=s$ for any $(z,v) \in \Gamma_{-}$, we deduce that
\begin{equation}\label{eq:normUo}
|U_{0}(t)f\|_{\X_{0}}=\int_{\Gamma_{-}}\d\mu_{-}(z,v)\int_{0}^{\tau_{+}(z,v)}\ind_{[0,s]}(t)\left|f(z-(t-s)v,v)\right|\d s.\end{equation}
Integrating this identity and using Fubini's Theorem repeatedly yields
\begin{equation*}\begin{split}
\int_{0}^{\infty}\|U_{0}(t)f\|_{\X_{0}}\d t&=\int_{\Gamma_{-}}\d\mu_{-}(z,v)\int_{0}^{\tau_{+}(z,v)}\d s\int_{0}^{s}\left|f(z-(t-s)v,v)\right|\d t\\
&=\int_{\Gamma_{-}}\d\mu_{-}(z,v)\int_{0}^{\tau_{+}(z,v)}\d s\int_{0}^{s}|f(z-\tau v,v)|\d \tau\\
&=\int_{\Gamma_{-}}\d\mu_{-}(z,v)\int_{0}^{\tau_{+}(z,v)}|f(z-\tau v,v)|\d \tau\int_{t}^{\tau_{-}(z,v)}\d s\end{split}\end{equation*}
i.e.
$$\int_{0}^{\infty}\|U_{0}(t)f\|_{\X_{0}}\d t \leq \int_{\Gamma_{-}}\tau_{+}(z,v)\d\mu_{-}(z,v)\int_{0}^{\tau_{+}(z,v)}|f(z-\tau v,v)|\d \tau.$$
Because $\tau_{+}(z,v) \leq D/|v|$, we see using again \eqref{10.47} that
$$\int_{0}^{\infty}\|U_{0}(t)f\|_{\X_{0}}\d t \leq D \int_{\Omega\times V}|v|^{-1}\,|f(x,v)|\d x\bm{m}(\d v) \leq D\|f\|_{\X_{1}}.$$
This proves the result for $k=0$. Given now $k \geq 1$, because the multiplication operator $\varpi_{k}$ and $U_{0}(t)$ commute $(t\geq0)$, one has
$$\|U_{0}(t)f\|_{\X_{k}}=\|U_{0}(t)\left(\varpi_{k}f\right)\|_{\X_{0}}$$
which allows to apply the result obtained so far to $\varpi_{k}f \in \X_{1}$ and to conclude in the general case. Let us now focus on the second estimate. One sees also from \eqref{eq:normUo} that
\begin{equation*}\begin{split}
\int_{0}^{\infty}t^{k}\|U_{0}(f)f\|_{\X_{0}}\d t&=\int_{0}^{\infty}t^{k}\d t\int_{\Gamma_{-}}\d\mu_{-}(z,v)\int_{0}^{\tau_{+}(z,v)}\ind_{[0,s]}(t)\left|f(z-(t-s)v,v)\right|\d s\\
&=\int_{\Gamma_{-}}\d\mu_{-}(z,v)\int_{0}^{\tau_{+}(z,v)}|f(z-\tau v,v)|\d \tau\int_{\tau}^{\tau_{-}(z,v)}\left(s+\tau\right)^{k}\d s\\
&\leq \frac{1}{k+1}\int_{\Gamma_{-}}\d\mu_{-}(z,v)\int_{0}^{\tau_{+}(z,v)} \left(\tau_{+}(z,v)+\tau\right)^{k+1}|f(z-\tau v,v)|\d \tau\\
&\leq \frac{2^{k+1}}{k+1}\int_{\Gamma_{-}}\tau_{+}(z,v)^{k+1}\d\mu_{-}(z,v)\int_{0}^{\tau_{+}}|f(z-\tau v,v)|\d\tau
\end{split}\end{equation*}

and one concludes as previously since $\tau_{+}(z,v) \leq D/|v|.$ 
\end{proof}
\subsection{About the resolvent of $\mathsf{T}_{\mathsf{H}}$}\label{sec:oper}

For any $\lambda \in \mathbb{C}$ such that $\mathrm{Re}\lambda
> 0$, define
\begin{equation*}\begin{cases}
\mathsf{M}_{\lambda} \::\;L^1_- \longrightarrow L^1_+\\
\\
\mathsf{M}_{\lambda}u(x,v)=u(x-\tau_{-}(x,v)v,v)e^{-\lambda\tau_{-}(x,v)},\:\:\:(x,v) \in \Gamma_+\;\;u \in L^{1}_{-}\,;
\end{cases}\end{equation*}
\begin{equation*}\begin{cases}
\mathsf{G}_{\lambda} \::\:\X_{0} \longrightarrow L^1_+\\
\mathsf{G}_{\lambda}\varphi(x,v)=\displaystyle
\int_0^{\tau_{-}(x,v)}\varphi(x-sv,v)e^{-\lambda s}\d s,\:\:\:(x,v) \in
\Gamma_+\;\,\varphi \in \X_{0}\end{cases}\end{equation*}

\begin{equation*}\begin{cases}
\mathsf{R}_{\lambda} \::\:\X_{0} \longrightarrow \X_{0}\\
\mathsf{R}_{\lambda}\varphi(x,v)=\displaystyle
\int_0^{t_{-}(x,v)}\varphi(x-tv,v)e^{-\lambda t}\d t,\:\:\:(x,v) \in
\Omega\times V;\:\:\varphi \in \X_{0}\end{cases}\end{equation*}

and
\begin{equation*}\begin{cases}
\mathsf{\Xi}_{\lambda} \::\:L^1_- \longrightarrow \X_{0}\\
\\
\mathsf{\Xi}_{\lambda}u(x,v)=u(x-t_{-}(x,v)v,v)e^{-\lambda
t_{-}(x,v)}\ind_{\{t_{-}(x,v) < \infty\}},\:\:\:(x,v) \in \Omega \times V\;\;u \in \lm;\end{cases}\end{equation*}
where $\ind_E$ denotes the characteristic function of the measurable
set $E$. The interest of these operators is related to the resolution of the boundary
value problem:
\begin{equation}\label{BVP1}
\begin{cases}
(\lambda- \mathsf{T}_{\mathrm{max}})f=g,\\
\mathsf{B}^-f=u,
\end{cases}
\end{equation}
where $\lambda > 0$, $g \in \X_{0}$ and $u$ is a given function over
$\Gamma_-.$ Such a boundary value problem, with $u \in \lm$ can be uniquely solved (see \cite[Theorem 2.1]{mjm1})
\begin{theo}\label{Theo4.2} Given  $\lambda >0$, $u \in \lm$ and $g \in \X_{0}$, the function
$$f=\mathsf{R}_{\lambda}g + \mathsf{\Xi}_{\lambda}u$$ is the \textbf{
unique} solution $f \in \D(\mathsf{T}_{\mathrm{max}})$ of the boundary value
problem \eqref{BVP1}. Moreover, $\B^{+}f \in \lp$ and
$$\|\B^{+}f\|_{\lp}+\l\,\|f\|_{\X_{0}} \leq \|u\|_{\lm} + \|g\|_{\X_{0}}.$$
\end{theo}

\begin{nb}\label{nb:lift} Notice that
$\mathsf{\Xi}_{\lambda}$ is a lifting operator which, to a given $u \in
\lm$, associates a function $f=\mathsf{\Xi}_{\lambda}u \in
\D(\mathsf{T}_{\mathrm{max}})$ whose trace on $\Gamma_-$ is exactly $u$.
More precisely,
\begin{equation}\label{propxil1}\T_\mathrm{max}\mathsf{\Xi}_{\lambda}u=\lambda \mathsf{\Xi}_\lambda u, \qquad
\mathsf{B}^-\mathsf{\Xi}_\lambda u=u,\:\:\mathsf{B}^+\mathsf{\Xi}_\lambda u = \mathsf{\ml} u, \qquad  \forall u
\in \lm.
\end{equation}
Moreover, for any $\lambda >0$, one sees with the choice $u=0$ that $\mathsf{R}_{\lambda}$ coincides with $\Rs(\lambda,\mathsf{T}_{0})$. The above theorem also shows that, for any $\lambda
> 0$
\begin{equation}\label{eq:XiLRL}
\|\mathsf{\Xi}_{\lambda}\|_{\mathscr{B}(\lm,\,\X_{0})} \leq 
\lambda^{-1}\,\qquad \text{ and } \quad
\|\mathsf{R}_{\lambda}\|_{\mathscr{B}(\X_{0})} \leq \lambda^{-1}.
\end{equation}
Moreover, one has the obvious estimates
$$\|\mathsf{M}_{\lambda}\|_{\mathscr{B}(\lm,\lp)} \leq 1, \qquad
\|\mathsf{G}_{\lambda}\|_{\mathscr{B}(\X_{0},\lp)} \leq 1
$$
for any $\l >0.$
\end{nb}
We can complement the above result with the following whose proof can be extracted from \cite[Proposition 2.6]{LMR}:
\begin{propo}\phantomsection\label{propo:resolvante}
For any $\l \in  {\C}_{+}$ it holds
\begin{equation}\label{eq:lambdaTH}
\Rs(\lambda,\mathsf{T}_{\mathsf{H}})=\mathsf{R}_{\lambda}+\mathsf{\Xi}_{\lambda}\mathsf{H}\Rs(1,\mathsf{M}_{\lambda}\mathsf{H})\mathsf{G}_{\lambda}
=\Rs(\lambda,\T_{0})+\sum_{n=0}^{\infty}\mathsf{\Xi}_{\lambda}\mathsf{H}\left(\mathsf{M}_{\lambda}\mathsf{H}\right)^{n}\mathsf{G}_{\lambda}
\end{equation}
where the series  converges in $\mathscr{B}(\X_{0})$.
\end{propo}

We make the following basic observation which will turn useful in the sequel
\begin{propo}\label{propo:regul} If $f \in \X_{k+1}$, $0\leq k\leq N_{\H}$, then
$$g_{\l}:=\Rs(\l,\T_{\H})f \in \X_{k}, \qquad \forall \l \in \C_{+}.$$
Moreover, if $\varrho_{f}=0$ then $\varrho_{g_{\l}}=0$ for all $\l \in \C_{+}.$
\end{propo}
\begin{proof} Assume that $\varrho_{f}=0$. The equation $\l\,g_{\l}-\T_{\H}g_{\l}=f$ implies, after integration, that
$$\l\,\int_{\Omega\times V}g_{\l}(x,v)\d x \otimes\bm{m}(\d v)=\int_{\Omega \times V}f(x,v)\d x\otimes\bm{m}(\d v)=0.$$
Because $\l \neq 0$, one sees that $\varrho_{g_{\l}}=0.$  
\end{proof}

\subsection{Some auxiliary operators}\label{sec:prelim}

For $\lambda=0$, we can extend the definition of these operators in an obvious way but not all the resulting operators are bounded in their respective spaces. However, we see from  the above integration formula \eqref{10.51}, that 
$$\mathsf{M}_{0} \in \mathscr{B}(\lm,\lp) \qquad \text{ with } \quad \|\mathsf{M}_{0}u\|_{\lp}=\|u\|_{\lm}, \qquad \forall u \in \lm.$$
In the same way, one deduces from   \eqref{10.47} that for any nonnegative $\varphi \in \X_{0}$:
\begin{equation}\begin{split}
\label{Eq:G0}
\int_{\Gamma_{+}}\mathsf{G}_{0}\varphi(x,v)\d\mu_{+}(x,v)&=\int_{\Gamma_{+}}\d\mu_{+}(x,v)\int_{0}^{\tau_{-}(x,v)}\varphi(x-sv,v)\d s\\
&=\int_{\Omega \times V}\varphi(x,v)\d x \otimes \bm{m}(\d v)\end{split}\end{equation}
which proves that 
$$\mathsf{G}_{0} \in \mathscr{B}(\X_{0},\lp) \qquad \text{ with } \quad  \|\mathsf{G}_{0}\varphi\|_{\lp}=\|\varphi\|_{\X_{0}}, \qquad \forall \varphi\in \X_{0}.$$
Notice that, more generally, for any $\eta \in \R$
$$\mathsf{G}_{i\eta} \in \mathscr{B}(\X_{0},\lp), \qquad \mathsf{M}_{i\eta} \in \mathscr{B}(\lm,\lp)$$
with 
$$\|\mathsf{G}_{i\eta}\|_{\mathscr{B}(\X_{0},\lp)} \leq 1,\qquad \|\mathsf{M}_{i\eta}\|_{\mathscr{B}(\lm,\lp)} \leq 1.$$

To be able to provide a rigorous definition of the operators $\mathsf{\Xi}_{0}$ and $\mathsf{R}_{0}$ we need the following
The interest of the above boundary spaces lies in the following (see \cite[Lemma 2.8]{LMR} where \eqref{eq:M0+Y} is proven for $k=1$ but readily extends to $k \in \N$):
\begin{lemme}\phantomsection\label{lem:motau} 
For any $u \in \Y^{-}_{1}$  one has $\mathsf{\Xi}_{0}u \in \X_{0}$ with 
\begin{equation}\label{eq:Xio}
\|\mathsf{\Xi}_{0}u\|_{\X_{0}}=\int_{\Gamma_{-}}u(x,v)\tau_{+}(x,v)\d\mu_{+}(x,v) \leq D\|u\|_{\Y^{-}_{1}}, \qquad \forall u \in \Y^{-}_{1}\end{equation}
where we recall that $D$ is the diameter of $\Omega$.  
Moreover, given $k \geq 1$, if $u \in \Y^{-}_{k}$ then $\mathsf{M}_{0}u \in \Y^{+}_{k}$ and $\mathsf{\Xi}_{0}u \in \X_{k-1}$ with
\begin{equation}\label{eq:M0+Y}
\|\mathsf{M}_{0}u\|_{\Y^{+}_{k}}= \|u\|_{\Y^{-}_{k}} \qquad \text{ and } \quad \|\mathsf{\Xi}_{0}u\|_{\X_{k-1}} \leq D\|u\|_{\Y^{-}_{k}}\end{equation}
If $f \in \X_{1}$ then $\mathsf{G}_{0}f \in \Y^{+}_{1}$  and $\mathsf{R}_{0} f \in \D(\mathsf{T}_{0}) \subset X$ and $\mathsf{T}_{0}\mathsf{R}_{0}f=-f$. \end{lemme} 

\begin{nb}\label{nb:PsiXn} We wish to emphasise here that, if $\H$ satisfies assumptions \ref{hypH} \textit{1)}, then 
$$\Psi_{\H} \in \X_{n} \qquad \forall n \leq N_{\H}$$
Indeed, recall from \cite[Proposition 4.2]{LMR}, that $\Psi_{\H}=\mathsf{\Xi}_{0}\H\,\bar{\varphi}$ where $\bar{\varphi} \in \lp$ is such that
$$\mathsf{M_{0}H}\bar{\varphi}=\bar{\varphi}.$$
From Assumption \ref{hypH} \textit{1)}, $\H\bar{\varphi} \in \Y_{n+1}^{-}$ and from \eqref{eq:M0+Y}, $\bar{\varphi} \in \Y_{n+1}^{-}$ and $\Psi_{\H} \in \X_{n}.$
\end{nb}

An immediate but fundamental consequence of the above Lemma is the following which will be used repeatedly in the sequel:
\begin{cor}\label{cor:boundMX}
It holds
\begin{equation}\label{eq:boundMX}
\mathsf{M}_{0}\H \in \mathscr{B}(\lp,\Y^{+}_{N_{\H}+1}), \qquad \mathsf{\Xi}_{0}\H \in \mathscr{B}(\lp,\X_{N_{\H}})\end{equation}
In particular, for any $f \in \lp$ and any $\e \geq 0$, the mapping
$$\eta \in \R \longmapsto \mathsf{\Xi}_{\e+i\eta}\H f \in \X_{0}$$
is of class $\mathscr{C}^{N_{\H}}$ with 
$$\left\|\dfrac{\d^{k}}{\d\eta^{k}}\mathsf{\Xi}_{\e+i\eta}\H f\right\|_{\X_{0}} \leq D^{k}\left\|\mathsf{\Xi}_{0}\H\right\|_{\mathscr{B}(\lp,\X_{k})} \|f\|_{\lp} \qquad k \in \{0,\ldots,N_{\H}\}.$$
\end{cor}
\begin{proof} The proof is a direct consequence of \eqref{eq:M0+Y} which implies $\mathsf{M}_{0} \in \mathscr{B}(\Y_{k}^{-},\Y^{+}_{k})$ while $\mathsf{\Xi}_{0} \in \mathscr{B}(\Y_{k}^{-},\X_{k-1})$ for any $k \geq	1.$ Since $\H \in \mathscr{B}(\lp,\Y_{N_{\H}+1}^{-})$, we get \eqref{eq:boundMX}. One checks then in a straightforward way that, for any $k \leq N_{\H}$,
$$\dfrac{\d^{k}}{\d\eta^{k}}\mathsf{\Xi}_{\e+i\eta}\H f=\left(-it_{-}\right)^{k}\mathsf{\Xi}_{\e+i\eta}\H f$$
where $\left(-it_{-}\right)^{k}$ is the multiplication operator by the measure mapping $(x,v) \in \Omega\times V \mapsto \left(-it_{-}(x,v)\right)^{k}$. This gives directly the result since $\left|\left(-it_{-}(x,v)\right)^{k}\right| \leq D^{k}\varpi_{k}(v)$ for any $(x,v) \in \Omega\times V$ and 
$$\|\mathsf{\Xi}_{\e+i\eta}\H\|_{\mathscr{B}(\lp,\X_{k})} \leq \|\mathsf{M}_{0}\H\|_{\mathscr{B}(\lp,\X_{k})} \leq \|\mathsf{\Xi}_{0}\H\|_{\mathscr{B}(\lp,\X_{N_{\H}})} \qquad \forall \e \geq0, \:\:\eta \in \R$$ 
as soon as  $k \leq \N_{\H}$.\end{proof}
\begin{nb} As observed in the Introduction, the fact that the maximal gain of integrability for $\H$ is measured by $N_\H$ is what make the above $\mathscr{C}^{N_\H}$ the \emph{maximal regularity} of the mapping $\eta \mapsto \Xi_{\e+i\eta}\H \in \mathscr{B}(\lp,\X_0)$. One may wonder here if some additional assumption like $$\H \in \mathscr{B}(\lp,\Y^-_{N_{\H}+1+\alpha}), \text{ for some } \alpha \in (0,1)$$  would induce some additional fractional derivative that could be exploited.\end{nb}
 One has the following
\begin{lemme}\label{lem:unifconti}
For any $\e \geq 0$, the mapping $\eta \in \R \longmapsto \mathsf{M}_{\e+i\eta}\H \in \mathscr{B}(\lp)$ is uniformly continuous on $\R.$ Consequently,
\begin{equation}\label{eq:powerlim}
\lim_{|\eta|\to\infty}\left\|\left(\mathsf{M}_{\e+i\eta}\H\right)^{\mathsf{p}}\right\|_{\mathscr{B}(\lp)}=0 \qquad \forall \e >0\end{equation}
where $\mathsf{p}$ is defined through \eqref{eq:power}.
\end{lemme}
\begin{proof} Let $\e \geq0$ be fixed. Given $\eta_{1},\eta_{2} \in \R$, one has
$$\mathsf{M}_{\e+i\eta_{1}}\H-\mathsf{M}_{\e+i\eta_{2}}\H=\left(\exp\left(-i\eta_{1}\tau_{-}\right)-\exp\left(-i\eta_{2}\tau_{-}\right)\right)\mathsf{M}_{\e}\H$$
so that, because the mapping $t \mapsto e^{it} \in \C$ is $2$-Lipschitz,
$$\left\|\mathsf{M}_{\e+i\eta_{1}}\H\varphi -\mathsf{M}_{\e+i\eta_{2}}\H\right\|_{\mathscr{B}(\lp)} \leq 2|\eta_{1}-\eta_{2}|\left\|\tau_{-}\mathsf{M}_{\e}\H \right\|_{\mathscr{B}(\lp)} \leq 2|\eta_{1}-\eta_{2}|\left\|\tau_{-}\mathsf{M}_{0}\H\right\|_{\mathscr{B}(\lp)}.$$
Now, because $\H \in \mathscr{B}(\lp,\Y_{1}^{-})$ and $\tau_{-}(x,v) \leq D/|v|$, one sees that 
$$\left\|\tau_{-}\mathsf{M}_{0}\H\right\|_{\mathscr{B}(\lp)} \leq D\left\|\mathsf{M}_{0}\H\right\|_{\mathscr{B}(\lp,\Y_{1}^{+})} \leq D\|\mathsf{M}_{0}\|_{\mathscr{B}(\Y^{-}_{1},\Y^{+}_{1})}\|\H\|_{\mathscr{B}(\lp,\Y^{-}_{1})}.$$
Thus, 
$$\left\|\mathsf{M}_{\e+i\eta_{1}}\H\varphi -\mathsf{M}_{\e+i\eta_{2}}\H\right\|_{\mathscr{B}(\lp)} \leq 2D|\eta_{1}-\eta_{2}|\|\H\|_{\mathscr{B}(\lp,\Y^{-}_{1})}$$
which proves that the mapping $\eta \in \R \mapsto \mathsf{M}_{\e+i\eta}\H \in \mathscr{B}(\lp)$ is uniformly continuous on $\R$. We deduce then \eqref{eq:powerlim} from the uniform continuity and the integrability at infinity ensured by \eqref{eq:power}.
\end{proof}

\section{About Dyson-Phillips representation}\label{sec:DP}

\subsection{First definitions}\label{sec:DP}
{We recall here a useful representation of the  semigroup $U_{\H}(t)$ as a kind of Dyson-Phillips expansion series introduced in \cite{luisa}. We recall the definition of the $C_{0}$-semigroup generated by $\T_{0}:$
$$U_{0}(t)f(x,v)=f(x-tv,v)\ind_{\{t < t_{-}(x,v)\}}, \qquad f \in \X_{0},\quad t \geq 0.$$}
 We begin with the following definition where $\D_{0}=\{f \in \D(\T_{\mathrm{max}})\;;\;\B^{-}f=0=\B^{+}f\}$ and $\bm{U}_{0}(t)=U_{0}(t)$ $(t\geq0)$:
\begin{defi}\label{defi:Uk}
Let $ t \geq 0$, $k \geq 1$ and $f \in \D_{0}$ be given. For $(x,v) \in \overline{\Omega} \times V$ with $t_{-}(x,v) < t$, there exists a unique $y \in \partial\Omega$ with $(y,v) \in \Gamma_{-}$ and a unique $0 < s < \min(t,\tau_{+}(y,v))$ such that $x=y+sv$ 
and then one sets
$$[\bm{U}_{k}(t)f](x,v)=\left[\H\B^{+}\bm{U}_{k-1}(t-s)f\right](y,v),$$
We set $[\bm{U}_{k}(t)f](x,v)=0$ if $t_{-}(x,v) \geq t$ and  $\bm{U}_{k}(0)f=0$.
\end{defi}
\begin{nb} Clearly, for $(x,v) \in \Omega \times V$ with $t_{-}(x,v) < t$, the unique $(y,v) \in \Gamma_{-}$ and $s \in (0,\min(t,\tau_{+}(y,v))$ such that $x=y+sv$ are 
$$y=x-t_{-}(x,v)v, \qquad \quad s=t_{-}(x,v)$$
so that the above definition reads
$$[\bm{U}_{k}(t)f](x,v)=\left[\H(\B^{+}\bm{U}_{k-1}(t-s)f\right](x-sv,v)\bigg\vert_{s=t_{-}(x,v)}.$$
\end{nb}
For a diffuse boundary operator $\H$ the expression of $\bm{U}_{n}(t)$ is fully explicit, namely
\begin{lemme}\label{lem:exprUn} If $\H$ is given by \eqref{eq:Hhkernel}, then, for any $n \in \N$ and any $(x,v) \in \Omega\times V$, it holds
\begin{multline}\label{eq:Untf}
\left[\bm{U}_{n}(t)f\right](x,v)=\\
\int_{\Gamma_{+}(y_{0})}\bm{k}(y_{0},v,v_{0})|v_{0}\cdot n(y_{0})|\bm{m}(\d v_{0})\int_{\Gamma_{+}(y_{1})}\bm{k}(y_{1},v_{0},v_{1})|v_{1}\cdot n(y_{1})|\bm{m}(\d v_{1})\ldots\\
\int_{\Gamma_{+}(y_{n-2})}\bm{k}(y_{n-2},v_{n-3},v_{n-2})\,|v_{n-2}\cdot n(y_{n-2})|\bm{m}(\d v_{n-2})\times\\
\times\int_{A_{t}(x,v,v_{0},\ldots,v_{n-2})}\bm{k}(y_{n-1},v_{n-2},v_{n-1}|\,|v_{n-1}\cdot n(y_{n-1})|\\
\times f\left(y_{n-1}-\left[t-t_{-}(x,v)-\sum_{k=0}^{n-2}\tau_{-}(y_{k},v_{k})\right]v_{n-1},v_{n-1}\right)\,\bm{m}(\d v_{n-1})\,,
\end{multline}
for any $f \in \D_{0}$ where 
$$y_{0}=x-t_{-}(x,v)v, \qquad y_{k+1}=y_{k}-\tau_{-}(y_{k},v_{k})v_{k}, \qquad\, k=0,\ldots,n-2.$$
and,  
\begin{multline*}
A_{t}(x,v,v_{0},\ldots,v_{n-2})\\
=\left\{v_{n-1} \in \Gamma_{+}(y_{n-1})\;;\; \sum_{k=0}^{n-2}\tau_{-}(y_{k},v_{k})<t-t_{-}(x,v)<\sum_{k=0}^{n-1}\tau_{-}(y_{k},v_{k})\right\}.\end{multline*}
\end{lemme}
\begin{proof} The proof is by direct inspection. For instance, it is easy to see that, given $f \in \D_{0}$ and $t >t_{-}(x,v)$, one has
\begin{multline*}
\left[\bm{U}_{1}(t)f\right](x,v)=\int_{\Gamma_{+}(y_{0})}\bm{k}(y_{0},v,v_{0})f(y_{0}-(t-t_{-}(x,v))v_{0},v_{0})\\
\times \ind_{\{t_{-}(x,v) < t < t_{-}(x,v)+\tau_{-}(y_{0},v_{0})\}}|v_{0}\cdot n(y_{0})|\bm{m}(\d v_{0}),
\end{multline*}
and
\begin{multline*}
\left[\bm{U}_{2}(t)f\right](x,v)=\int_{\Gamma_{+}(y_{0})}\bm{k}(y_{0},v,v_{0})|v_{0}\cdot n(y_{0})|\bm{m}(\d v_{0})\\
\int_{\Gamma_{+}(y_{1})}\bm{k}(y_{1},v_{0},v_{1})f(y_{1}-(t-t_{-}(x,v)-\tau_{-}(y_{0},v_{0}))v_{1},v_{1})\\
\times\ind_{\{t_{-}(x,v)+\tau_{-}(y_{0},v_{0})<t<t_{-}(x,v)+\tau_{-}(y_{0},v_{0})+\tau_{-}(y_{1},v_{1})\}}|v_{1}\cdot n(y_{1})|\bm{m}(\d v_{1}).
\end{multline*}
The proof for $n \geq 3$ is then easily deduced by induction.\end{proof}
One has then the following proven in \cite{luisa} (see also \cite[Appendix A and Theorem 3.8]{ALMJM}):
\begin{theo}\label{theo:UKT} For any $k \geq 1$, $f \in \D_{0}$ one has $\bm{U}_{k}(t)f \in \X_{0}$ for any $t \geq 0$ with
$$\|\bm{U}_{k}(t)f\|_{\X_{0}} \leq \,\|f\|_{\X_{0}}.$$
In particular, $\bm{U}_{k}(t)$ can be extended to be a bounded linear operator, still denoted $\bm{U}_{k}(t) \in\mathscr{B}(\X_{0})$ with
\begin{equation}\label{eq:normUk}
\|\bm{U}_{k}(t)\|_{\mathscr{B}(\X_{0})} \leq 1\qquad \forall t \geq 0, k \geq 1.\end{equation}
Moreover, the following holds for any $k \geq 1$
\begin{enumerate}
\item $(\bm{U}_{k}(t))_{t \geq 0}$ is a strongly continuous family of $\mathscr{B}(\X_{0})$.
\item For any $f \in \X_{0}$ and any $t,s \geq 0$, it holds
$$\bm{U}_{k}(t+s)f=\sum_{j=0}^{k}\bm{U}_{j}(t)\bm{U}_{k-j}(s)f.$$
\item For any $f \in \D_{0}$, the mapping $t \geq 0 \mapsto \bm{U}_{k}(t)f$ is differentiable with 
$$\dfrac{\d}{\d t}\bm{U}_{k}(t)f=\bm{U}_{k}(t)\T_{\mathrm{max}}f \qquad \forall t \geq 0.$$
\item For any $f \in \D_{0}$, one has $\bm{U}_{k}(t)f \in \D(\T_{\mathrm{max}})$ for all $t \geq 0$ with 
$$\T_{\mathrm{max}}\bm{U}_{k}(t)f=\bm{U}_{k}(t)\T_{\mathrm{max}}f=\bm{U}_{k}(t)\T_{0}f.$$
\item For any $f \in \D_{0}$ and any $t \geq 0$, the traces $\B^{\pm}\bm{U}_{k}(t)f \in L^{1}_{\pm}$ and the mappings $t \geq 0\mapsto \B^{\pm}\bm{U}_{k}(t)f \in L^{1}_{\pm}$ are continuous. Moreover, for all $f \in X$ and $t >0$, one has
$$\B^{\pm} \int_{0}^{t}\bm{U}_{k}(s)f \d s \in L^{1}_{\pm} \qquad \text{ with } \quad \B^{-}\int_{0}^{t}\bm{U}_{k}(s)f\d s=\H\B^{+}\int_{0}^{t}\bm{U}_{k-1}(s)f \d s.$$
\item For any $f \in \D_{0}$, it holds
$$\int_{0}^{t}\|\B^{+}\bm{U}_{k}(s)f\|_{\lp}\d s \leq  \,\int_{0}^{t}\|\B^{+}\bm{U}_{k-1}(s)f\|_{\lp} \d s, \qquad \forall t \geq 0.$$
\item For any $f \in \X_{0}$ and $\lambda >0$, setting 
$$\mathcal{L}_{k}(\l)f=\int_{0}^{\infty}\exp(-\lambda t)\bm{U}_{k}(t)f \d t$$ one has, for $k \geq 1$,
$$\mathcal{L}_{k}(\l)f \in \D(\T_{\mathrm{max}}) \qquad \text{ with } \qquad \T_{\mathrm{max}}\mathcal{L}_{k}(\l)f=\lambda\,\mathcal{L}_{k}(\l)f$$
and  $\B^{\pm}\mathcal{L}_{k}(\l)f \in L^{1}_{\pm}$ with 
$$\B^{-}\mathcal{L}_{k}(\l)f=\H\B^{+}\mathcal{L}_{k-1}(\l)f \qquad \B^{+}\mathcal{L}_{k}(\l)f=(\mathsf{M}_{\lambda}\H)^{k}\mathsf{G}_{\lambda}f.$$
\item For any $f \in \X_{0}$, the series $\sum_{k=0}^{\infty}\bm{U}_{k}(t)f$ is strongly convergent and it holds
$$U_{\H}(t)f=\sum_{k=0}^{\infty}\bm{U}_{k}(t)f$$
\end{enumerate}
\end{theo} 
\begin{nb} One sees from the point (7) together with \cite[Theorem 2.4]{LMR} that, for any $k \geq 1$, 
$$\mathcal{L}_{k}(\l)f=\mathsf{\Xi}_{\l}\H\B^{+}\mathcal{L}_{k-1}(\l)f.$$
Since $\mathcal{L}_{0}(\l)f=\mathsf{R}_{\l}f$ we deduce that
$$\mathcal{L}_{1}(\l)=\mathsf{\Xi}_{\l}\H\B^{+}\mathsf{R}_{\l}=\mathsf{\Xi}_{\l}\H\mathsf{G}_{\l},$$
and, since $\B^{+}\mathsf{\Xi}_{\l}=\mathsf{M}_{\l}$, one gets by induction that, for any $k \geq 1$,
$$\mathcal{L}_{k}(\l)=\mathsf{\Xi}_{\l}\H\left(\mathsf{M}_{\l}\H\right)^{k-1}\mathsf{G}_{\l}.$$
In particular, one sees that, in the representation series \eqref{eq:lambdaTH} that, for any $n \geq 0$
\begin{equation}\label{eq:LaplaceDP}
\mathsf{\Xi}_{\lambda}\mathsf{H}\left(\mathsf{M}_{\lambda}\mathsf{H}\right)^{n}\mathsf{G}_{\lambda}f=\int_{0}^{\infty}\exp(-\l\,t)\bm{U}_{n+1}(t)f\d t\end{equation}
for any $\l >0$ which is of course coherent with the above point (4).
\end{nb}

\subsection{Decay of the iterates} We extend the decay of the semigroup $\left(U_{0}(t))\right)_{t\geq0}$ obtained in Lemma \ref{lem:decayU0} to the iterates 
$\left(\bm{U}_{k}(t)\right)_{t\geq0}$. To do so, we  first observe that Assumption \ref{hypH} \textit{1)} implies a nice behaviour of $\H$ for small velocities. More precisely, introducing for any $\delta >0$, the operator $\widetilde{\H}^{(\delta)} \in \mathscr{B}(\lp,\lm)$ given by
\begin{equation}\label{defi:Hepsi}
\widetilde{\H}^{(\delta)}\psi(x,v)=\ind_{|v|\leq \delta}\H\psi(x,v) \qquad \forall \psi \in \lp, \quad (x,v) \in \Gamma_{-}\end{equation}
Then, for any $\psi \in \lp$, one has
\begin{equation*}\begin{split}
\|\widetilde{\H}^{(\delta)}\psi\|_{\lm}&=\int_{\Gamma_{-}}\ind_{|v| \leq \delta}|\H\psi(x,v)|\d\mu_{-}(x,v)\\
&\leq \delta^{n+1}\int_{\Gamma_{-}}\ind_{|v|\leq \delta}|\H\psi(x,v)||v|^{-(n+1)}\d\mu_{-}(x,v)\end{split}\end{equation*}
from which we deduce the estimate
\begin{equation}\label{eq:Hepsi}
\|\widetilde{\H}^{(\delta)}\|_{\mathscr{B}(\lp,\lm)} \leq \|\H\|_{\mathscr{B}(\lp,\Y^{-}_{n+1})}\,\delta^{n+1} \qquad \forall n \leq N_{\H}.\end{equation}
Having such a property in mind, we can deduce the decay of $\bm{U}_{k}(t)$ as $t \to \infty$ for any $k \in \N$. For the clarity of exposition, we give full details for the decay of $\bm{U}_{1}(t)$. 
\begin{lemme}\label{lem:decayU1}
Let $k \in \N$ and $f \in \X_{k}$. Then, there exists some universal constant $C_{k} >0$ (depending only on $\H$, $k$ and $D$ but not  on $f$) such that
$$\left\|\bm{U}_{1}(t)f\right\|_{\X_{0}} \leq  C_{k}\left(t^{-(N_{\H}+1)} + t^{-k}\right)\,\|f\|_{\X_{k}}, \qquad \forall t >0$$
where $N_{\H}$ is defined in Assumption \ref{hypH}.
\end{lemme}
\begin{proof} The proof is based upon the decomposition of $\H$ for small and large velocities. Namely, we introduce, for some $\delta >0$ to be determined, the splitting
$$\H=\H^{(\delta)}+\widetilde{\H}^{(\delta)}$$
where $\widetilde{\H}^{(\delta)}$ is defined in \eqref{defi:Hepsi} and 
$$\H^{(\delta)}\psi(x,v)=\ind_{|v| >\delta}\H\psi(x,v), \qquad \psi \in \lp, \quad (x,v) \in \Gamma_{-}.$$ With such a splitting, one has of course, 
$$\bm{U}_{1}(t)=\bm{U}_{1}^{(\delta)}(t)+\widetilde{\bm{U}}_{1}^{(\delta)}(t), \qquad t >0, \qquad \delta >0$$
where $\bm{U}_{1}^{(\delta)}(t),\widetilde{\bm{U}}_{1}^{(\delta)}(t)$ are  constructed as in Definition \ref{defi:Uk} with $\H$ replaced respectively by $\H^{(\delta)}$ and $\widetilde{H}^{(\delta)}$. Let now fix $k\geq 1,$ $f \in \X_{k}, t >0.$ One has
\begin{equation*}
\|\bm{U}_{1}(t)f\|_{\X_{0}} \leq \|\bm{U}_{1}^{(\delta)}(t)f\|_{\X_{0}}+\|\widetilde{\bm{U}}_{1}^{(\delta)}(t)f\|_{\X_{0}} 
\leq \|{\bm{U}}_{1}^{(\delta)}(t)f\|_{\X_{0}} + \|\widetilde{\H}^{(\delta)}\|_{\mathscr{B}(\lp,\lm)}\|f\|_{\X_{0}} \end{equation*}
where we used \eqref{eq:Vnt}. Using now \eqref{eq:Hepsi}, there is $C$ (depending only on $\H$) such that
\begin{equation}\label{eq:U1tep}
\|\bm{U}_{1}(t)f\|_{\X_{0}} \leq C\delta^{N_{\H}+1}\|f\|_{\X_{0}}+\|{\bm{U}}_{1}^{(\delta)}(t)f\|_{\X_{0}} \qquad \forall \delta >0.\end{equation}
Let us focus then on the estimate for $\ {\bm{U}_{1}^{(\delta)}}(t)f$. The crucial point is of course that $ {\bm{U}}_{1}^{(\delta)}(t)f$ is supported on $\Omega \times \{v \in V\;;\,|v| > \delta\}$ and, on this set, $\tau_{-}(\cdot,\cdot)$ is uniformly bounded since
$$t_{-}(x,v) \leq \frac{D}{|v|} \leq \frac{D}{\delta} \qquad \forall (x,v) \in \Omega \times V,\:|v| >\delta.$$ 
Let us then consider $x \in \Omega$ and $|v| >\delta$ and $t >\frac{D}{\delta} > t_{-}(x,v)$. We recall (see Lemma \ref{lem:exprUn}) that
\begin{multline*}
\left[\bm{U}_{1}^{(\delta)}(t)f\right](x,v)=\ind_{|v|>\delta}\int_{\Gamma_{+}(y_{0})}\bm{k}(y_{0},v,v_{0})f(y_{0}-(t-t_{-}(x,v))v_{0},v_{0})\\
\times \ind_{\{t_{-}(x,v) < t < t_{-}(x,v)+\tau_{-}(y_{0},v_{0})\}}|v_{0}\cdot n(y_{0})|\bm{m}(\d v_{0}).
\end{multline*}
One sees that, for $t > t_{-}(x,v)$ and $|v|>\delta$, 
$$\ind_{\{t_{-}(x,v) < t < t_{-}(x,v)+\tau_{-}(y_{0},v_{0})\}} \neq 0 \iff \tau_{-}(y_{0},v_{0}) > t-t_{-}(x,v) \geq t-\frac{D}{\delta}.$$
Since moreover $\tau_{-}(y_{0},v_{0}) \leq \frac{D}{|v_{0}|}$ one deduces that
$$\ind_{\{t_{-}(x,v) < t < t_{-}(x,v)+\tau_{-}(y_{0},v_{0})\}} \neq0  \:\:\Longrightarrow\:\:  \frac{D}{|v_{0}|} > t-\frac{D}{\delta}.$$
Therefore, 
\begin{multline*}
\left|\left[\bm{U}_{1}^{(\delta)}(t)f\right](x,v)\right| \leq \ind_{|v| >\delta}\int_{\Gamma_{+}(y_{0})} \ind_{|v_{0}| \leq \frac{D}{t-\frac{D}{\delta}}}\bm{k}(y_{0},v,v_{0})\\
\times \ind_{\{t_{-}(x,v) < t < t_{-}(x,v)+\tau_{-}(y_{0},v_{0})\}}\left|f(y_{0}-(t-t_{-}(x,v))v_{0},v_{0})\right||v_{0}\cdot n(y_{0})|\bm{m}(\d v_{0}).\end{multline*}
Introducing 
$$g(z,w)=|w|^{-k}\left|f(z,w)\right|, \qquad \forall (z,w) \in \Omega \times V$$
one deduces easily that, for any $t > \frac{D}{\delta}$,
\begin{equation*}
\left|\left[\bm{U}_{1}^{(\delta)}(t)f\right](x,v)\right| \leq \frac{D^{j}}{\left(t-\frac{D}{\delta}\right)^{j}}\left[\bm{U}_{1}^{(\delta)}(t)g\right](x,v)\end{equation*}
Therefore,
$$\left\|\bm{U}_{1}^{(\delta)}(t)f\right\|_{\X_{0}} \leq \frac{D^{j}}{\left(t-\frac{D}{\delta}\right)^{j}}\left\|\bm{U}_{1}^{(\delta)}(t)g\right\|_{\X_{0}}\leq \frac{D^{j}}{\left(t-\frac{D}{\delta}\right)^{j}}\left\|g\right\|_{\X_{0}}, \qquad \forall t > \frac{D}{\delta}$$
where we used \eqref{eq:normUk}. Combining this with \eqref{eq:U1tep}, and since $\max(\|f\|_{\X_{0}},\|g\|_{\X_{0}}) \leq \|f\|_{\X_{k}}$, one deduces that 
$$\|\bm{U}_{1}(t)f\|_{\X_{0}} \leq \left(C\e^{N_{\H}+1}+\frac{D^{j}}{\left(t-\frac{D}{\delta}\right)^{j}}\right)\|f\|_{\X_{k}}, \qquad \forall t > \frac{D}{\delta}.$$
Choosing $\delta$ such that $t=\frac{2D}{\delta}$, we get the result.
\end{proof}

We generalise this to the other iterates
\begin{lemme}\label{lem:decayUn}
Let $k \in \N$ and $f \in \X_{k}$. Then, there exists some universal constant $C_{k} >0$ (depending only on $\H$, $k$ and $D$ but not on $f$) such that, for any $n \geq 1$,
$$\left\|\bm{U}_{n}(t)f\right\|_{\X_{0}} \leq  C_{k}\left(\left(2^{n}-1\right)\left(\frac{n+1}{t}\right)^{(N_{\H}+1)}+\left(\frac{n+1}{t}\right)^{k}\right)\,\|f\|_{\X_{k}}, \qquad \forall t >0$$
where $N_{\H}$ is defined in Assumption \ref{hypH}.
\end{lemme}
\begin{proof} The proof uses the same ideas introduced in the proof for $n=1$. Given $\e >0$, we still introduce the splitting $\H=\widetilde{\H}^{(\delta)}
+\H^{(\delta)}$. With the representation of $\H$ given  by \eqref{eq:Hhkernel}, it is clear that
\begin{equation*}\begin{cases}
\widetilde{\H}^{(\delta)}\psi(x,v)&=\ds\int_{v'\cdot n(x) > 0}\widetilde{\bm{k}}^{(\delta)}(x,v,v')\psi(x,v')\,|v'\cdot n(x)|\bm{m}(\d v'),\\
\\
\H^{(\delta)}\psi(x,v)&=\ds\int_{v'\cdot n(x) > 0}\bm{k}^{(\delta)}(x,v,v')\psi(x,v')\,|v'\cdot n(x)|\bm{m}(\d v'), \qquad \psi \in \lp,\qquad (x,v) \in \Gamma_{-}
\end{cases}\end{equation*}
where
\begin{multline*}
\widetilde{\bm{k}}^{(\delta)}(x,v,v')=\ind_{|v|\leq \delta}\bm{k}(x,v,v'), \qquad \bm{k}^{(\delta)}(x,v,v')=\ind_{|v|>\delta}\bm{k}(x,v,v'),\\
\text{and}  \qquad (x,v) \in \Gamma_{-}, \:v'\in \Gamma_{+}(x).\end{multline*}
Since $\bm{k}(\cdot,\cdot,\cdot)=\widetilde{\bm{k}}^{(\delta)}(\cdot,\cdot,\cdot)+\bm{k}^{(\delta)}(\cdot,\cdot,\cdot)$, 
using the representation formula \eqref{eq:Untf} one sees by a simple combinatorial argument that, for $n \geq 1$, one can write
$$\bm{U}_{n}(t)=\bm{U}_{n}^{(\delta)}(t)+\widetilde{\bm{U}}_{n}^{(\delta)}(t)$$
where $\bm{U}_{n}^{(\delta)}(t)$ is given by \eqref{eq:Untf} with \emph{all kernels} $\bm{k}(\cdot,\cdot,\cdot)$ replaced with $\bm{k}^{(\delta)}(\cdot,\cdot,\cdot)$ whereas the reminder term $\widetilde{\bm{U}}_{n}^{(\delta)}(t)$ is the some of $2^{n}-1$ operators 
$$\widetilde{\bm{U}}_{n}^{(\delta)}(t)=\sum_{j=1}^{2^{n}-1}\bm{V}_{n}^{(j)}(t)$$
where, for any $j \in \{1,\ldots,2^{n}-1\}$, $\bm{V}_{n}^{(j)}(t)$ is defined by \eqref{eq:Untf} with \emph{at least one kernel} $\widetilde{\bm{k}}^{(\delta)}(\cdot,\cdot,\cdot)$. Alternatively, this means that $\bm{V}_{n}^{(j)}(t)$ is defined as in Definition \ref{defi:Vk} for a family of boundary operators $(\H_{1},\ldots,\H_{n})$ where there is at least one $i \in \{1,\ldots,n\}$ such that $\H_{i}=\widetilde{\H}^{(\delta)}$ (the other ones being indifferently $\H^{(\delta)}$ or $\widetilde{\H}^{(\delta)}$).  Using Proposition \ref{prop:Vnt}, one has then 
$$\|\bm{V}_{n}^{(j)}(t)\|_{\mathscr{B}(\X_{0})} \leq \|\widetilde{\H}^{(\delta)}\|_{\mathscr{B}(\lp,\lm)}.$$
Therefore
$$\left\|\widetilde{\bm{U}}_{n}^{(\delta)}(t)f\right\|_{\X_{0}} \leq \left(2^{n}-1\right)\|\widetilde{\H}^{(\delta)}\|_{\mathscr{B}(\lp,\lm)}\|f\|_{\X_{0}}$$
and 
\begin{equation}\label{eq:Untep}
\left\|\bm{U}_{n}(t)f\right\|_{\X_{0}}\leq \|\bm{U}_{n}^{(\delta)}(t)f\|_{\X_{0}}+ C(2^{n}-1)\,\delta^{N_{\H}+1}\|f\|_{\X_{0}}, \qquad t >0, \qquad \e >0\end{equation}
where we used \eqref{eq:Hepsi}. We focus now on the expression of $\bm{U}_{n}^{(\delta)}(t)f$. As before, $\bm{U}_{n}^{(\delta)}(t)f$ is supported on $\Omega\times \{v \in V\;;\;|v| >\delta\}$. We have, from \eqref{eq:Untf},
\begin{multline}\label{eq:Untfe}
\left[\bm{U}_{n}^{(\delta)}(t)f\right](x,v)=\\
\ind_{|v| >\delta}\int_{\Gamma_{+}^{\delta}(y_{0})}\bm{k}(y_{0},v,v_{0})|v_{0}\cdot n(y_{0})|\bm{m}(\d v_{0})\int_{\Gamma_{+}^{\delta}(y_{1})}\bm{k}(y_{1},v_{0},v_{1})|v_{1}\cdot n(y_{1})|\bm{m}(\d v_{1})\ldots\\
\int_{\Gamma_{+}^{\delta}(y_{n-2})}\bm{k}(y_{n-2},v_{n-3},v_{n-2})\,|v_{n-2}\cdot n(y_{n-2})|\bm{m}(\d v_{n-2})\times\\
\times\int_{A_{t}^{\delta}(x,v,v_{0},\ldots,v_{n-2})}\bm{k}(y_{n-1},v_{n-2},v_{n-1}|\,|v_{n-1}\cdot n(y_{n-1})|\\
\times f\left(y_{n-1}-\left[t-t_{-}(x,v)-\sum_{k=0}^{n-2}\tau_{-}(y_{k},v_{k})\right]v_{n-1},v_{n-1}\right)\,\bm{m}(\d v_{n-1})\,,
\end{multline}
where $y_{0}=x-t_{-}(x,v)v,$ $y_{k+1}=y_{k}-\tau_{-}(y_{k},v_{k})v_{k}$ $(k=0,\ldots,n-2),$ 
$$\Gamma_{+}^{\delta}(y_{k})=\Gamma_{+}(y_{j}) \cap \{v \in V\,;\,|v| > \delta\}, \qquad k=0,\ldots,n-2$$
and,  
\begin{multline*}
A_{t}^{\delta}(x,v,v_{0},\ldots,v_{n-2})\\
=\left\{v_{n-1} \in \Gamma_{+}^{\delta}(y_{n-1})\;;\; \sum_{k=0}^{n-2}\tau_{-}(y_{k},v_{k})<t-t_{-}(x,v)<\sum_{k=0}^{n-1}\tau_{-}(y_{k},v_{k})\right\}.\end{multline*}
Notice that, since $|v| >\delta$ and $|v_{k}| >\delta$ for any $k=0,\ldots,n-2,$ one has 
$$t_{-}(x,v)+\sum_{k=0}^{n-2}\tau_{-}(y_{k},v_{k}) \leq \frac{nD}{\delta}.$$
Let us consider then $t > \frac{nD}{\delta}$. One sees then that if $v_{n-1} \in A_{t}^{\delta}(x,v,v_{0},\ldots,v_{n-2})$ then
$$\frac{D}{|v_{n-1}|} \geq \tau_{-}(y_{n-1},v_{n-1}) \geq t-\left(t_{-}(x,v)+\sum_{k=0}^{n-2}\tau_{-}(y_{k},v_{k})\right) \geq t-\frac{nD}{\delta}$$
which implies 
\begin{equation}\label{eq:vn-1D}
|v_{n-1}| < \frac{D}{t-\frac{nD}{\delta}}.\end{equation}
As in the proof of Lemma \ref{lem:decayU1}, one introduce now
$$g(z,w)=|w|^{-k}\left|f(z,w)\right|, \qquad (z,w) \in\Omega\times V$$
and sees then from \eqref{eq:Untep},
\begin{multline*}
\left|\left[\bm{U}_{n}^{(\delta)}(t)f\right](x,v)\right|\\
\leq\ind_{|v| >\delta}\int_{\Gamma_{+}^{\delta}(y_{0})}\bm{k}(y_{0},v,v_{0})|v_{0}\cdot n(y_{0})|\bm{m}(\d v_{0})\int_{\Gamma_{+}^{\delta}(y_{1})}\bm{k}(y_{1},v_{0},v_{1})|v_{1}\cdot n(y_{1})|\bm{m}(\d v_{1})\ldots\\
\int_{\Gamma_{+}^{\delta}(y_{n-2})}\bm{k}(y_{n-2},v_{n-3},v_{n-2})\,|v_{n-2}\cdot n(y_{n-2})|\bm{m}(\d v_{n-2})\times\\
\times\int_{A_{t}^{\delta}(x,v,v_{0},\ldots,v_{n-2})}\bm{k}(y_{n-1},v_{n-2},v_{n-1}|\,|v_{n-1}\cdot n(y_{n-1})|\\
\times |v_{n-1}|^{k}g\left(y_{n-1}-\left[t-t_{-}(x,v)-\sum_{k=0}^{n-2}\tau_{-}(y_{k},v_{k})\right]v_{n-1},v_{n-1}\right)\,\bm{m}(\d v_{n-1})\,.
\end{multline*}
Therefore,  from \eqref{eq:vn-1D},
$$\left|\left[\bm{U}_{n}^{(\delta)}f\right](x,v)\right| \leq \left(\frac{D}{t-\frac{nD}{\delta}}\right)^{k}\left[\bm{U}_{n}^{(\delta)}(t)g\right](x,v) \qquad t >\frac{nD}{\delta}.$$
Consequently
$$\left\|\bm{U}_{n}^{(\delta)}(t)f\right\|_{\X_{0}} \leq \frac{D^{k}}{\left(t-\frac{nD}{\delta}\right)^{k}}\,\left\|\bm{U}_{n}^{(\delta)}(t)g\right\|_{\X_{0}}\leq \frac{D^{k}}{\left(t-\frac{nD}{\delta}\right)^{k}}\|g\|_{\X_{0}}, \qquad t >\frac{nD}{\delta}$$
thanks to \eqref{eq:normUk}. Combining this with \eqref{eq:Untep} we finally obtain
$$\left\|\bm{U}_{n}(t)f\right\|_{\X_{0}} \leq \frac{D^{k}}{\left(t-\frac{nD}{\delta}\right)^{k}}\|g\|_{\X_{0}}+ C(2^{n}-1)\,\delta^{N_{\H}+1}\|f\|_{\X_{0}}, \quad t > \frac{nD}{\delta}.$$
Picking now $\delta >0$ such that $t=\frac{(n+1)D}{\delta}$ we get the result since both $\|g\|_{\X_{0}}$ and $\|f\|_{\X_{0}}$ are smaller than $\|f\|_{\X_{k}}.$
\end{proof}

We deduce directly from Lemmas \ref{lem:decayU0} and \ref{lem:decayUn}
the following
\begin{propo}\label{prop:Snt} Assume that
$$f \in \X_{N_{\H}+1}$$
then, for any $n \geq 1$, there exists $\bm{C}_{n} >0$ such that
\begin{equation}\label{eq:UntNH}
\left\|\sum_{k=0}^{n}\bm{U}_{k}(t)f\right\|_{\X_{0}} \leq \bm{C}_{n}\,t^{-\left(N_{\H}+1\right)}\left\|f\right\|_{\X_{N_{\H}+1}} \qquad \forall t >0.\end{equation}
\end{propo}

\subsection{Representation formulae for remainder terms}

Introducing
$$\bm{S}_{n}(t):=U_{\H}(t)-\sum_{k=0}^{n}\bm{U}_{k}(t), \qquad n \geq 1,\;\,t >0$$
one has the following
\begin{propo} For any $n \in \N$, $n\geq \mathsf{p}$ where $\mathsf{p}$ defined through \eqref{eq:power}. Then, for any $f \in \X_{0}$,
one has
\begin{equation}\label{eq:maindecay}
\bm{S}_{n}(t)f=\frac{\exp(\e t)}{2\pi}\lim_{\ell\to\infty} \int_{-\ell}^{\ell}\exp\left(i\eta t\right)\Upsilon_{n}(\e+i\eta)f\d\eta, \qquad \forall f \in \X_{0}\end{equation}
for any $t >0$, $\e >0.$
%
%
%
\end{propo}
\begin{proof}  One notices that, for any $n \geq0$ and any $f \in \X_{0}$, it holds
\begin{equation*}\begin{split}
\int_{0}^{\infty}\exp\left(-\l t\right)\bm{S}_{n}(t)f\d t&=\sum_{k=n}^{\infty}\int_{0}^{\infty}\exp\left(-\l t\right)\bm{U}_{k+1}(t)f\d t\\
&=\sum_{k=n}^{\infty}\mathsf{\Xi}_{\l}\H\left(\mathsf{M}_{\l}\H\right)^{k}\mathsf{G}_{\l}f=\Upsilon_{n}(\l)f, \qquad \mathrm{Re}\l >0\end{split}\end{equation*}
where we used \eqref{eq:LaplaceDP} together with the fact that, for $\mathrm{Re}\l >0$, $\|\mathsf{M}_{\l}\H\|_{\mathscr{B}(\lp)} < 1.$ Since moreover, for any $f \in \X_{0}$, the mapping $t \geq 0 \mapsto \bm{S}_{n}(t)f$ is continuous and bounded, with $\bm{S}_{n}(0)f=0$, one applies the complex Laplace inversion formula \cite[Theorem 4.2.21]{arendt} to deduce 
\begin{equation}\label{eq:maindecay}
\bm{S}_{n}(t)f=\frac{\exp(\e t)}{2\pi}\lim_{L\to\infty}\frac{1}{2L}\int_{-L}^{L}\d \ell \int_{-\ell}^{\ell}\exp\left(i\eta t\right)\Upsilon_{n}(\e+i\eta)f\d\eta, \qquad \forall f \in \X_{0}\end{equation}
for any $t >0$, $\e >0$, i.e. $\bm{S}_{n}(t)f$ is the Cesar\`o limit of the family 
$$\left(\int_{-\ell}^{\ell}\exp\left(i\eta t\right)\Upsilon_{n}(\e+i\eta)f\d\eta\right)_{\ell}.$$ Let us prove it is actually a classical limit. Fix $\e >0$ and $f \in\X_{0}$. Arguing as in \eqref{eq:Psin},  
\begin{multline*}
\left\|\Upsilon_{n}(\e+i\eta)\right\|_{\mathscr{B}(\X_{0},\X_{k})} 
\leq \|\mathsf{\Xi}_{0}\|_{\mathscr{B}(\Y_{k}^{-},\X_{k})}\|\H\|_{\mathscr{B}(\lp,\Y^{-}_{k})}\left\|\left(\mathsf{M}_{\e+i\eta}\H\right)^{n}\right\|_{\mathscr{B}(\lp)}\\
\|\Rs(1,\mathsf{M}_{\e+i\eta}\H)\|_{\mathscr{B}(\lp)}\,\|\mathsf{G}_{\e+i\eta}\|_{\mathscr{B}(\X_{0},\lp)}\end{multline*}
Since, for any $\e >0$,
$$\left\|\Rs(1,\mathsf{M}_{\e+i\eta}\H)\right\|_{\mathscr{B}(\lp)} \leq \|\Rs(1,\mathsf{M}_{\e}\H)\|_{\mathscr{B}(\lp)}$$
while $\sup_{\eta}\|\mathsf{G}_{\e+i\eta}\|_{\mathscr{B}(\X_{0},\lp)} \leq 1$ we deduce that there exists $C_{\e} >0$ such that
$$\left\|\Upsilon_{n}(\e+i\eta)\right\|_{\mathscr{B}(\X_{0},\X_{k})} \leq C_{\e}\left\|\left(\mathsf{M}_{\e+i\eta}\H\right)^{n}\right\|_{\mathscr{B}(\lp)}, \qquad \forall \eta \in \R.$$
For $n \geq \mathsf{p}$, one has $\left\|\left(\mathsf{M}_{\e+i\eta}\H\right)^n\right\|_{\mathscr{B}(\lp)} \leq \left\|\left(\mathsf{M}_{\e+i\eta}\H\right)^{\mathsf{p}}\right\|_{\mathscr{B}(\lp)}$, we deduce from \eqref{eq:power} that there is $M_{\e} >0$ such that
$$\int_{-\infty}^{\infty}\left\|\Upsilon_{n}(\e+i\eta)\right\|_{\mathscr{B}(\X_{0},\X_{k})} \d \eta \leq M_{\e}, \qquad \forall \e >0.$$
This of course implies that
$$\int_{-\infty}^{\infty}\left\|\exp\left((\e+i\eta)t\right)\Upsilon_{n}(\e+i\eta)\right\|_{\mathscr{B}(\X_{0},\X_{k})} \d \eta \leq M_{\e}\exp(\e t), \qquad \forall \e >0.$$
In particular, for any $f \in \X_{0}$, the limit 
$$\lim_{\ell\to\infty}\frac{1}{2\pi}\int_{-\ell}^{\ell}\exp\left((\e+i\eta)t\right)\Upsilon_{n}(\e+i\eta)f\d\eta$$
exists in $\X_{k}$. Since its Cesar\`o limit is $\bm{S}_{n}(t)f$, we deduce the result.
\end{proof}
\begin{nb} It appears that the convergence actually holds in operator norm, i.e. the integral 
$$\frac{\exp(\e t)}{2\pi}\lim_{\ell\to\infty} \int_{-\ell}^{\ell}\exp\left(i\eta t\right)\Upsilon_{n}(\e+i\eta)\d\eta$$
converges to $\bm{S}_{n}(t)$ in $\mathscr{B}(\X_{0},\X_{k})$ as $\ell \to \infty$. As a consequence, under Assumption \ref{hypH}, one sees that $\bm{S}_{n}(t)$ is a compact operator for $n$ large enough.\end{nb}

It is clear that, from the above representation formula, no decay of $\bm{S}_{n}(t)$ can be expected because of the growing function $\exp(\e t)$. As said already in the Introduction, we will need therefore to derive a second representation formula showing that \eqref{eq:maindecay}
 actually holds for $\e=0$. Of course, to do so, we need first to suitably define the trace of $\Upsilon_{n}(\lambda)$ on the imaginary axis, i.e. to define properly the limit as $\e \to 0^{+}$ of $\Upsilon(\e+i\eta)f$ for suitable $f$. This will require several preliminary definitions and regularity estimates of the various operators defining $\Upsilon_{n}(\e+i\eta).$

\section{General regularity and spectral results}\label{sec:fineR}

This section is devoted to the core technical results which will allow to define the boundary function of the mapping $\l \in \C_{+} \longmapsto \Upsilon_{n}(\l) \in \mathscr{B}(\X_{0}).$ 

We first start with the convergence of the various operators $\mathsf{M}_{\l},\mathsf{\Xi}_{\l}$ and $\mathsf{G}_{\l}$. The proof of this result is postponed to the Appendix \ref{app:technA}
\begin{propo}\label{propo:convK} For any $f \in \X_{0}$, the limit
\begin{equation}\label{lem:unifGeis}
\lim_{\e\to0^{+}}\left\|\mathsf{G}_{\e+i\eta}f-\mathsf{G}_{i\eta}f\right\|_{\lp}=0\end{equation}
uniformly with respect to $\eta \in \R$ with
\begin{equation}\label{eq:RiemLeb}
\lim_{|\eta| \to \infty}\sup_{\e \in [0,1]}\left\|\mathsf{G}_{\e+i\eta}f\right\|_{\lp}=0.\end{equation}

 Moreover, for any $\eta \in \R$ and any $k \in \N$, 
\begin{equation}\label{eq:Meisk}
\left\|\mathsf{M}_{\delta+i\eta}-\mathsf{M}_{i\eta}\right\|_{\mathscr{B}(\Y^{-}_{k+1},\Y_{k})} \leq \delta\,D, \qquad \left\|\mathsf{\Xi}_{\e+i\eta}-\mathsf{\Xi}_{i\eta}\right\|_{\mathscr{B}(\Y_{k+1}^{-},\X_{k})} \leq \e\,D\end{equation}
where $D$ is the diameter of $\Omega$. Consequently, for any $k \leq N_{\H}$ where $N_{\H}$ is defined in \eqref{eq:HYn}
\begin{multline}\label{eq:MeisHk}
\left\|\mathsf{M}_{\delta+i\eta}\H-\mathsf{M}_{i\eta}\H\right\|_{\mathscr{B}(\lp,\Y_{k}^{+})} \leq  \delta\,D\,\|\H\|_{\mathscr{B}(\lp,\Y_{k+1}^{-})} \\
\text{ and } \quad \left\|\mathsf{\Xi}_{\e+i\eta}\H-\mathsf{\Xi}_{i\eta}\H\right\|_{\mathscr{B}(\lp,\X_{k})} \leq \e\,D\,\|\H\|_{\mathscr{B}(\lp,\Y_{k+1}^{-})} 
\qquad \forall \eta \in \R.\end{multline}
Moreover, for  any $j \in \N$, one has
\begin{equation}\label{cor:uniformpower}
\lim_{\delta\to0}\left\|\left(\mathsf{M}_{\e+i\eta}\H\right)^{j}-\left(\mathsf{M}_{i\eta}\H\right)^{j}\right\|_{\mathscr{B}(\lp)}=0\end{equation}
uniformly with respect to $\eta \in \R.$
\end{propo}

 \begin{nb}\label{cor:exten} An important consequence of the above Proposition \ref{propo:convK} is that the holomorphic functions
$$\lambda \in \C_{+} \mapsto \mathsf{M}_{\l}\H \in \mathscr{B}(\lp) \quad { and } \quad 
\lambda \in \C_{+} \mapsto \mathsf{\Xi}_{\l}\H \in \mathscr{B}(\lp,\X_{0})$$
can be extended to continuous functions on $\overline{\C}_{+}.$  Moreover, one easily deduces that
$$\int_{-\infty}^{\infty}\left\|\left(\mathsf{M}_{i\eta}\H\right)^{\mathsf{p}}\right\|_{\mathscr{B}(\lp)}\d\eta < \infty$$
where $\mathsf{p} >0$ is defined through \eqref{eq:power}.
\end{nb}

\subsection{Spectral properties of $\mathsf{M}_{\l}\H$ along the imaginary axis}

We study here more carefully the properties of $\mathsf{M}_{i\eta}\H$ for $\eta \in\R$. 
 \begin{propo}\phantomsection \label{prop:Meps} 
For any $\lambda \in \C \setminus \{0\}$ with $\mathrm{Re}\l \geq 0$, 
$$r_{\sigma}(\mathsf{M}_{\l}\H) < 1.$$
\end{propo}
\begin{proof} We give the proof only for $\mathrm{Re}\l=0$, the case $\mathrm{Re}\l >0$ being similar. Since $\left|\mathsf{M}_{i\eta}\H\psi(x,v)\right|=\left|\mathsf{M}_{0}\H\psi(x,v)\right|$ for any $\psi \in \lp$, $(x,v) \in \Gamma_{+},$ $\eta \in \R$, one sees that 
$$\mathsf{M}_{i\eta}\H \in \mathscr{B}(\lp) \qquad \text{ with } \quad \left|\mathsf{M}_{i\eta}\H \right| \leq \mathsf{M}_{0}\H$$
where $\left|\mathsf{M}_{i\eta}\H\right|$ denotes the absolute value operator of $\mathsf{M}_{i\eta}\H$ (see \cite{chacon}). The operator $\mathsf{M}_{0}\H$ being power compact, the same holds for $\left|\mathsf{M}_{i\eta}\H \right|$ by a domination argument  so that
$$r_{\mathrm{ess}}(\left|\mathsf{M}_{i\eta}\H \right|)=0$$
where $r_{\mathrm{ess}}(\cdot)$ denotes the essential spectral radius. We prove that $r_{\sigma}(\left|\mathsf{M}_{i\eta}\H\right|) < 1$ by contradiction: assume, on the contrary, $r_{\sigma}(\left|\mathsf{M}_{i\eta}\H\right|)=1 > r_{\mathrm{ess}}(\left|\mathsf{M}_{i\eta}\H\right|)=0$, then $r_{\sigma}(\left|\mathsf{M}_{i\eta}\H\right|)$ is an isolated eigenvalue of $\left|\mathsf{M}_{i\eta}\H\right|$ with finite algebraic multiplicity and also an eigenvalue of the dual operator, associated to a nonnegative eigenfunction. From the fact that $\left|\mathsf{M}_{i\eta}\H\right| \leq \mathsf{M}_{0}\H$ with $\left|\mathsf{M}_{i\eta}\H\right| \neq \mathsf{M}_{0}\H$, one can invoke \cite[Theorem 4.3]{marek} to get that
$$r_{\sigma}(\left|\mathsf{M}_{i\eta}\H\right|) < r_{\sigma}(\mathsf{M}_{0}\H)=1$$
which is a contradiction. Therefore, $r_{\sigma}(\left|\mathsf{M}_{i\eta}\H\right|) < 1$ and, since $r_{\sigma}(\mathsf{M}_{i\eta}\H) \leq r_{\sigma}(\left|\mathsf{M}_{i\eta}\H\right|)$, the conclusion holds true. 
\end{proof}
 
We deduce the following 
\begin{cor}\label{cor:ResMei} For any $\eta_{0} \in \R\setminus\{0\}$, there is $0 < \delta < \tfrac{1}{2}|\eta_{0}|$ such that
$$\lim_{\e\to0^{+}}\sup_{|\eta-\eta_{0}| < \delta}\bigg\|\Rs(1,\mathsf{M}_{\e+i\eta}\H)-\Rs(1,\mathsf{M}_{i\eta}\H)\bigg\|_{\mathscr{B}(\lp)}=0.$$
\end{cor}
\begin{proof} Notice that, if $0 <\delta<\tfrac{|\eta_{0}|}{2}$ then, $\eta \neq 0$ whenever  $|\eta-\eta_{0}| < \delta$. Without loss of generality, we can assume $\eta_{0} >0$. From Proposition \ref{prop:Meps}, there is $\varrho \in (0,1)$ such that $r_{\sigma}(\mathsf{M}_{i\eta_{0}}\H) < \varrho <1$. In particular, there is $\ell \in \N$ such that
$$\left\|\left(\mathsf{M}_{i\eta_{0}}\H\right)^{\ell}\right\|_{\mathscr{B}(\lp)}^{\frac{1}{\ell}} < \varrho<1.$$
Since $\mathsf{M}_{i\eta}\H$ converges to $\mathsf{M}_{i\eta_{0}}\H$ in operator norm as $\eta \to \eta_{0}$, there is $\delta < \frac{\eta_{0}}{2}$ such that
$$\left\|\left(\mathsf{M}_{i\eta}\H\right)^{\ell}\right\|_{\mathscr{B}(\lp)} < \varrho^{\ell} \qquad \forall \eta \in (\eta_{0}-\delta,\eta_{0}+\delta).$$
Because of Eq. \eqref{cor:uniformpower}, there is $\e_{0} >0$ small enough, such that, for any $0<\e<\e_{0}$ we also have
$$\left\|\left(\mathsf{M}_{\e+i\eta}\H\right)^{\ell}\right\|_{\mathscr{B}(\lp)} < \varrho^{\ell} \qquad \forall \eta \in (\eta_{0}-\delta,\eta_{0}+\delta).$$
One has then
$$\Rs(1,\mathsf{M}_{\e+i\eta}\H)=\sum_{n=0}^{\infty}\left(\mathsf{M}_{\e+i\eta}\H\right)^{n}=\sum_{k=0}^{\infty}\sum_{j=0}^{\ell-1}\left(\mathsf{M}_{\e+i\eta}\H\right)^{k\ell+j}$$
and, similarly
$$\Rs(1,\mathsf{M}_{i\eta}\H)=\sum_{k=0}^{\infty}\sum_{j=0}^{\ell-1}\left(\mathsf{M}_{i\eta}\H\right)^{k\ell+j} \qquad  \forall \eta \in (\eta_{0}-\delta_{0},\eta_{0}+\delta_{0}).$$
Therefore,
$$\Rs(1,\mathsf{M}_{\e+i\eta}\H)-\Rs(1,\mathsf{M}_{i\eta}\H)=\sum_{k=0}^{\infty}\sum_{j=0}^{\ell-1}\left[\left(\mathsf{M}_{\e+i\eta}\H\right)^{k\ell+j}-\left(\mathsf{M}_{i\eta}\H\right)^{k\ell+j}\right]$$
for any $\e \in (0,\e_{0}),$ $\eta \in (\eta_{0}-\delta_{0},\eta_{0}+\delta_{0}).$ Using again \eqref{cor:uniformpower}, each term of the series converges to $0$ as $\eta \to 0$ \emph{uniformly} with respect to $\eta \in  (\eta_{0}-\delta_{0},\eta_{0}+\delta_{0}).$ To prove the result, it is enough therefore to show that the remainder of the series can be made arbitrarily small in operator norm \emph{uniformly} on $(\eta_{0}-\delta_{0},\eta_{0}+\delta_{0}).$ Since, for any $k,j \geq 0$
$$\left\|\left(\mathsf{M}_{\e+i\eta}\H\right)^{k\ell+j}\right\|_{\mathscr{B}(\lp)} \leq \left\|\left(\mathsf{M}_{\e+i\eta}\H\right)^{k\ell}\right\|_{\mathscr{B}(\lp)} \leq \left\|\left(\mathsf{M}_{\e+i\eta}\H\right)^{\ell}\right\|_{\mathscr{B}(\lp)}^{k} \leq \varrho^{k\ell}$$
for any $\e \in [0,\e_{0})$, we get that, for any $n \geq 0$
$$\sup_{|\eta-\eta_{0}| < \delta_{0}}\sup_{\e \in (0,\e_{0})}\sum_{k=n}^{\infty}\sum_{j=0}^{\ell-1}\left\|\left(\mathsf{M}_{\e+i\eta}\H\right)^{k\ell+j}-\left(\mathsf{M}_{i\eta}\H\right)^{k\ell+j}\right\|_{\mathscr{B}(\lp)} \leq 2\ell\sum_{k=n}^{\infty}\varrho^{\ell\,k}$$
which tends to $0$ as $n \to \infty$. This combined with the term-by-term convergence of the series as $\e \to 0$  gives the result.
\end{proof}

\subsection{First consequence on the spectrum of $\T_{\H}$}

We recall the following result (see \cite[Theorem 1.1.(c)]{voigt84}): 
\begin{theo}\label{theo:spectT0}
Under Assumption \ref{hypO} \textit{3)},  $\mathfrak{S}(\T_{0})=\{\l\in \C\;;\,\mathrm{Re}\l \leq 0\}.$
\end{theo} 
We first deduce from this property of $\T_{0}$ the following:
\begin{theo}\phantomsection \label{theo:spectTH} If $\mathsf{H} \in \mathscr{B}(\lp,\lm)$ satisfies Assumption \ref{hypH} then $i\R \subset \mathfrak{S}(\mathsf{T_{H}}).$\end{theo}
\begin{proof} According to Theorem \ref{theo:spectT0}, it holds 
\begin{equation}\label{eq:epsi-s-T0}
\lim_{\varepsilon\to 0^{+}}\left\|\Rs(\varepsilon+i\eta,\mathsf{T_{0}})\right\|_{\mathscr{B}(\X_{0})}=+\infty, \qquad \forall \eta \in \R.\end{equation}
From Proposition \ref{prop:Meps} and Banach-Steinhaus Theorem \cite[Theorem 2.2, p. 32]{brezis}, for any $\eta \neq 0$,
$$\limsup_{\varepsilon \to 0^{+}}\left\|\Rs(1,\mathsf{M}_{\varepsilon+i\eta}\mathsf{H})\right\|_{\mathscr{B}(\lp)} < \infty.$$
Since, under \eqref{eq:HYn}, the range of $\mathsf{H}$ is included in $\Y^{-}_{1}$ and $\|\mathsf{\Xi_{0}}u\|_{\X_{0}} \leq \|u\|_{\Y^{-}_{1}}$ for any $u \in \Y^{-}_{1}$ (see Lemma  \ref{lem:motau}), one has
$$\sup_{\varepsilon >0,\eta\in \R}\|\mathsf{\Xi_{\varepsilon+i\eta}H}\|_{\mathscr{B}(\lp,\X_{0})} < \infty.$$ Moreover $\sup_{\varepsilon >0,\eta\in \R}\|\mathsf{G}_{\varepsilon+i\eta}\|_{\mathscr{B}(\X_{0},\lp)} < \infty$ we get that, for any $\eta \in \R$, $\eta \neq 0$, it holds:
$$\limsup_{\varepsilon \to 0^{+}}\left\|\mathsf{\Xi_{\varepsilon+i\eta}H}\Rs(1,\mathsf{M}_{\varepsilon+i\eta}\mathsf{H})\mathsf{G}_{{\varepsilon+i\eta}}\right\|_{\mathscr{B}(\X_{0})} < \infty.$$
This, together with \eqref{eq:epsi-s-T0} and \eqref{eq:lambdaTH} proves that, for any $\eta \in \R$, $\eta \neq 0$, it holds
$$\limsup_{\varepsilon \to 0^{+}}\left\|\Rs(\varepsilon+i\eta,\mathsf{T_{H}})\right\|_{\mathscr{B}(\X_{0})}=\infty,$$
whence $i\eta\in \mathfrak{S}(\mathsf{T_{H}})$ for any $\eta \neq 0$. Recalling that $0 \in \mathfrak{S}_{p}(\T_{\H})$ we get the conclusion.\end{proof}

\subsection{{Spectral properties of $\mathsf{M}_{\l}\H$ in the vicinity of $\l=0$.}}
We recall that, being $\mathsf{M}_{0}\H$ stochastic and irreducible, the spectral radius $r_{\sigma}(\mathsf{M}_{0}\H)=1$ is an algebraically simple and isolated eigenvalue of $\mathsf{M}_{0}\H$ and there exists $0 < r < 1$ such that
$$\mathfrak{S}(\mathsf{M}_{0}\H) \setminus \{1\} \subset \{z \in \C\;;\;|z| < r\}$$
and there is a normalised and positive eigenfunction $\varphi_{0}$ such that
$$\mathsf{M}_{0}\H\,\varphi_{0}=1, \qquad \int_{\Gamma_{+}}\varphi_{0}\,\d\mu_{+}=1.$$
Because $\mathsf{M}_{0}\H$ is stochastic, the dual operator $\left(\mathsf{M}_{0}\H\right)^{\star}$ (in $L^{\infty}(\Gamma_{+},\d\mu_{+})$) admits the eigenfunction 
$$\varphi_{0}^{\star}=\mathbf{1}_{\Gamma_{+}}$$ 
associated to the algebraically simple eigenvalue $1.$ The spectral projection of $\mathsf{M}_{0}\H$ associated to the eigenvalue $1$ is then defined as
$$\mathsf{P}(0)=\frac{1}{2i\pi}\oint_{\{|z-1|=r_{0}\}}\Rs(z,\mathsf{M}_{0}\H)\d z$$
where $r_{0} >0$ is chosen so that $\{z \in \C\;;\;|z-1|=r_{0}\} \subset \{z \in \C\;;\;|z| >r\}.$ Such a spectral structure is somehow inherited by $\mathsf{M}_{\l}\H$ for $\l$ small enough:
\begin{propo}\label{prop:eigenMLH}
For any $\l \in \overline{\C}_{+}$  the spectrum of $\mathsf{M}_{\l}\H$ is given by
$$\mathfrak{S}(\mathsf{M}_{\l}\H)=\{0\} \cup \{\nu_{j}(\l)\;;\;j \in \N_{\l} \subset \N\}$$
where, $\N_{\l}$ is a (possibly finite) subset of $\N$ and, for each $j \in \N_{\l}$, $\nu_{j}(\l)$ is an isolated eigenvalue of $\mathsf{M}_{\l}\H$ of finite algebraic multiplicities and $0$ being the only possible accumulation point of the sequence $\{\nu_{j}(\l)\}_{j\in \N_{\l}}$. Moreover,
$$|\nu_{j}(\l)| < 1 \qquad \text{ for any } j \in \N_{\l}, \qquad \l \neq 0.$$
Finally, there exists $\delta_{0} >0$ such that, for any $|\lambda| \leq \delta_{0}$, $\l \in \overline{\C}_{+}$,
$$\mathfrak{S}(\mathsf{M}_{\l}\H) \cap \{z \in \C\;;\;|z-1| <\epsilon\}=\{\nu(\l)\}$$
where $\nu(\l)$ is an algebraically simple eigenvalue of $\mathsf{M}_{\l}\H$ such that
$$\lim_{\l \to 0}\nu(\l)=1$$
and there exist an eigenfunction $\varphi_{\l}$ of $\mathsf{M}_{\l}\H$ and an  eigenfunction $\varphi^{\star}_{\l}$ of $\left(\mathsf{M}_{\l}\H\right)^{\star}$ associated to $\nu(\l)$ such that
$$\lim_{\l\to0}\|\varphi(\l)-\varphi_{0}\|_{\lp}=0, \qquad \lim_{\l\to0}\|\varphi_{\l}^{\star}-\varphi_{0}^{\star}\|_{L^{\infty}(\Gamma_{+},\d\mu_{+})}=0.$$
\end{propo}
\begin{proof} Since $\left|\left(\mathsf{M}_{\l}\H\right)^{2}\right| \leq \left(\mathsf{M}_{0}\H\right)^{2}$, one has that $(\mathsf{M}_{\l}\H)^{2}$ is weakly compact and the structure of $\mathfrak{S}(\mathsf{M}_{\l}\H)$ follows. The fact that all eigenvalues have modulus less than one comes from Proposition \ref{prop:Meps}. This gives the first part of the Proposition. For the second part, because $\mathsf{M}_{\l}\H$ converges in operator norm towards $\mathsf{M}_{0}\H$ as $\l \to 0$ $(\l \in \overline{\C}_{+})$, it follows from general results about the separation of the spectrum \cite[Theorem 3.16, p.212]{kato} that, for $|\l| < \delta_{0}$ small enough, the curve $\{z \in \C\;;\;|z-1|=r_{0}\}$ is separating the spectrum $\mathfrak{S}(\mathsf{M}_{\l}\H)$ into two disjoint parts, say
$$\mathfrak{S}(\mathsf{M}_{\l}\H)=\mathfrak{S}_{\text{in}}(\mathsf{M}_{\l}\H) \cup \mathfrak{S}_{\text{ext}}(\mathsf{M}_{\l}\H)$$
where $\mathfrak{S}_{\text{in}}(\mathsf{M}_{\l}\H) \subset \{z\in \C\;;|z-1| < r_{0}\}$ and $\mathfrak{S}_{\text{ext}}(\mathsf{M}_{\l}\H) \subset \{z\in \C\;;|z-1|>r_{0}\}.$ Moreover, the spectral projection of $\mathsf{M}_{\l}\H$ associated to $\mathfrak{S}_{\text{in}}(\mathsf{M}_{\l}\H)$, defined as,
\begin{equation}\label{eq:Pl}
\mathsf{P}(\l)=\frac{1}{2i\pi}\oint_{\{|z-1|=r_{0}\}}\Rs(z,\mathsf{M}_{\l}\H)\d z,\end{equation}
is converging in operator norm to $\mathsf{P}(0)$ as $\l \to 0$ $(\mathrm{Re}\l \geq 0)$ so that, in particular, up to reduce again $\delta_{0}$,
$$\mathrm{dim}(\mathrm{Range}(\mathsf{P}(\l)))=\mathrm{dim}(\mathrm{Range}(\mathsf{P}(0)))=1, \qquad |\l| < \delta_{0}, \mathrm{Re}\l \geq 0.$$
This shows that
$$\mathfrak{S}_{\text{in}}(\mathsf{M}_{\l}\H)=\mathfrak{S}(\mathsf{M}_{\l}\H) \cap \{z \in \C\;;\;|z-1| < \epsilon\}=\{\nu(\l)\}, \qquad |\l| < \delta_{0}, \mathrm{Re}\l \geq0,$$
where $\nu(\l)$ is a \emph{algebraically simple} eigenvalue of $\mathsf{M}_{\l}\H$. Notice that, clearly
$$\lim_{\l\to 0}\nu(\l)=1 \qquad (\mathrm{Re}\l \geq0).$$
In the same way, defining
$$\mathsf{P}(\l)^{\star}=\frac{1}{2i\pi}\oint_{\{|z-1|=r_{0}\}}\Rs(z,\left(\mathsf{M}_{\l}\H\right)^{\star})\d z, \qquad |\l| \leq \delta_{0}, \mathrm{Re}\l \geq0$$
it holds that
$$\lim_{\l\to0}\|\mathsf{P}(\l)^{\star}-\mathsf{P}(0)^{\star}\|_{\mathscr{B}(L^{\infty}(\Gamma_{+},\d\mu_{+}))}=0.$$
Set 
$$\varphi_{\l}:=\mathsf{P}(\l)\varphi_{0}, \qquad \l \in \C_{+}.$$
Since $\varphi_{\l}$ converges to $\mathsf{P}(0)\varphi_{0}=\varphi_{0} \neq 0$, we get that $\varphi_{\l} \neq 0$ for $\l$ small enough and, since $\nu(\l)$ is algebraically simple, $\varphi(\l)$ is an eigenfunction of $\mathsf{M}_{\l}\H$ for $|\l|$ small enough. In the same way, for $|\l|$ small enough,
$$\varphi^{\star}_{\l}:=\mathsf{P}(\l)^{\star}\varphi^{\star}_{0} \longrightarrow \mathsf{P}(0)^{\star}\varphi^{\star}_{0}=1$$
as $\l \to 0$ and $\varphi^{\star}_{\l}$ is an eigenfunction of $\left(\mathsf{M}_{\l}\H\right)^{\star}$ associated to the eigenvalue $\nu(\l).$
\end{proof}

From now, we define $\delta >0$ small enough, so that the rectangle
$$\mathscr{C}_{\delta}:=\{\l \in \C\;;\;0 \leq \mathrm{Re}\l \leq \delta\,,\,|\mathrm{Im}\l|\leq\delta\} \subset \{\l \in \C\;;\;|\l| < \delta_{0}\}\,,$$
where $\delta_{0}$ is introduced in the previous Proposition \ref{prop:eigenMLH}. 
\begin{lemme}\label{lem:p'0}
The mapping 
$$\l \in \mathscr{C}_{\delta} \longmapsto \mathsf{P}(\l) \in \mathscr{B}(\lp)$$
is differentiable with
$$\mathsf{P}'(0)=-\frac{1}{2i\pi}\oint_{\{|z-1|=r_{0}\}}\Rs(z,\mathsf{M}_{0}\H)(\tau_{-}\mathsf{M}_{0}\H)\Rs(z,\mathsf{M}_{0}\H)\d z.$$
More generally, for any $\eta \in (-\delta,\delta)$,
$$\frac{\d}{\d\eta}\mathsf{P}(i\eta)=-\frac{1}{2i\pi}\oint_{\{|z-1|=r_{0}\}}\Rs(z,\mathsf{M}_{i\eta}\H)\left(\dfrac{\d}{\d\eta}\mathsf{M}_{i\eta}\H\right)\Rs(z,\mathsf{M}_{i\eta}\H)\d z.$$
\end{lemme}
\begin{proof} The only difficulty is to prove the differentiability on the imaginary axis. As soon as $z \notin \mathfrak{S}\left(\mathsf{M}_{\l}\H\right)$ for any $\l \in \mathscr{C}_{\delta}$, one has
$$\dfrac{\d}{\d\l}\Rs(z,\mathsf{M}_{\l}\H)=-\Rs(z,\mathsf{M}_{\l}\H)\left(\frac{\d}{\d\l}\mathsf{M}_{\l}\H\right)\Rs(z,\mathsf{M}_{\l}\H),$$
so that
$$\dfrac{\d}{\d\l}\mathsf{P}(\l)=-\frac{1}{2i\pi}\oint_{\{|z-1|=r_{0}\}}\Rs(z,\mathsf{M}_{\l}\H)\left(\frac{\d}{\d\l}\mathsf{M}_{\l}\H\right)\Rs(z,\mathsf{M}_{\l}\H)\d z\qquad \forall \l \in \mathscr{C}_{\delta}$$
and, since $\lim_{\l \to 0}\frac{\d}{\d\l}\left(\mathsf{M}_{\l}\H\right)=-\lim_{\l\to0}\left(\tau_{-}\mathsf{M}_{\l}\H\right)=-\tau_{-}\mathsf{M}_{0}\H$ we easily get the differentiability in $0.$ The same computations also give
$$\dfrac{\d}{\d\eta}\mathsf{P}(\e+i\eta)=-\frac{1}{2i\pi}\oint_{\{|z-1|=r_{0}\}}\Rs(z,\mathsf{M}_{\e+i\eta})\left(\frac{\d}{\d\eta}\mathsf{M}_{\e+i\eta}\H\right)\Rs(z,\mathsf{M}_{\e+i\eta}\H)\d z, \qquad \forall \eta \in \R\setminus\{0\}.$$
Using now Prop. \ref{prop:derMeis}  which asserts that $\tfrac{\d}{\d\eta}\mathsf{M}_{\e+i\eta}\H$ converges to $\tfrac{\d}{\d\eta}\mathsf{M}_{i\eta}\H$ as $\e\to0^{+}$ uniformly with respect to $\eta$, we deduce the second part of the Lemma. 
\end{proof}

We can complement the above result with the following:
\begin{lemme}\label{lem:deriv} With the notations of Proposition \ref{prop:eigenMLH}, the function $\l \in \mathscr{C}_{\delta} \mapsto \nu(\l) \in \C$ is differentiable with derivative $\nu'(\l)$ such that the limit
$$\nu'(0)=\lim_{\l\to0}\nu'(\l)$$
exists and is given by
$$\nu'(0)=-\int_{\Gamma_{+}}\tau_{-}(x,v)\varphi_{0}(x,v)\d \mu_{+}(x,v)<0.$$
\end{lemme}
\begin{proof} Recall that we introduced in the proof of Proposition \ref{prop:eigenMLH} the functions
$$\varphi_{\l}=\mathsf{P}(\l)\varphi_{0}, \qquad \varphi_{\l}^{\star}=\mathsf{P}(\l)^{\star}\varphi_{0}^{\star}, \qquad \l \in \mathscr{C}_{\delta}$$
which are such that $\lim_{\l \to 0}\varphi_{\l}=\varphi_{0}$ and $\lim_{\l\to0}\varphi_{\l}^{\star}=\varphi^{\star}_{0}=\mathbf{1}_{\Gamma_{+}}$. Introducing the duality bracket $\langle\cdot,\cdot\rangle$ between $\lp$ and its dual $(\lp)^{\star}=L^{\infty}(\Gamma_{+},\d\mu_{+})$, we have in particular
$$\lim_{\l\to0}\langle \varphi_{\l},\varphi_{\l}^{\star}\rangle=\langle \varphi_{0},\varphi_{0}^{\star}\rangle=\int_{\Gamma_{+}}\varphi_{0}\d\mu_{+}=1.$$
Moreover, the mappings $\l \in \mathscr{C}_{\delta}\mapsto \varphi_{\l} \in \lp$ and $\l\in \mathscr{C}_{\delta} \mapsto \varphi_{\l}^{\star} \in (\lp)^{\star}$ are differentiable with 
$$\dfrac{\d}{\d\l}\varphi_{\l}=\dfrac{\d}{\d\l}\mathsf{P}(\l)\varphi_{0}, \qquad \dfrac{\d}{\d\l}\varphi_{\l}^{\star}=\dfrac{\d}{\d\l}\mathsf{P}(\l)^{\star}\varphi_{0}^{\star}.$$
Since 
$$\mathsf{M}_{\l}\H\varphi_{\l}=\nu(\l)\varphi_{\l}$$ 
so that $\langle \mathsf{M}_{\l}\H\varphi_{\l}, \varphi_{\l}^{\star}\rangle=\nu(\l)\langle\varphi_{\l},\varphi_{\l}^{\star}\rangle$,
we deduce first that $\l \in \mathscr{C}_{\delta} \mapsto \nu(\l)$ is differentiable, and  differentiating the above identity yields
$$\dfrac{\d}{\d\l}\left(\mathsf{M}_{\l}\H\varphi_{\l}\right)=\nu'(\l)\varphi_{\l}+\nu(\l)\dfrac{\d}{\d\l}\varphi_{\l}.$$
Computing the derivatives and multiplying with $\varphi_{\l}^{\star}$ and integrating over $\Gamma_{+}$ we get
$$\langle \left(\tfrac{\d}{\d\l}\mathsf{M}_{\l}\H\right)\varphi_{\l} + \mathsf{M}_{\l}\H\tfrac{\d}{\d\l}\varphi_{\l},\varphi_{\l}^{\star}\rangle =
\nu'(\l)\langle\varphi_{\l},\varphi_{\l}^{\star}\rangle + \nu(\l)\langle \tfrac{\d}{\d\l}\varphi_{\l},\varphi_{\l}^{\star}\rangle.$$
Using that $\tfrac{\d}{\d\l}\mathsf{M}_{\l}\H=-\tau_{-}\mathsf{M}_{\l}\H$ whereas 
$$\langle \mathsf{M}_{\l}\H\tfrac{\d}{\d\l}\varphi_{\l},\varphi_{\l}^{\star}\rangle=\langle\tfrac{\d}{\d\l}\varphi_{\l},(\mathsf{M}_{\l}\H)^{\star}\varphi_{\l}^{\star}\rangle=\nu(\l)\langle \tfrac{\d}{\d\l}\varphi_{\l},\varphi_{\l}^{\star}\rangle$$ we obtain
$$-\langle \tau_{-}\mathsf{M}_{\l}\H\varphi_{\l},\varphi_{\l}^{\star}\rangle + \nu(\l)\langle \tfrac{\d}{\d\l}\varphi_{\l},\varphi_{\l}^{\star}\rangle
=\nu'(\l)\langle \langle\varphi_{\l},\varphi_{\l}^{\star}\rangle + \nu(\l)\langle \tfrac{\d}{\d\l}\varphi_{\l},\varphi_{\l}^{\star}\rangle.$$
Thus
$$-\langle \tau_{-}\mathsf{M}_{\l}\H\varphi_{\l},\varphi_{\l}^{\star}\rangle=\nu'(\l)\langle \varphi_{\l},\varphi_{\l}^{\star}\rangle, \qquad \forall \l \in \mathscr{C}_{\delta}.$$
Letting $\l \to 0$, we get that
$$\lim_{\l\to0}\nu'(\l)=-\langle \tau_{-}\mathsf{M}_{0}\H\varphi_{0},\varphi_{0}^{\star}\rangle$$
which is the desired result since $\mathsf{M}_{0}\H\varphi_{0}=\varphi_{0}$. \end{proof}

\subsection{Boundary functions for $\Rs(\l,\T_{0})$ and $\Rs(\l,\T_{\H})$}\label{sec:bounRs} We have now all the tools at hands to define the traces of the functions $\l \in \C_{+} \mapsto \Rs(\l,\T_{0})f \in \X_{0}$ and $\l \in \C_{+}\mapsto \Rs(\l,\T_{\H})f$ along the imaginary axis. We will distinguish between the two cases
$\eta \neq 0$ and $\eta = 0$. For the latter case, the technical difficulty is tremendously increased due to the fact that $1$ lies in the spectrum of $\mathsf{M}_{0}\H$ and  we will resort to the  careful study of the spectral properties of $\mathsf{M}_{\l}\H$ for $\l \in \overline{\C}_{+}$ with $|\l|$ small. To handle this case, we will need  the additional assumption that $f$ has zero mean $\varrho_{f}=0$ (see \eqref{eq:0mean}) 
together with a slight additional integrability $f \in \X_{1}$. Notice that the constraint \eqref{eq:0mean} exactly means that $\mathbb{P}f=0$ (where $\mathbb{P}$ is the spectral projection associated to the (dominant) zero eigenvalue of $\T_{\H}$) or equivalently, $f=\mathbb{(I-P)}f$.

For such a case, the assumption \eqref{eq:0mean} that $f$ has zero mean will be fully exploited. Related to this assumption \eqref{eq:0mean}, we introduce 
$$\X_{k}^{0}:=\{f \in \X_{k}\;;\;\varrho_{f}=0\}, \qquad k \in \N.$$
which is a closed subspace of $\X_{k}$. Notice that, endowed with the $\X_{k}$-norm, $\X_{k}^{0}$ is a Banach space.  
Since 
$$\int_{\Omega \times V}U_{\H}(t)f\d x\otimes \bm{m}(\d v)=\int_{\Omega\times V}f\d x\otimes \bm{m}(\d v), \qquad \forall t \geq0, \quad f \in \X_{0}$$
one has
\begin{equation*}\begin{split}
\int_{\Omega\times V}\Rs(\l,\T_{\H})f\d x\otimes \bm{m}(\d v)&=\int_{\Omega\times V}\left(\int_{0}^{\infty}e^{-\l\,t}U_{\H}(t)f\d t\right)\d x\otimes \bm{m}(\d v)\\
&=\frac{1}{\l}\int_{\Omega\times V}f\d x\otimes \bm{m}(\d v), \qquad \forall \l \in \C_{+}\end{split}\end{equation*}
and therefore the resolvent and all its iterates $\Rs(\l,\T_{\H})^{k}$ leave $\X_{0}^{0}$ invariant ($k \geq 0$).

An important consequence of the spectral result stated in Prop. \ref{prop:eigenMLH} is the following
\begin{lemme}\label{lem:convZk} 
For any $f \in \X_{1}^{0}$ the limit 
$$\lim_{\e\to0^{+}}\Rs(1,\mathsf{M}_{\e+i\eta}\H)\mathsf{G}_{\e+i\eta}f=\mathsf{\Phi}(\eta)f$$
exists in $\X_{0}$ where  
\begin{equation}\label{eq:defPhi}\mathsf{\Phi}(\eta)f :=\begin{cases}
\Rs(1,\mathsf{M}_{i\eta}\H)\mathsf{G}_{i\eta}f \qquad &\text{ if } \eta \neq 0\\
\Rs\left(1,\mathsf{M}_{0}\H\left(\mathsf{I-P}(0)\right)\right)\mathsf{G}_{0}f  -\frac{1}{\nu'(0)}\left[\mathsf{P}'(0)\mathsf{G}_{0}f+\mathsf{P}(0)\mathsf{G}'_{0}f)\right] \qquad &\text{ if } \eta =0\end{cases}
\end{equation}
Moreover, the convergence is \emph{uniform} on any compact subset of $\R$. \medskip
\end{lemme}
\begin{proof} Let  $f \in \X_{1}$ be fixed with $\varrho_{f}=0$. For any $\e >0,$ $\eta \in\R$, one writes
\begin{multline*}
\Rs(1,\mathsf{M}_{\e+i\eta}\H)\mathsf{G}_{\e+i\eta}f=\Rs(1,\mathsf{M}_{\e+i\eta}\H(\mathsf{I-P}(\e+i\eta)))\mathsf{G}_{\e+i\eta}f\\
+\Rs(1,\mathsf{M}_{\e+i\eta}\H\mathsf{P}(\e+i\eta))\mathsf{G}_{\e+i\eta}f\end{multline*}
since $\mathsf{P}(\l)$ commutes with $\mathsf{M}_{\l}\H$. Notice that, with the notations of Proposition \ref{prop:eigenMLH}, 
$$\mathfrak{S}\left(\mathsf{M}_{\e+i\eta}\H\left[\mathsf{I-P}(\e+i\eta)\right]\right) \subset \{z \in \C\;;\;|z| <r\}$$
so that
$$r_{\sigma}\left(\mathsf{M}_{\e+i\eta}\H(\mathsf{I-P}(\e+i\eta))\right) \leq r <1.$$
One has then, for $r < r'< r_{0}$
$$\mathsf{I-P}(\e+i\eta))=\frac{1}{2i\pi}\oint_{\{|z|=r'\}}\Rs(z,\mathsf{M}_{\e+i\eta}\H)\d z$$
so that $\lim_{\e\to0^{+}}\mathsf{I-P}(\e+i\eta))=\mathsf{I-P}(i\eta))$ in $\mathscr{B}(\lp)$ uniformly with respect to $|\eta| < \delta.$ Consequently,
\begin{equation}\label{eq:convI-P}
\lim_{\e\to0^{+}}\sup_{|\eta| \leq \delta}\left\|\Rs(1,\mathsf{M}_{\e+i\eta}\H(\mathsf{I-P}(\e+i\eta)))\mathsf{G}_{i\eta}f-\Rs(1,\mathsf{M}_{i\eta}\H(\mathsf{I-P}(i\eta)))\mathsf{G}_{i\eta}f\right\|_{\lp}=0.\end{equation}
On the other hand, 
$$\Rs(1,\mathsf{M}_{\e+i\eta}\H\mathsf{P}(\e+i\eta))\mathsf{G}_{\e+i\eta}f=\frac{1}{1-\nu(\e+i\eta)}\mathsf{P}(\e+i\eta)\mathsf{G}_{\e+i\eta}f$$
so that
$$\lim_{\e\to0^{+}}\Rs(1,\mathsf{M}_{\e+i\eta}\H\mathsf{P}(\e+i\eta))\mathsf{G}_{\e+i\eta}f=\frac{1}{1-\nu(i\eta)}\mathsf{P}(i\eta)\mathsf{G}_{i\eta}f, \qquad \forall \eta \neq 0$$
where the limit is meant in $\lp$ and we used  the continuity of $\l \in \mathscr{C}_{\delta} \mapsto \nu(\l)$. Whenever $\eta=0$, we have
$$\Rs(1,\mathsf{M}_{\e}\H\mathsf{P}(\e))\mathsf{G}_{\e}f=\frac{1}{1-\nu(\e)}\mathsf{P}(\e)\mathsf{G}_{\e}f=\frac{\e}{1-\nu(\e)}\frac{\mathsf{P}(\e)\mathsf{G}_{\e}f-\mathsf{P}(0)\mathsf{G}_{0}f}{\e}$$
where we used the fact that, since $\varrho_{f}=0$ and $\mathsf{G}_{0}$ is stochastic, one has 
$$\int_{\Gamma_{+}}\mathsf{G}_{0}f\d\mu_{+}=0 \qquad \text{ so } \quad \mathsf{P}(0)\mathsf{G}_{0}f=0.$$ 
As already seen, the derivative $\mathsf{G}'(0)f$ exists since $f \in \X_{1}$ and therefore, by virtue of Lemma \ref{lem:p'0}, 
$$\lim_{\e \to 0^{+}}\frac{\mathsf{P}(\e)\mathsf{G}_{\e}f-\mathsf{P}(0)\mathsf{G}_{0}f}{\e}=\mathsf{P}'(0)\mathsf{G}_{0}f+\mathsf{P}(0)\mathsf{G}'_{0}f.$$
According to Lemma \ref{lem:deriv}, 
$$\lim_{\e\to0^{+}}\frac{\e}{1-\nu(\e)}=-\frac{1}{\nu'(0)} > 0$$
so that
$$\lim_{\e\to0^{+}}\Rs(1,\mathsf{M}_{\e}\H\mathsf{P}(\e))\mathsf{G}_{\e}f=-\frac{1}{\nu'(0)}\left[\mathsf{P}'(0)\mathsf{G}_{0}f+\mathsf{P}(0)\mathsf{G}'_{0}f\right].$$
Finally, we obtain that $\lim_{\e\to0^{+}}\Rs(1,\mathsf{M}_{\e}\H)\mathsf{G}_{\e}f$ exists in $\lp$ and is given by
$$\Rs(1,\mathsf{M}_{0}\H(\mathsf{I-P}(0)))\mathsf{G}_{0}f-\frac{1}{\nu'(0)}\left[\mathsf{P}'(0)\mathsf{G}_{0}f+\mathsf{P}(0)\mathsf{G}'_{0}f\right].$$
This proves the convergence. Let us prove that the convergence is uniform with respect to $|\eta| \leq \delta.$  According to \eqref{eq:convI-P}, we only need to prove that the convergence
$$\lim_{\e\to0^{+}}\Rs(1,\mathsf{M}_{\e+i\eta}\H\mathsf{P}(\e+i\eta))\mathsf{G}_{\e+i\eta}f$$
towards
$$F(\eta)=\begin{cases}
\Rs(1,\mathsf{M}_{i\eta}\H\mathsf{P}(i\eta))\mathsf{G}_{i\eta}f \qquad &\text{ if } \eta \neq 0\\
-\frac{1}{\nu'(0)}\left[\mathsf{P}'(0)\mathsf{G}_{0}f+\mathsf{P}(0)\mathsf{G}'_{0}f)\right] \qquad &\text{ if } \eta =0.\end{cases}
$$
is uniform with respect to $|\eta| < \delta.$ We argue by contradiction, assuming that there exist $c >0$, a sequence $(\e_{n})_{n} \subset (0,\infty)$ converging to $0$ and a sequence $(\eta_{n})_{n} \subset (-\delta,\delta)$ such that
\begin{equation}\label{eq:absurd}
\left\|\Rs(1,\mathsf{M}_{\e_{n}+i\eta_{n}}\H\mathsf{P}(\e_{n}+i\eta_{n}))\mathsf{G}_{\e_{n}+i\eta_{n}}f-F(\eta_{n})\right\|_{\lp} \geq c >0.\end{equation}
Up to considering a subsequence, if necessary, we can assume without loss of generality that $\lim_{n}\eta_{n}=\eta_{0}$ with $|\eta_{0}| \leq \delta.$ First, one sees that then $\eta_{0}=0$ since the convergence of $\Rs(1,\mathsf{M}_{\e+i\eta}\H\mathsf{P}(\e+i\eta)\mathsf{G}_{\e+i\eta}f$ to $F(\eta)$ is actually uniform in any neighbourhood around $\eta_{0} \neq 0$ (see \eqref{eq:convunifr}). Because $\eta_{0}=0$, defining $\lambda_{n}:=\e_{n}+i\eta_{n}$, $n \in \N$, the sequence $(\l_{n})_{n}\subset \mathscr{C}_{\delta}$ is converging to $0$. Now, as before,
$$\Rs(1,\mathsf{M}_{\l_{n}}\H\mathsf{P}_{\l_{n}})\mathsf{G}_{\l_{n}}f=\frac{\l_{n}}{1-\nu(\l_{n})}\frac{\mathsf{P}(\l_{n})\mathsf{G}_{\l_{n}}f-\mathsf{P}(0)\mathsf{G}_{0}f}{\l_{n}}, \qquad n\in \N$$
with
$$\lim_{n\to \infty}\frac{\l_{n}}{1-\nu(\l_{n})}=-\frac{1}{\nu'(0)}, \qquad \lim_{n\to\infty}
\frac{\mathsf{P}(\l_{n})\mathsf{G}_{\l_{n}}f-\mathsf{P}(0)\mathsf{G}_{0}f}{\l_{n}}=\left[\mathsf{P}'(0)\mathsf{G}_{0}f+\mathsf{P}(0)\mathsf{G}'_{0}f\right].$$
Therefore, 
$$\lim_{n\to\infty}\Rs(1,\mathsf{M}_{\l_{n}}\H\mathsf{P}(\l_{n}))\mathsf{G}_{\l_{n}}f=-\frac{1}{\nu'(0)}\left[\mathsf{P}'(0)\mathsf{G}_{0}f+\mathsf{P}(0)\mathsf{G}'_{0}f\right].$$
One also has
$$F(\eta_{n})=\Rs(1,\mathsf{M}_{i\eta_{n}}\H\mathsf{P}({i\eta_{n}}))\mathsf{G}_{i\eta_{n}}f=\frac{i\eta_{n}}{1-\nu(i\eta_{n})}\frac{\mathsf{P}(i\eta_{n})\mathsf{G}_{i\eta_{n}}f-\mathsf{P}(0)\mathsf{G}_{0}f}{i\eta_{n}}, \qquad n\in \N$$
so that $F(i\eta_{n})$ has the same limit $-\frac{1}{\nu'(0)}\left[\mathsf{P}'(0)\mathsf{G}_{0}f+\mathsf{P}(0)\mathsf{G}'_{0}f\right]$ as $n \to \infty$. This contradicts \eqref{eq:absurd}. 
\end{proof}

One deduce from this the following
\begin{propo}\label{propo:limi11} For any $f \in \X_{0}^{0}$, 
$$\mathsf{\Xi}_{\e+i\eta}\H\Rs(1,\mathsf{M}_{\e+i\eta}\H)\mathsf{G}_{i\eta}f \in \X_{k} \qquad \forall k \in \{0,\ldots,N_{\H}\}$$
and, for any $k \leq N_{\H}$, 
\begin{equation}\label{eq:limi}
\lim_{\e\to 0^{+}}\mathsf{\Xi}_{\e+i\eta}\H\Rs(1,\mathsf{M}_{\e+i\eta}\H)\mathsf{G}_{i\eta}f=\mathsf{\Xi}_{i\eta}\H\mathsf{\Phi}(\eta)f\end{equation}
where the convergence is meant in $\mathscr{C}_{0}(\R,\X_{k})$. 
\end{propo}
\begin{proof} The first part of the result is clear from the regularising properties of $\mathsf{\Xi}_{\e+i\eta}\H$. Let us focus on the proof of \eqref{eq:limi} and let us fix $f \in \X_{0}$ satisfying \eqref{eq:0mean} and $k \in \{0,\ldots,N_{\H}\}$. Let $[a,b]$ be a compact subset of $\R$. Using \eqref{eq:MeisHk} in Proposition \ref{propo:convK} we see that
\begin{multline}\label{eq:AB}
\sup_{\eta \in [a,b]}\left\|\mathsf{\Xi}_{\e+i\eta}\H\Rs(1,\mathsf{M}_{\e+i\eta}\H)\mathsf{G}_{i\eta}f-\mathsf{\Xi}_{i\eta}\H\mathsf{\Phi}(\eta)f\right\|_{\X_{k}}\\
\leq \e D\|\H\|_{\mathscr{B}(\lp,\Y_{k+1}^{-})}\sup_{\eta\in [a,b]}\left\|\Rs(1,\mathsf{M}_{\e+i\eta}\H)\mathsf{G}_{i\eta}f\right\|_{\X_{k}}
\\
+\|\Xi_{0}\|_{\mathscr{B}(\Y^{-}_{k+1},\X_{k})}\|\H\|_{\mathscr{B}(\lp,\Y_{k+1}^{-})}\sup_{\eta\in [a,b]}\left\|\Rs(1,\mathsf{M}_{\e+i\eta}\H)\mathsf{G}_{i\eta}f-\mathsf{\Xi}_{i\eta}\H\mathsf{\Phi}(\eta)f\right\|_{\X_{k}}
\end{multline}
where we used that $\sup_{\eta \in \R}\left\|\mathsf{\Xi}_{i\eta}\H\right\|_{\mathscr{B}(\lp,\X_{k})} \leq \|\mathsf{\Xi}_{0}\|_{\mathscr{B}(\Y^{-}_{k+1},\X_{k})}\|\H\|_{\mathscr{B}(\lp,\Y_{k+1}^{-})}.$ From Corollary \ref{lem:convZk} one concludes that $\sup_{\e\in (0,1)}\sup_{\eta \in [a,b]}\left\|\Rs(1,\mathsf{M}_{\e+i\eta}\H)\mathsf{G}_{i\eta}f\right\|_{\X_{k}}$ is finite while the last term in \eqref{eq:AB} converges to $0$ as $\e \to 0^{+}$. This shows that
$$\lim_{\e\to0^{+}}\sup_{\eta\in [a,b]}\left\|\mathsf{\Xi}_{\e+i\eta}\H\Rs(1,\mathsf{M}_{\e+i\eta}\H)\mathsf{G}_{i\eta}f-\mathsf{\Xi}_{i\eta}\H\mathsf{\Phi}(\eta)f\right\|_{\X_{k}}=0.$$
Let us then focus on $|\eta| > R$, $R >0$ arbitrary. We already saw that
$$\sup_{\e\in [0,1]}\sup_{|\eta| >R}\left\|\mathsf{\Xi}_{\e+i\eta}\H\Rs(1,\mathsf{M}_{\e+i\eta}\H)\right\|_{\mathscr{B}(\lp,\X_{k})} < \infty,$$
which, combined with \eqref{eq:RiemLeb} in Prop. \ref{propo:convK}, gives
$$\lim_{|\eta| \to \infty}\sup_{\e \in [0,1]}\left\|\mathsf{\Xi}_{\e+i\eta}\H\Rs(1,\mathsf{M}_{\e+i\eta}\H)\mathsf{G}_{\e+i\eta}f\right\|_{\X_{k}}=0$$
and this implies clearly the result.
\end{proof}
Regarding the behaviour of $\Rs(\e+i\eta,\T_{0})f$, one has
\begin{propo}\label{prop:convRsT0}
For any $f \in \X_{0}$ and $\e >0$, the mapping 
$$\eta \in \R \longmapsto \Rs(\e+i\eta,\T_{0})f \in \X_{0}$$
belongs to $\mathscr{C}_{0}^{k}(\R,\X_{0})$ for any $k \in \N.$ 

Moreover, given $k \in \N$, for any $f \in \X_{k+1}$, 
$$\lim_{\e\to0^{+}}\Rs(\e+i\eta,\T_{0})f$$ exists in $\mathscr{C}_{0}(\R,\X_{k})$. Its limit is denoted $\Rs(i\eta,\T_{0})f.$ 
\end{propo}
\begin{proof} We begin with the first part of the Proposition. Recalling that
$$\Rs(\e+i\eta,\T_{0})f=\int_{0}^{\infty}e^{-i\eta t}\e^{-\e t}U_{0}(t)f\d t$$
with $t \in \R \mapsto e^{-\e t}U_{0}(t)f \in \X_{0}$ Bochner integrable one deduces from Riemann-Lebesgue Theorem that
$$\lim_{|\eta|\to\infty}\left\|\Rs(\e+i\eta,\T_{0})f\right\|_{\X_{0}}=0.$$
Given $k \in \N$, because 
$$\dfrac{\d^{k}}{\d\eta^{k}}\Rs(\e+i\eta,\T_{0})f=(-i)^{k}\int_{0}^{\infty}t^{k}e^{-i\eta t} e^{-\e t}U_{0}(t)f\,\d t$$
the exact same argument shows that 
$$\lim_{|\eta|\to\infty}\left\|\frac{\d^{k}}{\d\eta^{k}}\Rs(\e+i\eta,\T_{0})f\right\|_{\X_{0}}=0$$
which proves that $\eta \in \R \mapsto \Rs(\e+i\eta,\T_{0})f$ belongs to $\mathscr{C}_{0}^{k}(\R,\X_{0}).$ Let us focus now on the limit for $\e \to 0^{+}$. Given $f \in \X_{k+1}$, we deduce from Lemma \ref{lem:UotInt} and the dominated convergence theorem that
$$\lim_{\e\to0^{+}}\Rs(\e+i\eta,\T_{0})f=\lim_{\e\to0^{+}}\int_{0}^{\infty}e^{-i\eta t}e^{-\e t}U_{0}(t)f \d t=\int_{0}^{\infty}e^{-i\eta t}U_{0}(f)f\d t$$
exists in $\X_{k}$. The limit is of course denoted $\Rs(i\eta,\T_{0})$ and one has
$$
\left\|\Rs(\e+i\eta,\T_{0})f-\Rs(i\eta,\T_{0})f\right\|_{\X_{k}} \leq \int_{0}^{\infty}\left|e^{-\e t}-1\right|\,\|U_{0}(t)f\|_{\X_{k}}\d t\, \qquad \forall \eta \in\R,\,\e >0.$$
Thus
\begin{equation}\label{eq:RsT0ieta}
\lim_{\e\to0^{+}}\sup_{\eta\in \R}\left\|\Rs(\e+i\eta,\T_{0})f-\Rs(i\eta,\T_{0})f\right\|_{\X_{k}}=0\end{equation}
still using the fact that $t \mapsto \|U_{0}(t)f\|_{\X_{k}}$ is integrable over $[0,\infty)$ and the dominated convergence theorem.
\end{proof}
\begin{nb}\label{nb:Ck} One deduces from the above Proposition and  Banach-Steinhaus Theorem  \cite[Theorem 2.2, p. 32]{brezis} that 
\begin{equation}\label{eq:Ck}
C_{k}:=\sup\left\{\left\|\Rs(\e+i\eta,\T_{0})\right\|_{\mathscr{B}(\X_{k+1},\X_{k})}\;;\;\e \in (0,1]\;;\;\eta \in \R\right\} <\infty \qquad \forall k \in \N.\end{equation}
\end{nb}
\begin{cor}\label{cor:double} Given $k \in \N$, if 
$$g\::\:\l \in \C_{+} \longmapsto g(\l) \in \X_{k+1}$$
is a continuous mapping such that
$$\lim_{|\eta|\to\infty}\|g(\e+i\eta)\|_{\X_{k+1}}=0 \qquad \forall \e >0$$
while the limit
$$\widetilde{g}(\eta):=\lim_{\e\to0^{+}}g(\e+i\eta)$$
exists in $\X_{k+1}$ uniformly with respect to $\eta \in \R$, then
$$\lim_{\e\to 0^{+}}\Rs(\e+i\eta,\T_{0})g(\e+i\eta)=\Rs(i\eta,\T_{0})\widetilde{g}(\eta)$$
in $\mathscr{C}_{0}(\R,\X_{k}).$
\end{cor}
\begin{proof} Since the convergence \eqref{eq:RsT0ieta} holds for any $f \in\X_{k+1}$, the convergence is of course uniform on any compact subset of $\X_{k+1}$. Since by assumption the mapping
$\widetilde{g}\::\:\eta \in \R \mapsto \widetilde{g}(\eta) \in \X_{0}$
belongs to $\mathscr{C}_{0}(\R,\X_{k+1})$, the set
$$\{\widetilde{g}(\eta)\,;\,\eta \in \R\} \quad \text{ is a compact subset of } \X_{k+1}.$$
Thus
$$\sup_{\e\to0^{+}}\sup_{\eta\in \R}\left\|\Rs(\e+i\eta,\T_{0})\widetilde{g}(\eta)-\Rs(i\eta,\T_{0})\widetilde{g}(\eta)\right\|_{\X_{k}}=0$$
Now, noticing that, for any $\e \in (0,1]$, it holds
\begin{multline*}
\left\|\Rs(\e+i\eta,\T_{0})g(\e+i\eta)-\Rs(i\eta,\T_{0})\widetilde{g}(\eta)\right\|_{\X_{k}} \leq
\left\|\Rs(\e+i\eta,\T_{0})\left(g(\e+i\eta)-\widetilde{g}(\eta)\right)\right\|_{\X_{k}}\\
+\left\|\Rs(\e+i\eta,\T_{0})\widetilde{g}(\eta)-\Rs(i\eta,\T_{0})\widetilde{g}(\eta)\right\|_{\X_{k}}\\
C_{k}\left\|g(\e+i\eta)-\widetilde{g}(\eta)\right\|_{\X_{k+1}}+\left\|\Rs(\e+i\eta,\T_{0})\widetilde{g}(\eta)-\Rs(i\eta,\T_{0})\widetilde{g}(\eta)\right\|_{\X_{k}}
\end{multline*}
where $C_{k}$ is defined in \eqref{eq:Ck}, we deduce easily that
$$\lim_{\e\to0^{+}}\sup_{\eta\in\R}\left\|\Rs(\e+i\eta,\T_{0})g(\e+i\eta)-\Rs(i\eta,\T_{0})\widetilde{g}(\eta)\right\|_{\X_{k}}=0$$
which proves the result.\end{proof}

The convergence established in Prop. \ref{prop:convRsT0} extends to derivatives of $\Rs(\e+i\eta,\T_{0})f$
\begin{lemme}\label{lem:convDerRsT0} Given $k \in \N$ and $f \in \X_{k+1}$. It holds
\begin{equation*}
\lim_{\e\to0^{+}}\sup_{\eta \in \R}\left\|\dfrac{\d^{k}}{\d\eta^{k}}\Rs(\e+i\eta,\T_{0})f-\dfrac{\d^{k}}{\d\eta^{k}}\Rs(i\eta,\T_{0})f\right\|_{\X_{0}}=0.
\end{equation*}
Consequently, the mapping
$$\eta \in \R \longmapsto \Rs(i\eta,\T_{0})f \in \X_{0}$$
belongs to $\mathscr{C}_{0}^{k}(\R,\X_{0}).$
\end{lemme}
\begin{proof} As already established
$$\dfrac{\d^{k}}{\d\eta^{k}}\Rs(\e+i\eta,\T_{0})f=(-i)^{k}\int_{0}^{\infty}e^{-i\eta t}e^{-\e t}U_{0}(t)f\d t$$
and, since 
$$\Rs(i\eta,\T_{0})f=\int_{0}^{\infty}e^{-i\eta t}U_{0}(t)f\d t$$
one sees easily that, if $f \in \X_{k+1}$, 
$$\frac{\d^{k}}{\d\eta^{k}}\Rs(i\eta,\T_{0})f=(-i)^{k}\int_{0}^{\infty}e^{-i\eta t}t^{k}U_{0}(t)f\d t$$
is well-defined in $\X_{0}$ thanks to Lemma \ref{lem:UotInt}. One concludes then exactly as in Prop. \ref{prop:convRsT0}. 
\end{proof}

Recalling that
$$\Rs(\e+i\eta,\T_{\H})=\Rs(\e+i\eta,\T_{0})+\mathsf{\Xi}_{\e+i\eta}\H\Rs(1,\mathsf{M}_{\e+i\eta}\H)\mathsf{G}_{\e+i\eta}$$
the previous results allows to prove the following
\begin{propo}\label{theo:existtraceTH} For any $k \in \N$, $k \leq N_{\H}$ and any $f \in \X_{k+1}^{0}$ the limit
$$\lim_{\e\to0^{+}}\Rs(\e+i\eta,\T_{\H})f$$
exists in $\mathscr{C}_{0}(\R,\X_{k})$. We denote by $\Rs(i\eta,\T_{\H})f$ the limit.
\end{propo}
\begin{proof} Using the above splitting $\Rs(\e+i\eta,\T_{\H})=\Rs(\e+i\eta,\T_{0})+\mathsf{\Xi}_{\e+i\eta}\H\Rs(1,\mathsf{M}_{\e+i\eta}\H)\mathsf{G}_{\e+i\eta}$, Prop. \ref{prop:convRsT0} shows the convergence of the first term and \eqref{eq:limi} gives the one of the second one. This proves the result with 
$$\Rs(i\eta,\T_{\H})f=\Rs(i\eta,\T_{0})f+\mathsf{\Xi}_{i\eta}\H\mathsf{\Phi}(\eta)f$$
and of course the mapping $\eta \in \R \mapsto \Rs(i\eta,\T_{\H})f$ belongs to $\mathscr{C}_{0}(\R,\X_{k})$.
\end{proof}
\begin{nb} Notice that 
$$\Rs(\e+i\eta,\T_{\H})\in \mathscr{B}(\X_{k+1},\X_{k}) \qquad \forall \e >0, \eta \in\R$$
and one sees from Prop. \ref{theo:existtraceTH} that $\Rs(i\eta,\T_{\H}) \in \mathscr{B}(\X_{k+1}^{0},\X_{k})$ with, thanks to Banach-Steinhaus Theorem,
$$\sup\left\{\|\Rs(\e+i\eta,\T_{\H})\|_{\mathscr{B}(\X_{k+1}^{0},\X_{k})}\;;\;\e \in (0,1]\;;\;\eta \in \R\right\}:=\overline{C}_{k} < \infty$$
for any $k \in \left\{0,\ldots,N_{\H}\right\}.$ Notice also that one cannot hope to go beyond the threshold value $k=N_{\H}$ since $\mathsf{\Xi}_{i\eta}\H$ maps $\lp$ in $\X_{N_{\H}}$ but not in $\X_{N_{\H}+1}.$\end{nb} 
As it was the case for $\Rs(\e+i\eta,\T_{0})$, the above convergence extends to derivatives. The crucial observation is the following general property of the resolvent
$$\dfrac{\d^{k}}{\d \l^{k}}\Rs(\l,\T_{\H}) =(-1)^{k}k!\Rs(\l,\T_{\H})^{k+1}, \qquad \l \in \C_{+}.$$
One has then the key technical result
\begin{propo}\label{lem:UPS}  Assume that $\H$ satisfies assumptions \ref{hypH} and let $f \in \X_{N_{\H}+1}^{0}$. Then,
$$\lim_{\e\to 0^{+}}\left[\Rs(\e+i\eta,\T_{\H})\right]^{k}f:=\left[\Rs(i\eta,\T_{\H}\right]^{k}f \qquad \text{ in } \quad \mathscr{C}_{0}\left(\R,\X_{N_{\H}-(k-1)}\right)$$
holds for any $k \in \{0,\ldots,N_{\H}+1\}.$ \end{propo}
\begin{proof} The proof is made by induction over $k \in \N$, $k \leq N_{\H}+1$ For $k=1$, the result holds true by Proposition \ref{theo:existtraceTH}. Let $k \in \N$, $k \leq N_{\H}$ and assume the result to be true for any $j\in \{1,\ldots,k\}.$ Let us prove the result is still true for $k+1$. One recalls that
$$\Rs(\e+i\eta,\T_{\H})=\Rs(\e+i\eta,\T_{0})+ {\Upsilon}_{0}(\e+i\eta)$$
where $\Upsilon_{0}(\l)=\mathsf{\Xi}_{\l}\H\Rs(1,\mathsf{M}_{\l}\H)\mathsf{G}_{\l}$. One has
\begin{multline}\label{eq:splitTH}
\left[\Rs(\e+i\eta,\T_{\H})\right]^{k+1}f=\Rs(\e+i\eta,\T_{0})\left[\Rs(\e+i\eta,\T_{\H})\right]^{k}f\\
+\Upsilon_{0}(\e+i\eta)\left[\Rs(\e+i\eta,\T_{\H})\right]^{k}f.\end{multline}
On the one hand, our induction hypothesis implies that
$\lim_{\e\to 0^{+}}\left[\Rs(\e+i\eta,\T_{\H}\right]^{k}f=\left[\Rs(i\eta,\T_{\H})\right]^{k}f$ in $\mathscr{C}_{0}(\R,\X_{N_{\H}+1-k})$. Thanks to Corollary \ref{cor:double}, we deduce that
\begin{equation}\label{eq:1sthand}\lim_{\e\to0^{+}}\Rs(\e+i\eta,\T_{0})\left[\Rs(\e+i\eta,\T_{\H})\right]^{k}f=\Rs(i\eta,\T_{0})\left[\Rs(i\eta,\T_{\H})\right]^{k}f \qquad \text{ in } \quad \mathscr{C}_{0}(\R,\X_{N_{\H}-k}).\end{equation}
On the other hand, from the induction hypothesis
$$\lim_{|\eta|\to\infty}\|\left[\Rs(i\eta,\T_{\H}\right]^{k}f\|_{\X_{N_{\H}+1-k}}=0$$
which implies in particular that
$$\lim_{|\eta|\to\infty}\|\left[\Rs(i\eta,\T_{\H}\right]^{k}f\|_{\X_{0}}=0.$$
Since moreover $\left[\Rs(i\eta,\T_{\H}\right]^{k}f \in \X_{0}^{0}$ and
$$\sup_{\eta \in \R}\left\|\mathsf{\Xi}_{i\eta}\H \mathsf{\Phi}(\eta)\right\|_{\mathscr{B}(\X_{0}^{0},\X_{N_{\H}-k})} < \infty$$
where $\mathsf{\Phi}(\eta)f$ is defined in Lemma \ref{lem:convZk},
one sees that
$$\lim_{|\eta|\to\infty}\left\|\mathsf{\Xi}_{i\eta}\H\mathsf{\Phi}(\eta)\left[\Rs(i\eta,\T_{\H}\right]^{k}f\right\|_{\X_{N_{\H}-k}}=0.$$
Now, to prove that 
$$\lim_{\e\to0}\Upsilon_{0}(\e+i\eta)\left[\Rs(\e+i\eta,\T_{\H})\right]^{k}f=\mathsf{\Xi}_{i\eta}\H\mathsf{\Phi}(\eta)\left[\Rs(i\eta,\T_{\H})\right]^{k}f$$ in $\mathscr{C}_{0}(\R,\X_{N_{\H}-k})$, one argues along the exact same lines as those used to prove \eqref{eq:1sthand} since, according to \eqref{eq:limi},
$$\lim_{\e\to0}\sup_{\eta \in \R}\|\Upsilon_{0}(\e+i\eta)g-\mathsf{\Xi}_{i\eta}\H\mathsf{\Phi}(\eta)g\|_{\X_{N_{\H}-k}}=0, \qquad \forall g \in \X_{0}^{0}$$
one can resume the argument of Corollary \ref{cor:double} to deduce that, for any continuous mapping $g\::\:\l \in \C_{+} \longmapsto g(\l) \in \X_{0}^{0}$ such that
$$\lim_{|\eta|\to\infty}\|g(\e+i\eta)\|_{\X_{0}}=0 \qquad \forall \e >0 \qquad \text{ and } \qquad \widetilde{g}(\eta):=\lim_{\e\to0^{+}}g(\e+i\eta)$$
exists in $\X_{0}^{0}$ uniformly with respect to $\eta \in \R$, it holds
$$\lim_{\e\to0}\sup_{\eta\in \R}\|\Upsilon_{0}(\e+i\eta)g(\e+i\eta)-\mathsf{\Xi}_{i\eta}\H\mathsf{\Phi}(\eta)\widetilde{g}(\eta)\|_{\X_{N_{\H}-k}}=0.$$
Applying this with $g(\e+i\eta)=\left[\Rs(\e+i\eta,\T_{\H})\right]^{k}f$ and $\widetilde{g}(i\eta)=\left[\Rs(i\eta,\T_{\H})\right]^{k}f$ we deduce that
$$\lim_{\e\to0}\Upsilon_{0}(\e+i\eta)\left[\Rs(\e+i\eta,\T_{\H})\right]^{k}f=\mathsf{\Xi}_{i\eta}\H\mathsf{\Phi}(\eta)\left[\Rs(i\eta,\T_{\H}\right]^{k}f \qquad \text { in } \mathscr{C}_{0}(\R,\X_{N_{\H}-k})$$
which, combined with \eqref{eq:1sthand} and \eqref{eq:splitTH} achieves the induction. \end{proof}

A fundamental consequence of the previous Proposition is 
\begin{cor}\label{cor:reguRf} Assume that $\H$ satisfies assumptions \ref{hypH}. For any $f \in \X_{N_{\H}+1}^{0},$ the mapping 
$$\eta \in \R \longmapsto \Rs(i\eta,\T_{\H})f $$
defined in Proposition \ref{theo:existtraceTH} belongs to $\mathscr{C}_{0}^{N_{\H}}(\R,\X_{0})$ and the convergence
$$\lim_{\e\to0^{+}}\Rs(\e+i\eta,\T_{\H})f=\Rs(i\eta,\T_{\H}f)$$
holds in $\mathscr{C}_{0}^{N_{\H}}(\R,\X_{0}).$\end{cor}
\begin{proof} Let $f \in \X_{N_{\H}+1}^{0}$ be fixed. Since, for any $k \in \{1,\ldots,N_{\H}\}$ and any $\e >0,\eta\in\R$,
$$\dfrac{\d^{k}}{\d\eta^{k}}\Rs(\e+i\eta,\T_{\H})f=(-i)^{k} k!\left[\Rs(\e+i\eta,\T_{\H})\right]^{k+1}f$$
the result follows directly  from Proposition \ref{lem:UPS} where the derivatives of $\Rs(i\eta,\T_{\H})f$ are defined by
$$\dfrac{\d^{k}}{\d\eta^{k}}\Rs(i\eta,\T_{\H})f=(-i)^{k} k!\left[\Rs(i\eta,\T_{\H})\right]^{k+1}f$$
for any $k \in \{0,\ldots,\N_{\H}\}.$
\end{proof}

\section{Definition and regularity of the boundary function of $\Upsilon_{n}(\l)$} \label{sec:near0}
 
This section is devoted to the construction of the trace along the imaginary axis, that is when $\lambda=i\eta$, $\eta \in \R$,  of  
$$\Upsilon_{n}(\l)f=\mathsf{\Xi}_{\l}\H\left(\mathsf{M}_{\l}\H\right)^{n}\Rs(1,\mathsf{M}_{\l}\H)\mathsf{G}_{\l}f, \qquad \l \in \mathbb{C}, \mathrm{Re}\l \geq0,\: n \in \N$$
for a suitable class of function $f.$ The crucial observation here is the following alternative representation of $\Upsilon_{n}(\l)f$ which can also be written as
\begin{equation}\label{eq:dif}
\Upsilon_{n}(\l)=\Rs(\l,\T_{\H})-\Rs(\l,\T_{0})-\sum_{k=0}^{n}\mathsf{\Xi}_{\l}\H\left(\mathsf{M}_{\l}\H\right)^{k}\mathsf{G}_{\l}, \qquad \l \in \C_{+}\end{equation}
as can easily be seen from the fact that
$$\Rs(\l,\T_{\H})=\Rs(\l,\T_{0})+\sum_{k=0}^{\infty}\mathsf{\Xi}_{\l}\H\left(\mathsf{M}_{\l}\H\right)^{k}\mathsf{G}_{\l}.$$ 
We already investigated the existence and regularity of the traces on the imaginary axis of the first two terms in \eqref{eq:dif} so we just need to focus on the properties of the \emph{finite sum} 
\begin{equation}\label{def:snl}
{s}_{n}(\l):=\sum_{p=0}^{n}\mathsf{\Xi}_{\l}\H\left(\mathsf{M}_{\l}\H\right)^{p}\mathsf{G}_{\l}, \qquad \l \in \overline{\C}_{+}.\end{equation}
The differentiability of the various involved (single) operators is summarized in the following whose proof is easy and postponed to Appendix \ref{app:technA}. In the sequel, the notion of differentiability of functions $h\::\:\l \in \overline{\C}_{+} \mapsto h(\l) \in Y$ (where $Y$ is a given Banach space) is the usual one but we have to emphasize the fact that  limits are always meant in $\overline{\C}_{+}$ \footnote{This means for instance that, if $\l_{0} \in \C_{+}$, $h$ is differentiable means that it is holomorphic in a neighborhoud of $\l_{0}$ whereas, for $\l_{0}=i\eta_{0}$, $\eta_{0} \in \R$, the differentiability at $\l_{0}$ of $h$ at means that there exists $h'(\l_{0}) \in Y$ such that 
$$\underset{\l \in \overline{\C}_{+}}{\lim_{\l \to \l_{0}}}\left\|\frac{h(\l)-h(\l_{0})}{\l-\l_{0}}-h'(\l_{0})\right\|_{Y}=0$$
where $\|\cdot\|_{Y}$ is the norm on $Y$.}
\begin{propo} \label{prop:DerivG} 
We have the following general differentiability properties:
\begin{enumerate}
\item For any $k \geq 1$ and any $f \in \X_{k}$, the limit
$$\lim_{\l \to 0}\dfrac{\d^{k}}{\d \l^{k}}\mathsf{G}_{\l}f=(-1)^{k}\mathsf{G}_{0}(t_{+}^{k}f)$$
exists in $\lp$ and 
\begin{equation}\label{lem:G1}\sup_{\lambda \in\overline{\C}_{+}}\left\|\dfrac{\d^{j}}{\d \lambda^{j}}\mathsf{\mathsf{G_{\lambda}}}f\right\|_{\lp} \leq  D^{j}\|f\|_{\X_{j}} \leq D^{j}\|f\|_{\X_{k}}, \qquad \forall j \in \{0,\ldots,k\}.\end{equation}
\item For any 
$k \leq N_{\H}$ 
\begin{equation}\label{prop:derMeis}\left\|\frac{\d^{k}}{\d \eta^{k}}\mathsf{M}_{\varepsilon+i\eta}\H-\frac{\d^{k}}{\d \eta^{k}}\mathsf{M}_{i\eta}\H\right\|_{\mathscr{B}(\lp)} \leq \varepsilon\,D\,\|\H\|_{\mathscr{B}(\lp,\Y_{k+1}^{-})} \qquad \forall \eta \in \R, \qquad \varepsilon >0.\end{equation}
\end{enumerate}
\end{propo}

\begin{nb}\label{nb:deriGl}
Notice that the above expression of $\dfrac{\d^{j}}{\d\lambda^{j}}\mathsf{G}_{\lambda}f$ shows that, for $f \in \X_{j}$,
\begin{equation}\label{eq:derivGj}
\dfrac{\d^{j}}{\d \lambda^{j}}\mathsf{G}_{\lambda}f=(-1)^{j}\mathsf{G}_{\l}\left[t_{+}^{j}f\right]\end{equation}
where $t_{+}^{j}$ denotes the multiplication operator by the function $(x,v) \mapsto t_{+}^{j}(x,v)$. \end{nb}

As a consequence
\begin{cor}\label{cor:MHclass}
For any $k \in \N$ such that $\H \in \mathscr{B}(\lp,\Y_{k+1}^{-})$  (i.e. $k \leq N_{\H}$), the  function
$$\l=\e+i\eta  \in \C_{+} \longmapsto \frac{\d^{j}}{\d \eta^{j}}\mathsf{M}_{\l}\H \in \mathscr{B}(\lp), \qquad 0 \leq j \leq k$$
can be extended to a continuous functions on $\overline{\C}_{+}.$ In particular, the  mapping
\begin{equation}\label{eq:ClassMH}
\eta \in \R \longmapsto \mathsf{M}_{i\eta}\H \in \mathscr{B}(\lp) \quad \text{ is of class } \mathscr{C}^{N_{\H}}\end{equation}
with bounded derivatives up to order $N_{\H}$. In the same way, the  function
$$\l=\e+i\eta  \in \C_{+} \longmapsto \frac{\d^{j}}{\d \eta^{j}}\mathsf{\Xi}_{\l}\H \in \mathscr{B}(\lp,\X_{0}), \qquad 0 \leq j \leq k$$
can be extended to a continuous functions on $\overline{\C}_{+}$ and the  mapping
\begin{equation}\label{eq:CNXi}
\eta \in \R \longmapsto \mathsf{\Xi}_{i\eta}\H \in \mathscr{B}(\lp) \qquad \text{ is of class } \mathscr{C}^{N_{\H}}\end{equation}
with bounded derivatives up to order $N_{\H}.$

Moreover, for any $\varphi\in \Y_{1}^{-}$, the limit $ {\lim_{\l\to 0}}\dfrac{\d}{\d\l}\mathsf{M}_{\l}\varphi$
exists in $\lp$. In particular, 
$$\lim_{\l\to 0}\dfrac{\d}{\d\l}\mathsf{M}_{\l}\H=-\tau_{-}\mathsf{M}_{0}\H$$
exists in $\mathscr{B}(\lp)$ where, as before, we use the same symbol $\tau_{-}$ for the measurable function $\tau_{-}(\cdot,\cdot)$ and the multiplication operator by that function and we recall that limits are meant in $ {\overline{\C}_{+}}$, i.e. $\underset{\l\to 0}{\lim}\{\ldots\}=\underset{\underset{\l \in \overline{\C}_{+}}{\l\to0}}{\lim}\{\ldots\}$.
\end{cor}
We refer to Appendix \ref{app:technA} for the full proof of Proposition \ref{prop:DerivG}  as well as that of Corollary \ref{cor:MHclass}. All the above results allow to prove the regularity the \emph{finite} sum $s_{n}(\l)$ defined by \eqref{def:snl}
\begin{propo}\label{prop:snl} For any $f \in \X_{N_{\H}+1}$, the mapping 
$$\eta \in \R \longmapsto s_{n}(\e+i\eta)f \in \X_{0} \qquad \text{ belongs to }  \mathscr{C}^{N_{\H}}_{0}(\R,\X_{0})$$ 
for any $\e \geq 0$ with
\begin{equation}\label{eq:snEps}
\lim_{|\eta|\to\infty}\sup_{\e \in (0,1]}\left\|\dfrac{\d^{k}}{\d\eta^{k}}\,s_{n}(\e+i\eta)f\right\|_{\X_{0}}=0\end{equation}
for any $k \in \{0,\ldots,N_{\H}\}.$ In particular
$$\lim_{\e\to0^{+}}\sup_{\eta \in \R}\left\|s_{n}(\e+i\eta)f-s_{n}(i\eta)f\right\|_{\X_{0}}=0.$$
\end{propo}
\begin{proof} Let $n \in \N$ be fixed. For simplicity of notations, for any $p \in \{0,\ldots,n\}$, we define 
$$\mathsf{L}_{p}(\l)=\left(\mathsf{M}_{\l}\H\right)^{p}, \qquad \l \in \C_{+}$$ 
and denotes its derivatives of order $j$ by $\mathsf{L}_{p}^{(j)}(\l)$. Computing derivatives with Leibniz rule we get, for any $k \in \N$
\begin{multline*}
\dfrac{\d^{k}}{\d\l^{k}}s_{n}(\l)f=\sum_{p=0}^{n}\sum_{\ell=0}^{k}\left(\begin{array}{c}k \\\ell\end{array}\right)\left(\dfrac{\d^{k-\ell}}{\d\l^{k-\ell}}\mathsf{\Xi}_{\l}\H\right)\,\dfrac{\d^{\ell}}{\d\l^{\ell}}\left[\mathsf{L}_{p}(\l)\mathsf{G}_{\l}f\right]\\
=\sum_{p=0}^{n}\sum_{\ell=0}^{k}\sum_{j=0}^{\ell}\left(\begin{array}{c}\ell \\j\end{array}\right)\left(\begin{array}{c}k \\\ell\end{array}\right)\left(\dfrac{\d^{k-\ell}}{\d\l^{k-\ell}}\mathsf{\Xi}_{\l}\H\right)\,\mathsf{L}_{p}^{(j)}(\l)\dfrac{\d^{\ell-j}}{\d\l^{\ell-j}}\left[\mathsf{G}_{\l}f\right]\end{multline*}
Now, as  observed (see the proofs of Proposition \ref{prop:DerivG}  and Corollary \ref{cor:MHclass}), for any $j$
$$\dfrac{\d^{j}}{\d\l^{j}}\mathsf{\Xi}_{\l}=(-1)^{j}t_{-}^{j}\mathsf{\Xi}_{\l}, \qquad \dfrac{\d^{j}}{\d\l^{j}}\mathsf{G}_{\l}f=(-1)^{j}\mathsf{G}_{\l}\left(t_{+}^{j}f\right)$$
where $t_{\pm}^{j}$ denote here the multiplication operator by $t_{\pm}(x,v)^{j}$. 
Therefore
\begin{equation}\label{eq:derivSn}
\dfrac{\d^{k}}{\d\l^{k}}s_{n}(\l)f=\sum_{p=0}^{n}\sum_{\ell=0}^{k}\sum_{j=0}^{\ell}\left(\begin{array}{c}\ell \\j\end{array}\right)\left(\begin{array}{c}k \\\ell\end{array}\right)(-1)^{k-j}t_{-}^{k-\ell}\mathsf{\Xi}_{\l}\H\mathsf{L}_{p}^{(j)}(\l)\mathsf{G}_{\l}\left[t_{+}^{\ell-j}f\right].\end{equation}
If $f \in \X_{k}$, then $t_{+}^{\ell-j}f \in \X_{0}$ and $\mathsf{L}_{p}^{(j)}(\l)\mathsf{G}_{\l}\left[t_{+}^{\ell-j}f\right] \in \lm$ so that (see Eqs. \eqref{eq:ClassMH}--\eqref{eq:CNXi}) 
$$\H\mathsf{L}_{p}^{(j)}(\l)\mathsf{G}_{\l}\left[t_{+}^{\ell-j}f\right] \in \Y_{N_{H}+1}^{-} \quad \text{ and } \quad  \mathsf{\Xi}_{\l}\H\mathsf{L}_{p}^{(j)}(\l)\mathsf{G}_{\l}\left[t_{+}^{\ell-j}f\right] \in \X_{N_{\H}}.$$
Then, if $k \leq N_{\H}$, 
$$t_{-}^{\ell-j}\mathsf{\Xi}_{\l}\H\mathsf{L}_{p}^{(j)}(\l)\mathsf{G}_{\l}\left[t_{+}^{\ell-j}f\right] \in \X_{0}$$
for all $\ell \in \{0,\ldots,k\}, j \in \{0,\ldots,\ell\}.$ This easily proves that the mapping
$$\l \in \overline{\C}_{+} \longmapsto s_{n}(\l)f \in \X_{0}$$
is of class $\mathscr{C}^{N_{\H}}$ with 
$$\sup_{\l \in \C_{+}}\left\|\frac{\d^{k}}{\d\l^{k}}s_{n}(\l)f\right\|_{\X_{0}} \leq C_{k}\|f\|_{\X_{N_{\H}+1}}$$
for some positive $C_{k} >0$ depending only on $k \in \{0,\ldots,N_{\H}\}.$ Let us now prove \eqref{eq:snEps} which will also prove the fact that the mapping $\eta \mapsto s_{n}(\e+i\eta)f$ belongs to $\mathscr{C}_{0}^{N_{\H}}(\R,\X_{0})$. The proof of \eqref{eq:snEps} follows exactly the lines of Proposition \ref{propo:limi11}. Indeed, for any $k \leq N_{\H}$,  $p \in \{0,\ldots,n\}$, $\ell \in \{0,\ldots,k\}$ and $j \in \{0,\ldots,\ell\}$, one easily see that, for any $R >0$,
\begin{multline*}
\sup_{\e\in [0,1]}\sup_{|\eta| >R}\left\|t_{-}^{k-\ell}\mathsf{\Xi}_{\e+i\eta}\H\mathsf{L}_{p}^{(j)}(\e+i\eta)\right\|_{\mathscr{B}(\lp,\X_{0})} \leq
\sup_{\e \in [0,1]}\sup_{|\eta|>R}\left\|\mathsf{\Xi}_{\e+i\eta}\H\mathsf{L}_{p}^{(j)}(\e+i\eta)\right\|_{\mathscr{B}(\lp,\X_{k})}\\
\leq \|\mathsf{\Xi}_{0}\|_{\mathscr{B}(\Y^{-}_{k+1},\X_{k})}\|\H\|_{\mathscr{B}(\lp,\Y^{-}_{k+1})}\sup_{|\eta| >R}\|\mathsf{L}_{p}^{(j)}(i\eta)\|_{\mathscr{B}(\lp,\X_{k})}.
\end{multline*}
Now, one can prove that that there is $C_{k} >0$ such that
$$\sup_{|\eta| >R}\|\mathsf{L}_{p}^{(j)}(i\eta)\|_{\mathscr{B}(\lp,\X_{k})} \leq C_{k} < \infty$$
(see Lemma \ref{lem:estJ} in Appendix \ref{app:technA}) from which
$$\sup_{\e\in [0,1]}\sup_{|\eta| >R}\left\|t_{-}^{k-\ell}\mathsf{\Xi}_{\e+i\eta}\H\mathsf{L}_{p}^{(j)}(\e+i\eta)\right\|_{\mathscr{B}(\lp,\X_{0})} < \infty.$$
Combining this with \eqref{eq:RiemLeb} in Prop. \ref{propo:convK} and the representation \eqref{eq:derivSn} proves \eqref{eq:snEps}. The fact that $s_{n}(\e+i\eta)f$ converges to $s_{n}(i\eta)f$ in $\mathscr{C}_{0}^{N_{\H}}(\R,\X_{0})$ is then deduced from \eqref{eq:derivSn} and the limits established in Proposition \ref{propo:convK}.\end{proof}

We have all the tools to prove the first part of Theorem \ref{theo:introtr} in the Introduction.
\begin{theo}\label{theo:existtrace} Let $f \in \X_{N_{\H}+1}$ be such that
\begin{equation}\label{eq:0mean}
\varrho_{f}=\int_{\Omega\times V}f(x,v)\d x \otimes \bm{m}(\d v)=0.\end{equation}
Then, for any $n\geq0$ the limit
$$\lim_{\e\to0^{+}}\Upsilon_{n}(\e+i\eta)f,$$
exists in $\mathscr{C}_{0}^{N_{\H}}(\R,\X_{0})$. Its limit is denoted $\mathsf{\Psi}_{n}(\eta)f$.\medskip
\end{theo}
\begin{proof}
 We know from Corollary \ref{cor:reguRf} that
$$\lim_{\e\to0^{+}}\Rs(\e+i\eta,\T_{\H})f=\Rs(i\eta,\T_{\H})f$$
holds in $\mathscr{C}_{0}^{N_{\H}}(\R,\X_{0}).$ In the same way, Lemma \ref{lem:convDerRsT0} shows that 
$$\lim_{\e\to0^{+}}\Rs(\e+i\eta,\T_{0})f=\Rs(i\eta,\T_{0})f$$
holds in $\mathscr{C}_{0}^{N_{\H}}(\R,\X_{0})$. Since one sees easily from Prop. \ref{prop:snl} that
$$\lim_{\e\to0^{+}}s_{n}(\e+i\eta)f=s_{n}(i\eta)f \qquad \text{ in } \quad \mathscr{C}_{0}^{N_{\H}}(\R,\X_{0})$$
we get the result from the representation \eqref{eq:dif}.
\end{proof}
\begin{nb} Notice that, using the representation 
$$\Upsilon_{n}(\e+i\eta)f=\mathsf{\Xi}_{\e+i\eta}\H\left(\mathsf{M}_{\e+i\eta}\H\right)^{n}\Rs(1,\mathsf{M}_{\e+i\eta}\H)\mathsf{G}_{\e+i\eta}, \qquad \e >0, \qquad \eta \in \R$$
together with Lemma \ref{lem:convZk}, it is easy to see, by uniqueness of the limit, that
\begin{equation}\label{eq:Psitrace}
\mathsf{\Psi}_{n}(\eta)f=\mathsf{\Xi}_{i\eta}\H\left(\mathsf{M}_{i\eta}\H\right)^{n}\mathsf{\Phi}(\eta)f, \qquad \forall \eta \in \R.
\end{equation}
where $\mathsf{\Phi}(\eta)f$ is defined in \eqref{eq:defPhi}.\end{nb}

In the following, we show also that, if $n$ is large enough, the boundary function is also integrable.
\begin{lemme}\label{prop:reguPsif}  Assume that $n \geq 2^{N_{\H}}\mathsf{p}$ (with $\mathsf{p}$ defined in \eqref{eq:power}) and $f \in \X_{N_{\H}+1}^{0}$. Then, the derivatives of the trace function
$$\eta \in \R \mapsto \mathsf{\Psi}_{n}(\eta)f \in \X_{0}$$ are  integrable with moreover
$$\int_{\R}\left\|\frac{\d^{k}}{\d\eta^{k}}\mathsf{\Psi}_{n}(\eta)f\right\|_{\X_{0}}\d\eta < \infty \qquad \forall k \in \{0,\ldots,N_{\H}\}.$$
\end{lemme}
\begin{proof} We use here the representation of the boundary function $\mathsf{\Psi}_{n}(\eta)f$ in \eqref{eq:Psitrace}.
We recall that, for $|\eta| >R$, it holds $\mathsf{\Phi}_{n}(\eta)f=\Rs(1,\mathsf{M}_{i\eta}\H)\mathsf{G}_{i\eta}f$ and 
$$\mathsf{\Psi}_{n}(\eta)f=\Upsilon_{n}(i\eta)f=\mathsf{\Xi}_{i\eta}\H\left(\mathsf{M}_{i\eta}\H\right)^{n}\Rs(1,\mathsf{M}_{i\eta}\H)\mathsf{G}_{i\eta}f$$
where we can write
$$\Rs(1,\mathsf{M}_{i\eta}\H)=\sum_{m=0}^{\infty}\sum_{r=0}^{\mathsf{p}-1}\left(\mathsf{M}_{i\eta}\H\right)^{m\mathsf{p}+r}$$
where $R>0$ is chosen such that $\left\|\left(\mathsf{M}_{i\eta}\H\right)^{\mathsf{p}}\right\|_{\mathscr{B}(\lp)} \leq \frac{1}{2}$ for $|\eta| >R.$
Introducing, as in the proof of Proposition \ref{prop:snl}, 
$$\mathsf{L}_{k}(\l)=\left(\mathsf{M}_{\l}\H\right)^{k}, \qquad \l \in \overline{\C}_{+}\;;\;k \in \N$$
one has then
$$\mathsf{\Psi}_{n}(\eta)f=\sum_{m=0}^{\infty}\sum_{r=0}^{\mathsf{p}-1}\mathsf{\Xi}_{i\eta}\H\mathsf{L}_{m\mathsf{p}+r+n}(i\eta)\mathsf{G}_{i\eta}f$$
Exactly as in Proposition \ref{prop:snl}, Eq. \eqref{eq:derivSn}, we have, 
\begin{multline*}
\dfrac{\d^{k}}{\d\eta^{k}}\mathsf{\Psi}_{n}(\eta)f=\sum_{m=0}^{\infty}\sum_{r=0}^{\mathsf{p}-1}\sum_{\ell=0}^{k}\sum_{j=0}^{\ell}\left(\begin{array}{c}\ell \\j\end{array}\right)\left(\begin{array}{c}k \\\ell\end{array}\right)(-i)^{k-j}t_{-}^{k-\ell}\mathsf{\Xi}_{i\eta}\H\mathsf{L}_{m\mathsf{p}+r+n}^{(j)}(i\eta)\mathsf{G}_{i\eta}\left[t_{+}^{\ell-j}f\right]\end{multline*}
Recall now that, for any $\ell-j \leq k \leq N_{\H}$, 
$$\sup_{|\eta| >R}\left\|\mathsf{G}_{i\eta}\left[t_{+}^{\ell-j}f\right]\right\|_{\lp} \leq \|f\|_{\X_{N_{\H}+1}}$$
and
$$\left\|t_{-}^{k-\ell}\mathsf{\Xi}_{i\eta}\H\right\|_{\mathscr{B}(\lp)} \leq \|\mathsf{\Xi}_{0}\|_{\mathscr{B}(\Y_{k}^{-},\X_{k})}\|\H\|_{\mathscr{B}(\lp,\Y_{k}^{-})}$$
from which
\begin{equation}\label{eq:derivPsi}
\left\|\dfrac{\d^{k}}{\d\eta^{k}}\mathsf{\Psi}_{n}(\eta)f\right\|_{\X_{0}} \leq C_{k}\|f\|_{\X_{N_{\H}+1}}\sum_{m=0}^{\infty}\sum_{r=0}^{\mathsf{p}-1}\sum_{j=0}^{k}\left\|\mathsf{L}_{m\mathsf{p}+r+n}^{(j)}(i\eta)\right\|_{\mathscr{B}(\lp)}, \qquad \forall |\eta| >R.\end{equation}
We chose now $n\geq 2^{k}\mathsf{p}$ and use inequality \eqref{eq:estJ} in  Lemma \ref{lem:estJ} to deduce that, for any $m \geq 0,r\geq0$ and $j \in \{0,\ldots,k\}$, there is $\bm{C}_{k} >0$ such that
$$\left\|\mathsf{L}_{m\mathsf{p}+r+n}^{(j)}(i\eta)\right\|_{\mathscr{B}(\lp)} \leq \bm{C}_{k}\left((m+1)\mathsf{p}+n\right)^{k}\left\|\mathsf{L}_{\floor{\frac{mp+r+n}{2^{k}}}}(i\eta)\right\|_{\mathscr{B}(\lp)}.$$
Then, since $n \geq 2^{k}\mathsf{p}$, one has $\floor{\frac{mp+r+n}{2^{k}}} \geq p+\floor{\frac{m}{2^{k}}}p$ so that
\begin{multline*}
\left\|\mathsf{L}_{m\mathsf{p}+r+n}^{(j)}(i\eta)\right\|_{\mathscr{B}(\lp)} \leq \bm{C}_{k}\|\mathsf{L}_{p}(i\eta)\|_{\mathscr{B}(\lp)}\left((m+1)\mathsf{p}+n\right)^{k}\left\|\mathsf{L}_{\floor{\frac{m}{2^{k}}}p}(i\eta)\right\|_{\mathscr{B}(\lp)}\\
\leq \bm{C}_{k}\|\mathsf{L}_{\mathsf{p}}(i\eta)\|_{\mathscr{B}(\lp)}\left((m+1)\mathsf{p}+n\right)^{k}2^{-\floor{\frac{m}{2^{k}}}}\end{multline*}
where we recall that $\|\mathsf{L}_{\mathsf{p}}(i\eta)\|_{\mathscr{B}(\lp)} \leq \frac{1}{2}$ for any $|\eta| > R.$ Since 
$$\sum_{m=0}^{\infty} \sum_{r=0}^{\mathsf{p}-1}\sum_{j=0}^{k}\left((m+1)\mathsf{p}+n\right)^{k}2^{-\floor{\frac{m}{2^{k}}}} \leq \mathsf{p}(k+1)\sum_{m=0}^{\infty}\left((m+1)\mathsf{p}+n\right)^{k}2^{-\floor{\frac{m}{2^{k}}}} < \infty$$
we deduce from \eqref{eq:derivPsi} that there is a positive constant $\beta_{k} >0$ such that
\begin{equation}\label{eq:estimLpPsi}
\left\|\dfrac{\d^{k}}{\d\eta^{k}}\mathsf{\Psi}_{n}(\eta)f\right\|_{\X_{0}} \leq \beta_{k}\|f\|_{\X_{N_{\H}+1}}\left\|\mathsf{L}_{\mathsf{p}}(i\eta)\right\|_{\mathscr{B}(\lp)} \qquad \forall |\eta| >R.\end{equation}
Using \eqref{eq:power} once again, we deduce the result.
\end{proof}
 
\subsection{A second representation of the remainder terms}\label{sec:repr}


Theorem \ref{theo:existtrace} allows to establish the following new representation of $\bm{S}_{n}(t)$ where we recall that $\bm{S}_{n}(t)=\bm{U}_{\H}(t)-\sum_{k=0}^{n}\bm{U}_{k}(t)$ has been defined in Proposition \ref{prop:Snt}.

\begin{theo}\label{theo:represent2} For any $n \geq \mathsf{p}$ where $\mathsf{p}$ is defined through \eqref{eq:power} and any $f \in \X_{N_{\H}+1}$ satisfying \eqref{eq:0mean}, one has
\begin{equation}\label{eq:SntInt} 
\bm{S}_{n}(t)f=\lim_{\ell\to\infty}\frac{1}{2\pi} \int_{-\ell}^{\ell}\exp\left(i\eta t\right)\mathsf{\Psi}_{n}(\eta)f\d\eta=\frac{1}{2\pi}\int_{-\infty}^{\infty}\exp\left(i\eta t\right)\mathsf{\Psi}_{n}(\eta)f\d\eta, \qquad \forall t >0 \end{equation}
where the convergence holds in $\X_{0}$. Consequently, for any $n \geq 2^{N_{\H}}\mathsf{p}$ and any $f \in \X_{N_{\H}+1}^{0}$, 
\begin{equation}\label{eq:SnDeri}
\bm{S}_{n}(t)f=\left(-\frac{i}{t}\right)^{N_{\H}}\int_{-\infty}^{\infty}\exp\left(i\eta t\right)\dfrac{\d^{N_{\H}}}{\d\eta^{N_{\H}}}\mathsf{\Psi}_{n}(\eta)f\frac{\d\eta}{2\pi}\end{equation}
holds true for any $t \geq0$ where the convergence of the integral holds in $\X_{0}$.

\end{theo}
\begin{proof} We first deduce from the uniform convergence obtained in Theorem \ref{theo:existtrace} that,  for any $\ell > 0$ and any $f \in \X_{N_{\H}+1}$ and any $t \geq0$
\begin{equation}\label{conv:integ}
\lim_{\e\to0^{+}}\frac{1}{2\pi}\int_{-\ell}^{\ell}\exp\left((\e+i\eta)t\right)\Upsilon_{n}(\e+i\eta)f\d \eta=
\frac{1}{2\pi}\int_{-\ell}^{\ell}\exp\left(i\eta t\right)\mathsf{\Psi}_{n}(\eta)f\d \eta\end{equation}
where the convergence occurs in $\X_{0}$. 

For any $\eta\neq0$, one has
$\mathsf{\Psi}_{n}(\eta)f=\mathsf{\Xi}_{i\eta}\H\left(\mathsf{M}_{i\eta}\H\right)^{n}\Rs(1,\mathsf{M}_{i\eta}\H)\mathsf{G}_{i\eta}f$
so that \begin{multline}\label{eq:Psin}
\|\mathsf{\Psi}_{n}(\eta)f\|_{\X_{0}} \leq \|\mathsf{\Xi}_{0}\|_{\mathscr{B}(\Y_{1}^{-},\X_{0})}\|\H\|_{\mathscr{B}(\lp,\Y_{1}^{-})}\left\|\left(\mathsf{M}_{i\eta}\H\right)^{n}\right\|_{\mathscr{B}(\lp)}\\
\left\|\Rs(1,\mathsf{M}_{i\eta}\H)\right\|_{\mathscr{B}(\lp)}\,\|\mathsf{G}_{i\eta}\|_{\mathscr{B}(\X_{1},\lp)}\|f\|_{\X_{1}}\\
\leq\|\mathsf{\Xi}_{0}\|_{\mathscr{B}(\Y_{1}^{-},\X_{0})}\|\H\|_{\mathscr{B}(\lp,\Y_{1}^{-})}\left\|\left(\mathsf{M}_{i\eta}\H\right)^{n}\right\|_{\mathscr{B}(\lp)}\left\|\Rs(1,\mathsf{M}_{i\eta}\H)\right\|_{\mathscr{B}(\lp)}\|f\|_{\X_{1}}\end{multline}
where we used that $|\mathsf{\Xi}_{i\eta}| \leq \mathsf{\Xi}_{0}$, $\|\mathsf{G}_{i\eta}\|_{\mathscr{B}(\X_{0},\lp)}\leq \|\mathsf{G}_{0}\|_{\mathscr{B}(\X_{0},\lp)} \leq 1$ and $\|\mathsf{M}_{i\eta}\H\|_{\mathscr{B}(\lp)} \leq 1$. Recall (see Lemma \ref{lem:unifconti}) that  
we can find $R >0$ large enough such that
$$\left\|\left(\mathsf{M}_{i\eta}\H\right)^{\mathsf{p}}\right\|_{\mathscr{B}(\lp)} \leq \frac{1}{2}, \qquad \forall |\eta| > R.$$
Arguing exactly as in Corollary \ref{cor:ResMei}, one proves easily that, for $|\eta| >R$, 
$$\Rs(1,\mathsf{M}_{i\eta}\H)=\sum_{m=0}^{\infty}\sum_{j=0}^{\mathsf{p}-1}\left(\mathsf{M}_{i\eta}\H\right)^{m\mathsf{p}+j}$$ from which
\begin{equation*}\label{eq:RsMin}
\|\Rs(1,\mathsf{M}_{i\eta}\H)\|_{\mathscr{B}(\lp)} \leq \mathsf{p} \sum_{m=0}^{\infty}\|\left(\mathsf{M}_{i\eta}\H\right)^{\mathsf{p}}\|_{\mathscr{B}(\lp)}^{m}=\frac{\mathsf{p}}{1-\left\|\left(\mathsf{M}_{i\eta}\H\right)^{\mathsf{p}}\right\|_{\mathscr{B}(\lp)}} \leq 2\mathsf{p}\end{equation*}
Combining this with \eqref{eq:Psin}, one gets
$$\|\mathsf{\Psi}_{n}(\eta)f\|_{ \X_{0}} \leq 2\mathsf{p}\|\mathsf{\Xi}_{0}\|_{\mathscr{B}(\Y_{1}^{-},\X_{0})}\|\H\|_{\mathscr{B}(\lp,\Y_{1}^{-})}\left\|\left(\mathsf{M}_{i\eta}\H\right)^{n}\right\|_{\mathscr{B}(\lp)}\|f\|_{\X_{1}}, \qquad |\eta| >R.$$
Moreover, for $n \geq \mathsf{p}$, $\|\left(\mathsf{M}_{i\eta}\H\right)^{n}\|_{\mathscr{B}(\lp)} \leq \|\left(\mathsf{M}_{i\eta}\H\right)^{\mathsf{p}}\|_{\mathscr{B}(\lp)}$ so that, there exists $C >0$  such that
$$\|\mathsf{\Psi}_{n}(\eta)f\|_{\X_{0}}  \leq C\|\left(\mathsf{M}_{i\eta}\H\right)^{\mathsf{p}}\|_{\mathscr{B}(\lp)}\|f\|_{\X_{1}} \qquad |\eta| >R, \qquad n \geq \mathsf{p}.$$
The same exact reasoning shows that, actually, 
$$\sup_{\e \in (0,1)}\|\exp\left((\e+i\eta)t\right)\Upsilon_{n}(\e+i\eta)f\|_{\X_{0}} \leq C\exp(t)\left\|\left(\mathsf{M}_{i\eta}\H\right)^{\mathsf{p}}\right\|_{\mathscr{B}(\lp)}\|f\|_{\X_{1}}$$
for any $|\eta| > R$ and any $n\geq \mathsf{p}.$ Since the mapping $|\eta| > R \longmapsto \left\|\left(\mathsf{M}_{i\eta}\H\right)^{\mathsf{p}}\right\|_{\mathscr{B}(\lp)}$ is integrable thanks to \eqref{eq:power}, we deduce, from the dominated convergence theorem that
$$\lim_{\e\to0^{+}}\int_{|\eta| >R}\exp\left((\e+i\eta)t\right)\Upsilon_{n}(\e+i\eta)f\d\eta=\int_{|\eta| >R}\exp(i\eta t)\mathsf{\Psi}_{n}(\eta)f\d \eta$$
where the convergence holds in $\X_{0}$. 
This proves \eqref{eq:SntInt} according to the representation formula \eqref{eq:maindecay}. The proof of \eqref{eq:SnDeri} is then deduced easily after $N_{\H}$ integration by parts, using Lemma \ref{prop:reguPsif}.\end{proof}
 %

\subsection{Decay rate: proof of  Theorem \ref{theo:maindec}}

We can prove our main result  stated in the Introduction
\begin{proof}[Proof of Theorem \ref{theo:maindec}] Let us fix $f \in \X_{N_{\H}+1}$. To prove the result, we can assume without loss of generality that $\varrho_{f}=0$, i.e $f \in \X_{N_{\H}+1}^{0}.$ Of course, the term $\mathsf{\Theta}_{f}(\cdot)$ in Theorem \ref{theo:maindec} is given by
$$
\mathsf{\Theta}_{f}(\eta)=\dfrac{\d^{N_{\H}}}{\d \eta^{N_{\H}}}\mathsf{\Psi}_{n}(\eta)f \in  \X_{0}, \qquad \eta \in \R$$
for some suitable choice of $n \in \N.$ Recall first that, for any $n \in \N$ and any $t\geq 0$
$$U_{\H}(t)f=\sum_{k=0}^{\infty}\bm{U}_{k}(t)f=\sum_{k=0}^{n}\bm{U}_{k}(t)+\bm{S}_{n}(t)f$$
where, according to Proposition \ref{prop:Snt},
$$\left\|\sum_{k=0}^{n}\bm{U}_{k}(t)f\right\|_{\X_{0}} \leq C_{n}(1+t)^{-N_{\H}-1}, \qquad \forall t\geq 0$$
for some positive constant $C_{n}$ depending on $n$ and $f$ (but not on $t$). Choosing now $n\geq 2^{N_{\H}}\mathsf{p}$ and using \eqref{eq:SnDeri}, one obtains
$$\left\|U_{\H}(t)f\right\|_{\X_{0}} \leq  C_{n}(1+t)^{-N_{\H}-1}+t^{-N_{\H}}\mathcal{F}_{n}(t)$$
where
$$\mathcal{F}_{n}(t)=\left\|\frac{(-i)^{N_{\H}}}{2\pi}\int_{-\infty}^{\infty}\exp\left(i\eta t\right)\dfrac{\d^{N_{\H}}}{\d\eta^{N_{\H}}}\mathsf{\Psi}_{n}(\eta)f\d\eta\right\|_{\X_{0}}$$
is such that $\lim_{t\to \infty} \mathcal{F}_{n}(t)=0$ according to Riemann-Lebesgue Theorem (recall the mapping $\eta \mapsto \dfrac{\d^{N_{\H}}}{\d\eta^{N_{\H}}}\mathsf{\Psi}_{n}(\eta)f \in \X_{0}$ is integrable over $\R$ according to Lemma \ref{prop:reguPsif}). This proves the first part of the result. Let us now prove the second part of it and assume that \eqref{eq:decay-power} holds true, i.e.
\begin{equation*}
\int_{|\eta| >R}\left\|\left(\mathsf{M}_{i\eta}\H\right)^{\mathsf{p}}\right\|_{\mathscr{B}(\lp)} \d \eta \leq \frac{C(\mathsf{p})}{R^{\beta}}, \qquad \forall R >0\end{equation*}
for some $C(\mathsf{p}) >0.$ We already observed that the mapping
$$\mathsf{\Theta}_{f}\::\:\eta \in \R \longmapsto  \dfrac{\d^{N_{\H}}}{\d\eta^{N_{\H}}}\mathsf{\Psi}_{n}(\eta)f \in \X_{0}$$
belongs to $\mathscr{C}_{0}(\R,\X_{0})$ so is  uniformly continuous. This allows to define a (minimal) modulus of continuity 
$$\omega_{f}(s):=\sup\left\{\left\|\mathsf{\Theta}_{f}(\eta_{1})-\mathsf{\Theta}_{f}(\eta_{2})\right\|_{\X_{0}}\,;\,\eta_{1},\eta_{2}\in \R,\,|\eta_{1}-\eta_{2}|\leq s\right\},\quad  s\geq 0.$$
The estimate then comes from some standard reasoning about Fourier transform (see for instance \cite[Theorem 3.3.9 (b), p. 196]{grafakos} for similar considerations for the decay of Fourier coefficients of H\"older function). Namely, introducing the Fourier transform  of the (Bochner integrable) function $\mathsf{\Theta}_{f}$ as 
$$\widehat{\mathsf{\Theta}}_{f}(t)=\int_{\R}\exp(i\eta t)\mathsf{\Theta}_{f}(\eta)\d\eta \in  \X_{0}, \qquad t \geq 0$$
one has then, since $e^{i\pi}=-1=\exp(i\pi t/t),$ $t >0$,
$$\widehat{\mathsf{\Theta}}_{f}(t)=-\int_{\R}\exp\left(i\eta t + i \frac{\pi}{t} t\right)\mathsf{\Theta}_{f}(\eta)\d \eta=-\int_{\R}\exp\left(i y t\right)\mathsf{\Theta}_{f}\left(y-\frac{\pi}{t}\right)\d y$$
which gives, taking the mean of both the expressions of $\widehat{\mathsf{\Theta}}_{f}(t)$,
$$\widehat{\mathsf{\Theta}}_{f}(t)=\frac{1}{2}\int_{\R}\exp\left(i\eta t\right)\left(\mathsf{\Theta}_{f}(\eta)-\mathsf{\Theta}_{f}\left(\eta-\frac{\pi}{t}\right)\right)\d\eta.$$
Consequently, if one assumes that $R > 2\pi$, 
\begin{equation*}
\left\|\widehat{\mathsf{\Theta}}_{f}(t)\right\|_{\X_{0}} \leq \frac{1}{2}\int_{|\eta|\leq R}\left\|\mathsf{\Theta}_{f}(\eta)-\mathsf{\Theta}_{f}\left(\eta-\frac{\pi}{t}\right)\right\|_{\X_{0}}\d\eta 
+\int_{|\eta| >\frac{R}{2}}\left\|\mathsf{\Theta}_{f}(\eta)\right\|_{ \X_{0}}\d\eta\end{equation*}
where we used that $\{\eta \in \R\;;\;|\eta+\frac{\pi}{t}| > R\} \subset \{\eta \in \R\;;\;|\eta| > R-\pi\} \subset \{\eta \in \R\;;\;|\eta| >\frac{R}{2}\}$ since $t\geq1,$ $R-\pi >\frac{R}{2}$. Therefore, using the modulus of continuity $\omega_{f}$ and the additional assumption \eqref{eq:decay-power}, we deduce that
\begin{equation}\begin{split}\label{eq:Hol}
\left\|\widehat{\mathsf{\Theta}}_{f}(t)\right\|_{\X_{0}}&\leq 2R\omega_{f}\left(\frac{\pi}{t}\right)  + \int_{|\eta| >\frac{R}{2}}\left\|\mathsf{\Theta}_{f}(\eta)\right\|_{\X_{0}}\d\eta\\
&\leq 2R\omega_{f}\left(\frac{\pi}{t}\right) + 2^{\beta}C(\mathsf{p})R^{-\beta}\|f\|_{\X_{N_{\H}+1}}, \qquad \forall R > 2\pi, \qquad t \geq 1.\end{split}
\end{equation}
Optimising then the parameter $R$, i.e. choosing
$$R=\left(\frac{2^{\beta-1}\beta\,C(\mathsf{p})\|f\|_{\X_{N_{\H}+1}} }{\omega_{f}\left(\frac{\pi}{t}\right)}\right)^{\frac{1}{\beta+1}}$$
(up to work with $t \geq t_{0} \geq 1$ to ensure that $R >2\pi$), we obtain the desired estimate.\end{proof} 
\begin{nb} Notice that, if \eqref{eq:decay-power} holds true \emph{for any} $\beta >0$ large enough, then one sees that the decay rate of $\widehat{\mathsf{\Theta}}_{f}(t)$ can be estimated as 
$$\|\widehat{\mathsf{\Theta}}_{f}(t)\|_{\X_{0}}=\mathbf{O}\left(\omega_{f}\left(\frac{\pi}{t}\right)\right)^{\alpha}$$
for any $\alpha \in (0,1)$. This means that the decay rate of $\|\widehat{\mathsf{\Theta}}_{f}(t)\|_{\X_{0}}$ can be made as close as possible (without reaching it) to that of $t \mapsto \omega_{f}\left(\frac{\pi}{t}\right)$.\end{nb}

\subsection{Comments and conjecture}\label{sec:conj}

On the basis of the second part of Theorem \ref{theo:maindec}, a careful study of the modulus of continuity $\omega_{f}(\cdot)$ of the mapping
$$\eta \in \R \longmapsto \mathsf{\Theta}_{f}(\eta)=\dfrac{\d^{N_{\H}}}{\d\eta^{N_{\H}}}\mathsf{\Psi}_{n}(\eta)f \in \X_{0}$$
where $\mathsf{\Psi}_{n}(\eta)f=\mathsf{\Xi}_{i\eta}\H\left[\mathsf{M}_{i\eta}\H\right]^{n}\mathsf{\Phi}(\eta)f$ (with $\mathsf{\Phi}(\eta)f$ defined in \eqref{eq:defPhi})
would be fundamental to make more explicit our main rate of convergence
$$\left\|U_{\H}(t)f-\varrho_{f}\Psi_{H}\right\|_{\X_{0}} \leq \frac{C_{f}}{(1+t)^{N_{\H}}} \bm{\epsilon}(t) \qquad \forall t\geq 0.$$
We actually are able to evaluate the modulus of continuity of the bounded and uniformly continuous  mapping
\begin{equation}\label{eq:chi}
\bm{\chi}\::\:\eta \in \R \longmapsto \dfrac{\d^{N_{\H}}}{\d\eta^{N_{\H}}}\mathsf{\Xi}_{i\eta}\H \in \mathscr{B}(\lp,\X_{0})\end{equation}
(see Remark \ref{nb:conj}) and we conjecture that it also the modulus of continuity of $\mathsf{\Theta}_{f}(\cdot)$. Namely,
\begin{equation}\label{eq:conj}
\omega_{f}(s) \leq C_{0}\,\bm{\omega}(s) + C_{1}s \qquad \forall s \geq0\end{equation}
for some positive $C_{0},C_{1} >0$  where $\bm{\omega}(\cdot)$ is 
the (minimal) modulus of continuity of $\bm{\chi}.$ 

We point out that this is indeed the case on $\R \setminus (-\delta,\delta)$, $\delta >0$. However, the general strategy of this paper seems to be such that we cannot deal with the modulus of continuity in the neighborhoud of the origin of a derivative of order $N_{\H}$.

  
\begin{nb}\label{nb:conj} The interest of the above conjecture is of course that investigating the modulus of continuity of $\bm{\chi}$ is technically much simpler. In particular, under the additional assumption that there exists $\alpha \in (0,1)$ such that
\begin{equation}\label{eq:Yfrac}
\H \in \mathscr{B}(\lp,\Y_{N_{\H}+1+\alpha}^{-})\end{equation}
then, one can prove that the modulus of continuity $\bm{\omega}(\cdot)$ of $\bm{\chi}$ is such that
$$\bm{\omega}(s) \leq c_{1}|s|^{\frac{k_{\H}}{1-\alpha+k_{\H}}}, \qquad \forall s \geq0$$
for some $c_{1} >0$ and where $k_{\H}$ is defined as 
$$k_{\H}:=\max\{s \geq 0\;;\;\limsup_{\e\to0^{+}}\e^{-s}\left\|\widetilde{\H}^{(\e)}\right\|_{\mathscr{B}(\lp,\lm)} < \infty\}$$
with $\widetilde{\H}^{(\e)}=\ind_{|v| < \e}\H$ defined in \eqref{eq:Hepsi}. In this case, if the conjecture \eqref{eq:conj} is correct, the rate of convergence in Theorem \ref{theo:maindec} would be upgraded in
\begin{equation}\label{eq:rateamel}\left\|U_{\H}(t)f-\varrho_{f}\Psi_{\H}\right\|_{\X_{0}}=\mathbf{O}\left(t^{-N_{\H}-\frac{\beta k_{\H}}{(1+\beta)(1-\alpha+k_{\H})}}\right), \quad \forall f \in \X_{N_{\H}+1}\end{equation}
In particular, for the model described in Example \ref{exe:maxw}, since $\beta$ can be chosen arbitrarily and \eqref{eq:Yfrac} holds for any $\alpha \in (0,1)$, the rate of convergence would be 
$$\left\|U_{\H}(t)f-\varrho_{f}\Psi_{\H}\right\|_{\X_{0}}=\mathbf{O}\left(t^{-d+\delta}\right), \qquad \forall f \in \X_{d}$$
for any $\delta >0$ (see Section \ref{sec:exam} for details).
\end{nb}
\section{About Assumptions \ref{hypH}}\label{sec:REGU}

We recall here that, in our previous contribution \cite{LMR}, we introduced a general class of diffuse  boundary operators that we called \emph{regular} and proved in  \cite[Theorem 5.1]{LMR}, that, if $\H$ is a regular and stochastic diffuse boundary operator, then
$$\H\mathsf{M}_{0}\H \in \mathscr{B}(\lp,\lm) \text{ is weakly-compact,}$$ i.e. Assumption  \ref{hypH} \textit{2)} is met. We also provided in \cite{LMR} practical criteria ensuring that $\mathsf{M}_{0}\H$ is irreducible, satisfying then Assumption \ref{hypH} \textit{3)}. Point \textit{1)} of Assumption \ref{hypH} is of course something that has to be computed for each specific boundary operator $\H$. We focus on this Section on some practical criteria yielding to Assumption \ref{hypH} \textit{4)}. 

\subsection{Some useful change of variables}\label{appen:chv}
We establish now a fundamental change of variable formula which has its own interest. A systematic use of such a change of variable will be made in the companion paper \cite{LM-iso}. We begin with the following technical lemma:

\begin{lemme}\phantomsection\label{lem:jaco}
Let $x \in \partial{\Omega}$ be fixed. We denote by $\mathbf{B}_{d-1}$ the closed unit ball of $\R^{d-1}$ and define
$$\mathfrak{p}\::\:z \in \mathbf{B}_{d-1} \longmapsto \mathfrak{p}(z)=x-\tau_{-}(x,\sigma(z))\sigma(z) \in \partial{\Omega}$$
where
$$\sigma(z)=\left(\sigma_{1}(z),\ldots,\sigma_{d}(z)\right)=\left(z_{1},\ldots,z_{d-1},\sqrt{1-|z|^{2}}\right) \in \S^{d-1}, \qquad z \in \mathbf{B}_{d-1}.$$
Defining
$$\mathcal{O}_{x}:=\left\{z \in \mathbf{B}_{d-1}\;;\;\sigma(z) \cdot n(x) >0\;;\;\sigma(z) \cdot n(\mathfrak{p}(z)) < 0\right\},$$ 
it holds that $\mathfrak{p}$ is differentiable on $\mathcal{O}_{x}$ and
$$\mathrm{det}\Bigg(\Bigg(\left\langle \frac{\partial \mathfrak{p}(z)}{\partial z_{i}},\frac{\partial \mathfrak{p}(z)}{\partial z_{j}}\right\rangle\Bigg)_{1\leq i,j\leq d-1}\Bigg)=\left(\frac{\tau_{-}(x,\sigma(z))^{d-1}}{\left(\sigma(z)\cdot n(\mathfrak{p}(z))\right)\sigma_{d}(z)}\right)^{2} \qquad \forall z \in \mathcal{O}_{x}.$$\end{lemme}
\begin{proof} The fact that $\mathfrak{p}(\cdot)$ is differentiable on $\mathcal{O}_{x}$ is a consequence of \cite[Lemma A.4]{LMR}. We recall in particular here that the set $\widehat{\Gamma}_{+}(x)=\{\omega \in \S^{d-1}\;;\;(x,\omega) \in \Gamma_{+}\text{ and } \bm{\xi}(x,\omega)=(x-\tau_{-}(x,\omega)\omega,\omega) \in \Gamma_{-}\}$ is an open subset of $\{\omega \in \S^{d-1}\;;\;(x,\omega) \in \Gamma_{+}\}$ and $\tau_{-}(x,\cdot)$ is differentiable on $\widehat{\Gamma}_{+}(x)$ with
\begin{equation}\label{eq:difftau}
\partial_{\omega}\tau_{-}(x,\omega)=\frac{\tau_{-}(x,\omega)}{w\cdot n(\bm{\xi}(x,\omega))}n(\bm{\xi}_{s}(x,\omega)), \qquad \omega \in \widehat{\Gamma}_{+}(x),
\end{equation}
where $\bm{\xi}_{s}(x,\omega)=x-\tau_{-}(x,\omega)\omega$. Since, for any $z \in \mathcal{O}_{x}$, $\mathfrak{p}(z)=\bm{\xi}_{s}(x,\sigma(z))$, this translates in a straightforward way to the differentiability of $\mathfrak{p}$. Moreover, one deduces from \eqref{eq:difftau} that
\begin{equation*}\begin{split}
\partial_{i}\mathfrak{p}(z)&=-\left(\nabla_{\omega}\tau_{-}(x,\sigma(z))\cdot\partial_{i}\sigma(z)\right)\sigma(z)-\tau_{-}(x,\sigma(z))\partial_{i}\sigma(z)\\
&=-\frac{\tau_{-}(x,\sigma(z))}{\sigma(z)\cdot n(\mathfrak{p}(z))}\left[\left(n(\mathfrak{p}(z))\cdot\partial_{i}\sigma(z)\right)\sigma(z)+\left(n(\mathfrak{p}(z))\cdot\sigma(z)\right)\partial_{i}\sigma(z)\right],\end{split}\end{equation*}
where we denote for simplicity $\partial_{i}=\frac{\partial}{\partial z_{i}}$. We therefore get
\begin{multline*}
\left\langle \partial_{i}\mathfrak{p}(z)\,;\,\partial_{j}\mathfrak{p}(z)\right\rangle=\left(\frac{\tau_{-}(x,\sigma(z))}{\sigma(z)\cdot n(\mathfrak{p}(z))}\right)^{2}\bigg(\left(n(\mathfrak{p}(z))\cdot\partial_{i}\sigma(z)\right)\left(n(\mathfrak{p}(z))\cdot\partial_{j}\sigma(z)\right) \\
 +\left(n(\mathfrak{p}(z))\cdot\sigma(z)\right)^{2}\left(\partial_{i}\sigma(z)\cdot \partial_{j}\sigma(z)\right)\bigg).\end{multline*}
Let us fix $z \in \mathcal{O}_{x}$. We denote by $\mathfrak{P}(z)$ the matrix whose entries are $\mathfrak{P}_{ij}(z)=\left\langle \partial_{i}\mathfrak{p}(z)\,;\,\partial_{j}\mathfrak{p}(z)\right\rangle$, $1 \leq i,j \leq d-1$. Using that
$$\partial_{i}\sigma(z)\cdot\partial_{j}\sigma(z)=\delta_{ij}+\frac{z_{i}z_{j}}{\sigma_{d}^{2}(z)},$$
where $\sigma_{d}(z)$ is the last component of $\sigma(z)$, i.e. $\sigma_{d}(z)=\sqrt{1-|z|^{2}}$, one sees that
$$\mathfrak{P}_{ij}(z)=\tau_{-}^{2}(x,\sigma(z))\Bigg[\delta_{ij}+\frac{z_{i}z_{j}}{\sigma_{d}^{2}(z)}+\frac{1}{(\sigma(z)\cdot n(\mathfrak{p}(z)))^{2}}\left(n(\mathfrak{p}(z))\cdot\partial_{i}\sigma(z)\right)\left(n(\mathfrak{p}(z))\cdot\partial_{j}\sigma(z)\right)\Bigg].$$
For simplicity, we will simply denote by $n$ the unit vector $n(\mathfrak{p}(z))$ and $\sigma=\sigma(z)$. Introducing the vectors $u,p\in \R^{d-1}$ whose components are 
$$u_{i}:=\frac{n\cdot\partial_{i}\sigma}{\sigma\cdot n}, \qquad p_{i}:=\frac{z_{i}}{\sigma_{d}}, \qquad i=1,\ldots,d-1$$
we have $\mathfrak{P}_{ij}(z)=\tau_{-}(x,\sigma)^{2}\left[\delta_{ij}+p_{i}p_{j}+u_{i}u_{j}\right]$ so that
$$\mathrm{det}\left(\mathfrak{P}(z)\right)=\left(\tau_{-}(x,\sigma)\right)^{2(d-1)}\mathrm{det}\left(\mathbf{I}_{d-1}+p\otimes p +u \otimes u\right).$$
Recall that, for any invertible matrix $\bm{A}$ and any vectors $\bm{x},\bm{y} \in \R^{d-1}$, then
\begin{equation}\label{detA}
\mathrm{det}\left(\bm{A}+\bm{x}\otimes\bm{y}\right)=\mathrm{det}(A)\left(1+\langle \bm{y},\bm{A}^{-1}\bm{x}\rangle\right).\end{equation}
We apply this identity first by considering $\bm{A}=\mathbf{I}_{d-1}+p\otimes p$ and get
$$\mathrm{det}\left(\mathfrak{P}(z)\right)=\left(\tau_{-}(x,\sigma)\right)^{2(d-1)}\mathrm{det}(\bm{A})\left(1+\langle u,\bm{A}^{-1}u\rangle\right).$$
To compute $\mathrm{det}(\bm{A})$, one uses again \eqref{detA} to deduce
$$\mathrm{det}(\bm{A})=1+\langle p,p\rangle=1+\frac{|z|^{2}}{\sigma^{2}_{d}}=\sigma_{d}^{-2}.$$
One also can compute in a direct way the inverse of $\bm{A}$ given by $\bm{A}^{-1}=\mathbf{I}_{d-1}-\frac{1}{1+|p|^{2}}p\otimes p$ from which
$$\langle u,\bm{A}^{-1}u\rangle=|u|^{2}-\frac{\langle p,u\rangle^{2}}{1+|p|^{2}}.$$
This results in 
\begin{equation}\label{eq:detPz}
\mathrm{det}\left(\mathfrak{P}(z)\right)=\left(\tau_{-}(x,\sigma)\right)^{2(d-1)}\left(1+|u|^{2}-\frac{\langle p,u\rangle^{2}}{1+|p|^{2}}\right).\end{equation}
Let us make this more explicit. One easily checks that
$$u_{i}=\frac{1}{\sigma\cdot n}\left(n_{i}-\frac{n_{d}}{\sigma_{d}}z_{i}\right) \qquad \text{ and } \qquad \langle p,u\rangle=\frac{1}{(\sigma\cdot n)\sigma_{d}}\sum_{i=1}^{d-1}\left(n_{i}z_{i}-\frac{n_{d}}{\sigma_{d}}z_{i}^{2}\right).$$
Noticing that $\sum_{i=1}^{d-1}n_{i}z_{i}=(\sigma\cdot n)-\sigma_{d}n_{d}$, it holds
$$\langle p,u\rangle=\frac{1}{(\sigma\cdot n)\sigma_{d}}\left(\sigma\cdot n - \sigma_{d}n_{d}-\frac{n_{d}}{\sigma_{d}}|z|^{2}\right)=\frac{1}{\sigma_{d}^{2}}\left(\sigma_{d}-\frac{n_{d}}{\sigma\cdot n}\right),$$
where we use again that $\sigma_{d}^{2}+|z|^{2}=1$. Since one also has
$$|u|^{2}=\frac{1}{(\sigma\cdot n)^{2}}\sum_{i=1}^{d-1}\left(n_{i}-\frac{n_{d}}{\sigma_{d}}z_{i}\right)^{2}=\frac{1}{\sigma_{d}^{2}(\sigma\cdot n)^{2}}\sum_{i=1}^{d-1}\left(\sigma_{d}n_{i}-z_{i}n_{d}\right)^{2},$$
we get easily after expanding the square and using that $\sum_{i=1}^{d-1}n_{i}^{2}=1-n_{d}^{2}$, 
$$|u|^{2}=\frac{1}{\sigma_{d}^{2}(\sigma\cdot n)^{2}}\left(\sigma_{d}^{2}+n_{d}^{2}-2(\sigma\cdot n)\sigma_{d}n_{d}\right).$$
One finally obtains, using \eqref{eq:detPz},
$$\mathrm{det}\left(\mathfrak{P}(z)\right)=\frac{\tau_{-}(x,\sigma)^{2(d-1)}}{\sigma_{d}^{2}}\left(1+\frac{1}{\sigma_{d}^{2}(\sigma\cdot n)^{2}}\left(\sigma_{d}^{2}+n_{d}^{2}-2(\sigma\cdot n)\sigma_{d}n_{d}\right)-\frac{1}{\sigma_{d}^{2}}\left(\sigma_{d}-\frac{n_{d}}{\sigma\cdot n}\right)^{2}\right)$$
and little algebra gives
$$\mathrm{det}\left(\mathfrak{P}(z)\right)=\frac{\tau_{-}(x,\sigma)^{2(d-1)}}{\sigma_{d}^{2}(\sigma\cdot n)^{2}}$$
which is the desired result.\end{proof}

We complement the above with the following
\begin{lemme}\phantomsection\label{lem:sard}
For any $x \in \partial\Omega$, introduce
$$\mathcal{G}_{x}:=\{\omega \in \S^{d-1}\;;\;(x,\omega) \in \Gamma_{+}\;;\;\omega \cdot n(x-\tau_{-}(x,\omega)\omega)=0\}.$$
Then, 
$$|\mathcal{G}_{x}|=0,$$
where here $|\cdot|$ denotes the Lebesgue surface measure over $\S^{d-1}$. Moreover, with the notations of Lemma \ref{lem:jaco}, the set
$$\mathbf{G}_{x}=\left\{z \in \mathbf{B}_{d-1}\;;\;\sigma(z)=(z,\sqrt{1-|z|^{2}}) \in \mathcal{G}_{x}\right\}$$
has zero Lebesgue measure (in $\R^{d-1}$).
\end{lemme}
\begin{proof} The proof is based on Sard's Theorem. Let $x \in \partial\Omega$ be fixed. Introducing the function
$$\Psi\::\:y \in \partial\Omega \setminus\{x\} \longmapsto \Psi(y)=\frac{x-y}{|x-y|},$$
it holds that $\Psi$ is a $\mathscr{C}^{1}$ function. For any $\omega \in \S^{d-1}$ setting $y_{\omega}=x-\tau_{-}(x,\omega)\omega \in \partial\Omega$, one has
$$\omega=\Psi(y_{\omega}).$$
Let us prove that $\mathcal{G}_{x}$ is included in the set of \emph{critical values} of $\Psi$. To do this, we show that, if $\omega \in \mathcal{G}_{x}$, then $y_{\omega}$ is a critical point of $\Psi$, i.e. the differential $\d\Psi(y_{\omega})$ is not injective. Since actually $\Psi$ is defined and smooth on $\R^{d}\setminus\{x\}$, its differential on $\partial\Omega \setminus \{x\}$ is nothing but the restriction of its differential on $\R^{d}\setminus\{x\}$ on the tangent hyperplane to $\partial\Omega$, i.e., for any $y \in \partial\Omega \setminus\{x\}$, one has
$$\d\Psi(y)\::\:h \in \mathcal{T}_{y} \longmapsto -\frac{1}{|x-y|}\mathds{P}_{z_{y}}(h),$$
where $\mathcal{T}_{y}$ is the tangent space of $\partial\Omega$ at $y\in \partial \Omega \setminus\{x\}$, $z_{y}=\frac{x-y}{|x-y|}$ and, for any $z \in \R^{d}$,  $\mathds{P}_{z}$ denotes the orthogonal projection onto the hyperplane orthogonal to $z$, 
$$\mathds{P}_{z}h=h_{z}^{\perp}:=h-\langle h,\,\bar{z}\rangle\,\bar{z}, \qquad \bar{z}=\frac{z}{|z|} \in \mathbb{S}^{d-1}, h \in \R^{d}.$$
Now, one notices that
$$\omega \in \mathcal{G}_{x} \implies \omega \cdot n(y_{\omega})=0, \qquad \omega=\Psi(y_{\omega}).$$
In particular, one has $\omega \in \mathcal{T}_{y_{\omega}}$ and $\mathds{P}_{z_{y_{\omega}}}(\omega)=0$ since $z_{y_{\omega}}=\frac{x-y_{\omega}}{|x-y_{\omega}|}=\omega$. In particular, $\d\Psi(y_{\omega})$ is not injective ($\omega \neq 0$ belongs to its kernel). We proved that, if $\omega \in \mathcal{G}_{x}$, then it is a \emph{critical value} of $\Psi$ and Sard's Theorem implies in particular that the measure of $\mathcal{G}_{x}$ is zero. 

Now, to prove that $\mathbf{G}_{x}$ is also of zero measure, one simply notices that $\mathbf{G}_{x}$ is  the image of $\mathcal{G}_{x}$ through the smooth function
$$\mathcal{P}\::\:\omega \in \S^{d-1}\setminus \{w_{d} =0\} \longmapsto \mathcal{P}(\omega)=(\omega_{1},\ldots,\omega_{d-1}) \in \mathbf{B}_{d-1}.$$
In particular, from the first part of the lemma, $\mathbf{G}_{x} $ is included in the set of critical values of the smooth function $\mathcal{P} \circ \Psi$, and we conclude again with Sard's Theorem.
\end{proof}

 \begin{lemme}\phantomsection\label{lem:ChVa} Assume that $\partial\Omega$ satisfies Assumption \ref{hypO}. For any $x \in \partial\Omega$, we set
 $$\S_{+}(x)=\left\{\sigma \in \S^{d-1}\;;\;\sigma \cdot n(x) >0\right\}.$$
Then, for any nonnegative measurable mapping $g\::\:\S^{d-1} \mapsto \R$, one has,
$$\int_{\S_{+}(x)}g(\sigma)\,|\sigma\cdot n(x)|\d\sigma=\int_{\partial\Omega}g\left(\frac{x-y}{|x-y|}\right)\mathcal{J}(x,y)\pi(\d y),$$
and 
\begin{equation}\label{eq:Jxy}
\mathcal{J}(x,y)=\ind_{\Sigma_{+}(x)}(y)\frac{|(x-y)\cdot n(x)|}{|x-y|^{d+1}}\,|(x-y)\cdot n(y)|, \qquad \forall y \in \Sigma_{+}(x)\end{equation}
with
$$\Sigma_{+}(x)=\left\{y \in \partial\Omega\::\:\left]x,y\right[ \subset \Omega\,;\,(x-y) \cdot n(x) > 0\;;\;n(x-y) \cdot n(y) < 0\right\}$$
where $\left]x,y\right[=\{tx+(1-t)y\,;\,0 < t < 1\}$ is the open segment joining $x$ and $y$  \footnote{Observe that, if $\Omega$ is convex, then 
$$\Sigma_{+}(x)=\left\{y \in \partial\Omega\;;\;\, (x-y) \cdot n(x) >0 \text{ and } (x-y)\cdot n(y) < 0\right\}$$
and $\partial\Omega=\bigcup_{x\in\partial\Omega}\Sigma_{+}(x)$ whereas, if $\Omega$ is not convex $\bigcup_{x \in \partial\Omega}\Sigma_{+}(x) \neq \partial\Omega$.}. 

In particular, for any nonnegative measurable $G\::\:\partial \Omega \to \R$ and any measurable $\varphi\::\:\R^{+}\to\C$,  we have
\begin{equation}\label{eq:CofV}
\int_{\S_{+}(x)}\left|\sigma \cdot n(x)\right|\,\varphi(\tau_{-}(x,\sigma)) \,G(x-\tau_{-}(x,\sigma)\sigma)\d\sigma=\int_{\Sigma_{+}(x)}G(y)\varphi(|x-y|)\mathcal{J}(x,y)\pi(\d y).\end{equation}
\end{lemme}
\begin{proof} Let now $x \in \partial\Omega$ be given. We can assume without generality that the system of coordinates in $\R^{d}$ is such that $n(x)=(0,\ldots,1)$. For a given $f\::\:\S_{+}(x)\to \R^{+}$, it holds
$$\int_{\S_{+}(x)}f(\sigma)\d \sigma=\int_{B_{d-1}}f(z,\sqrt{1-|z|^{2}})\frac{\d z}{\sqrt{1-|z|^{2}}},$$
where $B_{d-1}=\{z \in \R^{d-1}\;;\;|z| < 1\}.$ Moreover, according to Lemma \ref{lem:sard},
$$\int_{\S_{+}(x)}f(\sigma)\d\sigma=\int_{\S_{+}(x)\setminus\mathcal{G}_{x}}f(\sigma)\d\sigma,$$
while
$$\int_{B_{d-1}}f(z,\sqrt{1-|z|^{2}})\frac{\d z}{\sqrt{1-|z|^{2}}}=\int_{B_{d-1}\setminus \mathbf{G}_{x}}f(z,\sqrt{1-|z|^{2}})\frac{\d z}{\sqrt{1-|z|^{2}}}.$$
Then, for the special choice of $f(\sigma)=g(\sigma)\,|\sigma \cdot n(x)|$,
we get
\begin{equation}\label{eq:G+x}
\int_{\S_{+}(x)}|\sigma\cdot n(x)|\,g(\sigma)\d\sigma=\\
\int_{B_{d-1}} g(\sigma(z)))\ind_{B_{d-1}\setminus\mathbf{G}_{x}}(z)\d z\end{equation}
with $\sigma(z) =(z_{1},\ldots,z_{d-1},\sqrt{1-|z|^{2}})$ for $|z| <1.$ 

Notice that, with the notations of Lemma \ref{lem:jaco}, one has $B_{d-1}\setminus \mathbf{G}_{x}=\mathcal{O}_{x}$. Still using the notations of Lemma \ref{lem:jaco}, we introduce the mapping$$\mathfrak{p}\::\:z \in \mathcal{O}_{x} \mapsto y=\mathfrak{p}(z)=x-\tau_{-}(x,\sigma(z))\sigma(z)$$
which is such that $\mathfrak{p}(\mathcal{O}_{x})=\Sigma_{+}(x)$.
With this change of variable, notice that
$$\tau_{-}(x,\sigma(z))=|x-y|,$$
since $\sigma(z) \in \S^{d-1}$, and therefore
$$\sigma(z)=\frac{x-\mathfrak{p}(z)}{\tau_{-}(x,\sigma(z))}=\frac{x-y}{|x-y|}.$$
According to \cite[Lemma 5.2.11 \& Theorem 5.2.16, pp. 128-131]{stroock}, from this parametrization, the Lebesgue surface measure $\pi(\d y)$ on $\partial\Omega$ is given by
$$\pi(\d y)=\sqrt{\mathrm{det}\Bigg(\Bigg(\left\langle \frac{\partial \mathfrak{p}(z)}{\partial z_{i}},\frac{\partial \mathfrak{p}(z)}{\partial z_{j}}\right\rangle\Bigg)_{1\leq i,j\leq d-1}\Bigg)}\d z_{1}\ldots \d z_{d-1}=\sqrt{\mathrm{det}(\mathfrak{P}(z))}\d z,$$
from which we deduce directly that
$$\int_{B_{d-1}} \ind_{B_{d-1}\setminus\mathbf{G}_{x}}(z)g(\sigma(z))\d z=\int_{\partial\Omega}g\left(\frac{x-y}{|x-y|}\right)\mathcal{J}(x,y)\pi(\d y),$$
where
$$\mathcal{J}(x,y)=\frac{1}{\sqrt{\mathrm{det}(\mathfrak{P}(z))}}\ind_{\mathcal{O}_{x}}(z)$$
has to be expressed in terms of $x$ and $y$. Using Lemma \ref{lem:jaco}, one has
$$\mathcal{J}(x,y)=\left|\frac{(\sigma(z)\cdot n(\mathfrak{p}(z)))\sigma_{d}(z)}{\tau_{-}(x,\sigma(z))^{d-1}}\right|\ind_{\mathcal{O}_{x}}(z)$$
with, as mentioned, $\tau_{-}(x,\sigma(z))=|x-y|,$ $\mathfrak{p}(z)=y$ and $\sigma(z)=\frac{x-y}{|x-y|}$. Notice that 
$$\sigma_{d}(z)=\sigma(z)\cdot n(z)=\frac{(x-y)}{|x-y|}\cdot n(x)$$ 
which gives the desired expression \eqref{eq:Jxy} of $\mathcal{J}(x,y)$. Now, if $g(\sigma)=\varphi(\tau_{-}(x,\sigma))G(x-\tau_{-}(x,\sigma)\sigma)$, we get \eqref{eq:CofV}.\end{proof}

We end this section with a useful technical result.
\begin{lemme}\label{lem:convede}
 Assume that $\partial\Omega$ satisfies Assumption \ref{hypO}. Then
$$\lim_{\delta\to0^{+}}\sup_{y \in \partial\Omega}\int_{|x-y| \leq \delta}\mathcal{J}(x,y)\pi(\d x)=0.$$
\end{lemme}
\begin{proof} First, one notices that the straightforward estimate 
\begin{equation}\label{eq:boundJ}
\mathcal{J}(x,y) \leq |x-y|^{1-d}\end{equation}
is not strong enough to derive the result (see the subsequent Lemma \ref{lem:1} for more details on this point). We need to proceed in a different way. Observe that, thanks to Remark \ref{eq:sym}, for any $y \in \partial\Omega$, it holds
$$\int_{|x-y|\leq\delta}\mathcal{J}(x,y)\pi(\d x)=\int_{|x-y|\leq\delta}\mathcal{J}(y,x)\pi(\d x)=\int_{\S_{+}(y)}\mathbf{1}_{(0,\delta]}(\tau_{-}(y,\sigma))\d\sigma,$$
where we used Lemma \ref{lem:ChVa} with the functions $\varphi(r)=\mathbf{1}_{[0,\delta]}(r)$ and $G\equiv 1$. Clearly, for any \emph{fixed} $y \in \partial\Omega$
\begin{equation}\label{eq:lim}
\lim_{\delta\to0^{+}}\int_{\S_{+}(y)}\mathbf{1}_{(0,\delta]}(\tau_{-}(y,\sigma))\d\sigma=0\end{equation}
according to the dominated   convergence theorem, so one needs  to check that the convergence \eqref{eq:lim} is \emph{uniform} with respect to $y\in \partial\Omega$. Assume it is not the case so that there exist $c >0$, a sequence $\{y_{n}\}_{n} \subset\partial\Omega$ and a sequence $(\delta_{n})_{n} \subset (0,\infty)$ converging to $0$ such that
$$\int_{\S_{+}(y_{n})}\mathbf{1}_{(0,\delta_{n}]}(\tau_{-}(y_{n},\sigma))\d\sigma \geq c \qquad \forall n \in \N.$$
First, one deduces from Fatou's lemma that
\begin{multline}\label{fatou}
0 < c \leq \limsup_{n}\int_{\S^{d-1}}\ind_{\S_{+}(y_{n})}(\sigma)\ind_{(0,\delta_{n}]}(\tau_{-}(y_{n},\sigma))\d\sigma\\
 \leq \int_{\S^{d-1}}\limsup_{n}\ind_{\S_{+}(y_{n})}(\sigma)\ind_{(0,\delta_{n}]}(\tau_{-}(y_{n},\sigma))\d\sigma.\end{multline}
Of course, there is no loss of generality in assuming that $\{y_{n}\}_{n}$ converges to some $y \in \partial\Omega.$ Now,  $\partial\Omega$ being of class $\mathscr{C}^{1}$, it holds that $\lim_{n}n(y_{n})=n(y)$ and therefore there is $n_{0}\in \N$ such that
$$\S_{+}(y) \subset \S_{+}(y_{n}) \qquad \forall n \geq n_{0}.$$
Moreover, for $\sigma \in \S_{+}(y)$, $\tau_{+}(y,\sigma) >0$, and, since $\tau_{+}$ is lower-semicontinuous on $\partial\Omega\times V$ (see \cite[Lemma 1.5]{voigt}), it holds that
$$\liminf_{n\to\infty}\tau_{-}(y_{n},\sigma) \geq \tau_{-}(y,\sigma) >0 \qquad \forall \sigma \in \S_{+}(y).$$
As a consequence, one has
$$\limsup_{n\to\infty}\ind_{(0,\delta]}(\tau_{-}(y_{n},\sigma))=0 \qquad \forall \sigma \in \S_{+}(y).$$
Since $\{\sigma \in \S^{d-1}\;;\;\sigma \cdot n(x)=0\}$ is a subset of $\S^{d-1}$ of zero Lebesgue measure, we see that
$$\limsup_{n}\ind_{\S_{+}(y_{n})}(\sigma)\ind_{(0,\delta]}(\tau_{-}(y_{n},\sigma))=0$$
for almost every $\sigma \in \S^{d-1}$ which contradicts \eqref{fatou}. This proves the result.
\end{proof}

Whenever the boundary $\partial\Omega$ is more regular than merely $\mathscr{C}^{1}$ one can strengthen the estimate \eqref{eq:boundJ}. Namely, one has the following result (see \cite[Lemma 2]{guo03} for a similar result for $\partial\Omega$ of class $\mathscr{C}^{2}$)
\begin{lemme}\label{lem:1}
Assume that $\partial\Omega$ is of class $\mathscr{C}^{1,\alpha}$, $\alpha \in (0,1)$ then, there exists a positive constant $C_{\Omega} >0$ such that
$$\left|(x-y)\,\cdot \,n(x)\right| \leq C_{\Omega}\,|x-y|^{1+\alpha}, \qquad \forall x,y \in \partial \Omega.$$
Consequently, with the notations of Lemma \ref{lem:ChVa}, there is a positive constant $C >0$ such that
$$\mathcal{J}(x,y) \leq \frac{C}{|x-y|^{d-1-2\alpha}}, \qquad \forall x,y \in \partial \Omega, x\neq y.$$
\end{lemme}
\begin{proof} The intuition behind the estimate is that, from the smoothness of $\partial \Omega$,  for any $x\neq y \in \partial\Omega$, if $\bm{e}_{x}(y)=\frac{x-y}{|x-y|}$ denotes the unit vector with direction $x-y$, then
$$\lim_{y \to x}\bm{e}_{x}(y)\,\cdot \,n(x)=0,$$
since $\bm{e}_{x}(y)$ tends to be tangent to $\partial \Omega.$ Then $( x-y)\,\cdot \,n(x)$ is of the order $|x-y|^{1+\alpha}$ for $x \simeq y$. Let us make this rigorous. For a given $x \in \partial \Omega$, one can find a local parametrization of a neighbourhood $\mathcal{O}_{x} \subset \partial\Omega$, containing $x$ as 
$$\mathcal{O}_{x}=\left\{(u,\Phi(u))\;;\;u \in U\right\},$$
where $U$ is an open subset of $\R^{d-1}$ and $\Phi\,:\,U \to \mathcal{O}_{x}$ is a $\mathscr{C}^{1,\alpha}$-diffeomorphism. Denoting by $|\cdot|$ the euclidian norm of $\R^{d-1}$ and by $\nabla \Phi$ the gradient of $\Phi$ (in $\R^{d-1}$), we get, with $x=(u_{0},\Phi(u_{0})) \in \partial \Omega$, 
$$n(x)=\frac{1}{\sqrt{1+|\nabla \Phi(u_{0})|^{2}}}\left(\nabla \Phi(u_{0}),-1\right),$$
so that
$$(x-y)\,\cdot n(x)=\frac{1}{\sqrt{1+|\nabla \Phi(u_{0})|^{2}}} \bigg(\langle u-u_{0}\,;\,\nabla \Phi(u_{0})\rangle_{d-1}-\left(\Phi(u)-\Phi(u_{0})\right)\bigg),$$
where $\langle \cdot,\cdot\rangle_{d-1}$ is the inner product in $\R^{d-1}.$ Since, 
\begin{equation*}\begin{split}
\Phi(u)-\Phi(u_{0})&=\int_{0}^{1}\langle u-u_{0},\nabla \Phi(tu+(1-t)u_{0}) \rangle_{d-1}\d t\\
&=\langle u-u_{0}\,;\,\nabla \Phi(u_{0})\rangle_{d-1}
+\langle u-u_{0},\int_{0}^{1}\left(\nabla \Phi(tu+(1-t)u_{0})-\nabla \Phi(u_{0})\right)\d t\rangle_{d-1}\,,
\end{split}\end{equation*}
we see that
$$\left|\Phi(u)-\Phi(u_{0}-\langle u-u_{0}\,;\,\nabla \Phi(u_{0})\rangle_{d-1}\right| \leq |u-u_{0}|\int_{0}^{1}\left|\nabla \Phi(tu+(1-t)u_{0})-\nabla \Phi(u_{0})\right|\d t\,.$$
Since $\nabla \Phi \in \mathscr{C}^{0,\alpha}$, denoting  by
$$C_{\Phi}=\sup_{u_{1},u_{2}\in U}\frac{|\nabla \Phi(u_{1})-\nabla \Phi(u_{2})|}{|u_{1}-u_{2}|^{\alpha}}$$ the H\"older semi-norm of $\nabla \Phi$, we get 
$$\left|\Phi(u)-\Phi(u_{0}-\langle u-u_{0}\,;\,\nabla \Phi(u_{0})\rangle_{d-1}\right| \leq C_{\Phi}|u-u_{0}|^{1+\alpha}\int_{0}^{1}t^{\alpha}\d t=\frac{C_{\Phi}}{\alpha+1}|u-u_{0}|^{1+\alpha}, \qquad u \in U.$$
We deduce then that
$$|(x-y)\,\cdot n(x)| \leq \frac{C_{\Phi}}{\alpha+1}|u-u_{0}|^{1+\alpha} \leq \frac{C_{\Phi}}{\alpha+1}|x-y|^{1+\alpha}, \qquad \text{ for } u \simeq u_{0}.$$
 Since $\partial \Omega$ is compact and $\Omega$ bounded, this easily yields the conclusion. Now, from \eqref{eq:Jxy}, we get
$$\mathcal{J}(x,y) \leq C_{\Omega}^{2}|x-y|^{d-1-2\alpha}  \qquad \forall y \in \Gamma_{+}(x),$$
which achieves the proof.
\end{proof}
\begin{nb} The above result is still true for $\partial\Omega$ of class $\mathscr{C}^{2}$ which would correspond to $\alpha=1$. In such a case, one has
$$|(x-y)\,\cdot n(x)| \leq C_{\Omega}|x-y|^{2}, \qquad x,y \in \partial\Omega$$
and 
$$\mathcal{J}(x,y) \leq C_{\Omega}^{2}|x-y|^{d-3}, \qquad \forall y \in \Gamma_{+}(x).$$
Notice in particular, that, in such a case one has $\mathcal{J}(x,y) \leq C_{\Omega}^{2}D^{3-d}$ for $d=2,3$, i.e. $\mathcal{J}$ is bounded. In such a case, applying Lemma \ref{lem:ChVa} with $\varphi=1$, one sees that in dimension $d=2,3$, the boundedness of $\mathcal{J}(x,y)$ implies  the existence of a positive constant $C >0$ such that
$$\int_{\S_{+}(x)}G(x-\tau_{-}(x,\sigma)\sigma)\,|\sigma \cdot n(x)|\d\sigma \leq C \int_{\Sigma_{+}(x)}G(y)\pi(\d y) \qquad \forall G \geq 0.$$
This  easily allows us to recover  \cite[Lemma 2.3, Eq. (2.6)]{EGKM}. 
\end{nb}

\subsection{Practical criterion ensuring Assumption \ref{hypH} \textit{4)}}

We provide here some practical Assumptions under which \eqref{eq:power} will hold.  
 We recall first the following  generalization of the polar decomposition theorem (see \cite[Lemma 6.13, p.113]{voigt}):
\begin{lemme}\phantomsection\label{lem:polar}
Let $\bm{m}_{0}$ be the image of the measure $\bm{m}$ under the transformation $v \in \R^{d} \mapsto |v| \in [0,\infty),$ i.e. $\bm{m}_{0}(I)=\bm{m}\left(\{v \in \R^{d}\;;\;|v| \in I\}\right)$ for any Borel subset $I \subset \R^{+}.$ Then, for any $\psi \in L^{1}(\R^{d},\bm{m})$ it holds
$$\int_{\R^{d}}\psi(v)\bm{m}(\d v)=\frac{1}{|\S^{d-1}|}\int_{0}^{\infty}\bm{m}_{0}(\d\varrho)\int_{\S^{d-1}}\psi(\varrho\,\sigma)\d\sigma$$
where $\d\sigma$ denotes the Lebesgue measure on $\S^{d-1}$ with surface $|\S^{d-1}|.$ 
\end{lemme}

We can deduce from the above change of variables the following useful expression for $\H\mathsf{M}_{\l}\H$. Recall that $\H$ is assumed to be given by \eqref{eq:Hhkernel}
\begin{propo}\label{lemHLH} For any $\l \in \overline{\C}_{+}$, it holds
\begin{equation}\label{eq:hmH}
\mathsf{HM_{\l}H}\varphi(x,v)=\int_{\Gamma_{+}}\mathscr{J}_{\l}(x,v,y,w)\varphi(y,w)\,|w\cdot n(y)|\bm{m}(\d w)\pi(\d y)\end{equation}
where
\begin{multline}\label{eq:Jlam}
\mathscr{J}_{\l}(x,v,y,w)=\mathcal{J}(x,y)\int_{0}^{\infty}\varrho\,\bm{k}\left(x,v,\varrho\frac{x-y}{|x-y|}\right)\times\\
\times\bm{k}\left(y,\varrho\frac{x-y}{|x-y|},|w|\right)\exp\left(-\lambda\frac{|x-y|}{\varrho}\right)\frac{\bm{m}_{0}(\d\varrho)}{|\S^{d-1}|}\end{multline}
for any $(x,v) \in \Gamma_{-},$ $(y,w) \in \Gamma_{+}.$
\end{propo}
\begin{proof} The proof follows by direct inspection. Indeed,  for any $\varphi \in \lp$ and $(x,v) \in \Gamma_{-}$: 
\begin{multline*}
\mathsf{H}\mathsf{M}_{\lambda}\mathsf{H}\varphi(x,v)=\int_{\Gamma_{+}(x)}\exp(-\lambda\,\tau_{-}(x,v'))\bm{k}(x,v,v')|v'\cdot n(x)|\bm{m}(\d v')\\
\int_{w \cdot n(x-\tau_{-}(x,v')v')>0}\bm{k}(x-\tau_{-}(x,v')v',v',w)\varphi(x-\tau_{-}(x,v')v',w)|w\cdot n(x-\tau_{-}(x,v')v')|\bm{m}(\d w).\end{multline*}
Then, using polar coordinates $v'=\varrho\,\sigma$ and the fact that $\tau_{-}(x,v')=\varrho^{-1}\tau_{-}(x,\sigma)$, we can use Proposition \ref{lem:ChVa}, Eq. \eqref{eq:CofV} (with $h(s)=\exp(-\lambda\,\varrho^{-1}s))$ and Lemma \ref{lem:polar} to get
\begin{multline*}
\mathsf{H}\mathsf{M}_{\lambda}\mathsf{H}\varphi(x,v)= \frac{1}{|\S^{d-1}|} 
\int_{\Sigma_{+}(x)} \mathcal{J}(x,y)\pi(\d y)\int_{\Gamma_{+}(y)}\varphi(y,w)|w\cdot n(y)|\bm{m}(\d w)\times
\\
\times\int_{0}^{\infty}\varrho\,\bm{k}\left(x,v,\varrho\frac{x-y}{|x-y|}\right) \exp\left(-\lambda\frac{|x-y|}{\varrho}\right)\pi(\d y)
\bm{k}\left(y,\varrho\frac{x-y}{|x-y|},|w|\right) \bm{m}_{0}(\d\varrho)
\end{multline*}
which gives the result.\end{proof}
\begin{nb}\phantomsection\label{nb:JlG} In the special case of Example \ref{exe:gener}, one checks readily that
\begin{multline*}
\mathscr{J}_{\l}(x,v,y,w)=\gamma(x)^{-1}\G(x,v)\mathcal{J}(x,y)\\
\times\gamma^{-1}(y) \int_{0}^{\infty}\varrho\, \G(y,\varrho)\exp\left(-\lambda\frac{|x-y|}{\varrho}\right)\frac{\bm{m}_{0}(\d\varrho)}{|\S^{d-1}|}\end{multline*}
for any $(x,v) \in \Gamma_{-}$, $y \in \partial\Omega.$ In particular, $\mathscr{J}_{\l}(x,v,y,w)$ does not depend on $w.$\end{nb}
 
Thanks to this representation of $\H\mathsf{M}_{\l}\H$, we can make the following set of assumptions ensuring \eqref{eq:power} to hold true.

\begin{hyp}\label{hypr0} Let $\H \in \mathscr{B}(\lp,\lm)$ be given by \eqref{eq:Hhkernel}  where the kernel $\bm{k}(x,v,v')$ is nonnegative, measurable and satisfies \eqref{eq:normalise}. Assume that $\bm{m}_{0}$ is given by \footnote{Notice that this amounts to a measure $\bm{m}$ which is absolutely continuous with respect to the Lebesgue measure over $\R^{d}$, namely $\bm{m}(\d v)=\varpi(|v|)\d v$.}
$$\bm{m}_{0}(\d \varrho)=|\S^{d-1}|\varrho^{d-1}\varpi(\varrho)\d\varrho$$
for some positive $\varpi(\varrho) >0$ with
\begin{equation}\label{eq:limkvarpi}
\lim_{\varrho\to\infty}\varrho^{d+2}\bm{k}\left(y,v,\varrho\frac{x-y}{|x-y|}\right)\bm{k}\left(y,\varrho\frac{x-y}{|x-y|},w\right)\varpi(\varrho)=0, \qquad \forall (x,v)  \in \Gamma_{-}, (y,w) \in \Gamma_{+};\end{equation}
\begin{equation}\label{eq:inftyintK}%
\underset{\sigma\in \S^{d-1}}{\sup_{(y,w) \in \Gamma_{+}}}
\int_{0}^{\infty}\varrho^{d+1}\left[
\varrho\,\varpi(\varrho)\left|\nabla_{2}\bm{k}(y,\varrho\sigma,w)\right| + \bm{k}(y,\varrho\sigma,w)\left(\varrho\, \left|\varpi'(\varrho)\right| + \varpi(\varrho)\right)\right]\d\varrho < \infty;\end{equation}
 and \begin{equation}\label{eq:inftyderK}
\sup_{x \in \partial\Omega}\underset{\sigma\in \S^{d-1}}{\sup_{(y,w) \in \Gamma_{+}}}\int_{{0}}^{\infty}\varrho^{d+2}\varpi(\varrho)\bm{k}(y,\varrho\sigma,w)\d\varrho\int_{\Gamma_{-}(x)}\left|\nabla_{3}\bm{k}(x,v,\varrho\sigma)\right|\bm{\mu}_{x}(\d v) < \infty.
\end{equation}
where we adopted the notations
$$\nabla_{2}\bm{k}(x,v,w)=\nabla_{v}\bm{k}(x,v,w), \qquad \nabla_{3}\bm{k}(x,v,w)=\nabla_{w}\bm{k}(x,v,w),$$
for $(x,v) \in \Gamma_{-},$ $w \in \Gamma_{+}(x)$. 
\end{hyp}

We can then prove the following:
\begin{lemme}\label{lem:Jl} Under Assumption \ref{hypr0} and if $\partial\Omega$ is of class $\mathcal{C}^{1,\alpha}$ with $\alpha >\frac{1}{2}$, then for any $\lambda \in \overline{\C}_{+}$, $\l \neq 0$, it holds
$$\sup_{(y,w) \in \Gamma_{+}}\int_{\Gamma_{+}}\left|\mathscr{J}_{\l}(x,v,y,w)\right|\d\mu_{-}(x,v) \leq \frac{C}{|\l|}$$
for some positive $C >0$ depending only on $\bm{k}$ and $\partial\Omega$.\end{lemme}
\begin{proof}  
From \eqref{eq:Jlam} and Lemma \ref{lem:1}, one has for all $(x,v) \in \Gamma_{-},$ $(y,w) \in \Gamma_{+}$
\begin{multline*}
\left|\mathscr{J}_{\l}(x,v,y,w)\right| \leq \frac{C_{\Omega}}{|x-y|^{d-1-2\alpha}}\bigg|\int_{{0}}^{\infty}\varrho^{d}\, \bm{k}\left(x,v,\varrho\frac{x-y}{|x-y|}\right)\times\\
\times\bm{k}\left(y,\varrho\frac{x-y}{|x-y|},w\right)\exp\left(-\lambda\frac{|x-y|}{\varrho}\right)\varpi(\varrho)\d\varrho\bigg|.\end{multline*}
for some positive constant $C_{\Omega}$. We observe that the last integral can be written as:
$$\frac{1}{\l|x-y|}\int_{{0}}^{\infty}\varrho^{d+2}\, \bm{k}\left(x,v,\varrho\frac{x-y}{|x-y|}\right)
\bm{k}\left(y,\varrho\frac{x-y}{|x-y|},w\right)\underset{=\frac{\d}{\d\varrho}\exp\left(-\l|x-y|\varrho^{-1}\right)}{\underbrace{\left[\frac{\l|x-y|}{\varrho^{2}}\exp\left(-\lambda\frac{|x-y|}{\varrho}\right)\right]}}\varpi(\varrho)\d\varrho$$
which, after integration by parts and using \eqref{eq:limkvarpi}, is equal to
$$
-\frac{1}{\l|x-y|}\int_{0}^{\infty}\exp\left(-\l|x-y|\varrho^{-1}\right)
\frac{\d}{\d\varrho}\left[\varrho^{d+2}\, \bm{k}\left(x,v,\varrho\frac{x-y}{|x-y|}\right)
\bm{k}\left(y,\varrho\frac{x-y}{|x-y|},|w|\right)\right] \d\varrho.$$
This results in the following estimate for the kernel $\mathscr{J}_{\l}(x,v,y,w)$:
$$
\left|\mathscr{J}_{\l}(x,v,y,w)\right| \leq \frac{C_{\Omega}}{|\l|\,|x-y|^{d-2\alpha}}|I(\l,x,y,v,w)|$$
with 
$$I(\l,x,v,y,w)=\int_{0}^{\infty}\exp\left(-\l|x-y|\varrho^{-1}\right)
\frac{\d}{\d\varrho}\left[\varrho^{d+2}\, \bm{k}\left(x,v,\varrho\frac{x-y}{|x-y|}\right)
\bm{k}\left(y,\varrho\frac{x-y}{|x-y|},|w|\right)\right] \d\varrho$$
for any $\l \neq 0,$ $(x,v) \in \Gamma_{-},$ $(y,w) \in \Gamma_{+}.$ 
Distributing the derivative with respect to $\varrho$ thanks to Leibniz rule, one writes 
$$I(\l,x,v,y,w)=\sum_{j=1}^{4}I_{j}(\l,x,v,y,w)$$
with  
\begin{equation*}\begin{cases}
I_{1}(\l,x,v,y,w)&=\ds \, \int_{{0}}^{\infty}\sigma_{x,y}\cdot\nabla_{2}\bm{k}\left(y,\varrho\sigma_{x,y},w\right)\varrho^{d+2}\varpi(\varrho)\bm{k}\left(x,v,\varrho\sigma_{x,y}\right)\exp\left(-\l|x-y|\varrho^{-1}\right)\d\varrho\\
\\
I_{2}(\l,x,v,y,w)&=\ds\int_{{0}}^{\infty}\varrho^{d+2}\varpi(\varrho)\sigma_{x,y}\cdot\nabla_{3}\bm{k}\left(x,v,\varrho\sigma_{x,y}\right)\, \bm{k}\left(y,\varrho\sigma_{x,y},w\right)\exp\left(-\l|x-y|\varrho^{-1}\right)\d\varrho\\
\\
I_{3}(\l,x,v,y,w)&=\ds\int_{{0}}^{\infty}\varrho^{d+2}\varpi'(\varrho)\bm{k}(x,v,\varrho\sigma_{x,y})\,\bm{k}(y,\varrho\sigma_{x,y},w)\exp\left(-\l|x-y|\varrho^{-1}\right)\d\varrho\\
\\
I_{4}(\l,x,v,y,w)&=(d+2)\ds\int_{{0}}^{\infty}\varrho^{d+1}\varpi(\varrho)\bm{k}(x,v,\varrho\sigma_{x,y})\, \bm{k}(y,\varrho\sigma_{x,y},w)\exp\left(-\l|x-y|\varrho^{-1}\right)\d\varrho\end{cases}\end{equation*}
where we adopt the short-hand notation $\sigma_{x,y}=\frac{x-y}{|x-y|},$ $(x\neq y)$. Using the normalisation condition \eqref{eq:normalise}, one has
\begin{multline*}
\int_{\Gamma_{-}(x)} \left|I_{1}(\l,x,v,y,w)\right||v \cdot n(x)|\bm{m}(\d v) 
\leq \ds\int_{{0}}^{\infty}\varrho^{d+2}\varpi(\varrho)\,\left|\sigma_{x,y}\cdot\nabla_{2}\bm{k}\left(y,\varrho\sigma_{x,y},w\right)\right|\d\varrho\\
\leq \int_{0}^{\infty}\varrho^{d+2}\varpi(\varrho)\,\left|\nabla_{2}\bm{k}\left(y,\varrho\sigma_{x,y},w\right)\right|\d\varrho
.
\end{multline*}
Thus, assumption \eqref{eq:inftyintK} yields
$$\sup_{(y,w)\in \Gamma_{+}}\int_{\Gamma_{-}(x)} \left|I_{1}(\l,x,v,y,w)\right||v \cdot n(x)|\bm{m}(\d v)| \leq C.$$
In the same way, one sees easily that \eqref{eq:inftyintK} implies that
$$\sup_{(y,w)\in \Gamma_{+}}\int_{\Gamma_{-}(x)} \left(\left|I_{3}(\l,x,v,y,w)\right| + \left|I_{4}(\l,x,v,y,w)\right|\right)\,| v \cdot n(x)|\bm{m}(\d v) \leq C\,.$$
Finally, one checks easily that \eqref{eq:inftyderK} implies
$$\sup_{x \in \partial\Omega}\sup_{(y,w) \in \Gamma_{+}}\int_{\Gamma_{-}(x)} \left|I_{2}(\l,x,v,y,w)\right||v \cdot n(x)|\bm{m}(\d v)| \leq C\,.$$
Combining all these estimates, we finally obtain that there exists some positive constant $C$  such that
$$
\int_{\Gamma_{-}(x)}\left|\mathscr{J}_{\l}(x,v,y,w)\right|\,|v\cdot n(x)|\bm{m}(\d v) 
\leq \frac{C}{|\l||x-y|^{d-2\alpha}}  \qquad \forall x \in \partial\Omega, \qquad \forall (y,w) \in \Gamma_{+}.$$
Since, for $\alpha > \frac{1}{2}$, 
$$\sup_{y\in \partial\Omega}\int_{\partial\Omega}\frac{\pi(\d x)}{|x-y|^{d-2\alpha}} < \infty$$
we get the desired result.
\end{proof}

The above, combined with Proposition \ref{lemHLH} yields the following
\begin{propo}\label{lem:norm2} Assume that Assumption \ref{hypr0} are in force and $\partial\Omega$ is of class $\mathcal{C}^{1,\alpha}$ with $\alpha >\frac{1}{2}$. There exists a positive constant $C$ such that 
$$\left\|(\mathsf{M}_{\l}\H)^{2}\right\|_{\mathscr{B}(\lp)} \leq \frac{C}{|\l|} $$
holds for any $\l \in \overline{\C}_{+}$, $\l\neq0$. In particular, \eqref{eq:power} holds true with $\mathsf{p}=4$.
\end{propo}
\begin{proof} It is clear from Proposition \ref{lemHLH} that, for any $\psi \in \lp$,
\begin{multline*}
\|(\mathsf{M}_{\l}\H)^{2}\psi\|_{\lp} \leq \|\mathsf{M}_{\l}\|_{\mathscr{B}(\lm,\lp)}\left\|\H\mathsf{M}_{\l}\H\psi\right\|_{\lm}\\
\leq \|\mathsf{M}_{\l}\|_{\mathscr{B}(\lm,\lp)}\int_{\Gamma_{+}}|\psi(y,w)|\d\mu_{+}(y,w)\int_{\Gamma_{-}}\left|\mathscr{J}_{\l}(x,v,y,w)\right|\d\mu_{-}(x,v)\end{multline*}
so that, using that $\|\mathsf{M}_{\l}\|_{\mathscr{B}(\lm,\lp)} \leq 1$ we get
$$\|(\mathsf{M}_{\l}\H)^{2}\psi\|_{\lp} \leq  \sup_{(y,w) \in \Gamma_{+}}\int_{\Gamma_{-}}\left|\mathscr{J}_{\l}(x,v,y,w)\right|\d\mu_{-}(x,v)$$
and we conclude then with Lemma \ref{lem:Jl}. Since then, for any $\e \geq0$ and $\eta >0$,
$$\left\|\left(\mathsf{M}_{\e+i\eta}\H\right)^{4}\right\|_{\mathscr{B}(\lp)}\leq \frac{C^{2}}{\left|\e+i\eta\right|^{2}} \leq \frac{C^{2}}{|\eta|^{2}}$$
we deduce that
$$\sup_{\e >0}\int_{|\eta| >1}\left\|\left(\mathsf{M}_{\e+i\eta}\H\right)^{4}\right\|_{\mathscr{B}(\lp)}\d\eta \leq C^{2}\int_{|\eta|>1}\frac{\d\eta}{\eta^{2}} < \infty$$
which proves \eqref{eq:power} with $\mathsf{p}=4.$
\end{proof}

\subsection{Examples}\label{sec:exam} We revisit here the examples of practical application in the kinetic theory of gases introduced in the Introduction. We focus here on the case on which
$$V=\R^{d}, \qquad \bm{m}(\d v)=\d v$$
for simplicity but of course the case of measure $\bm{m}$ absolutely continuous with respect to the Lebesgue measure $\bm{m}(\d v)=\varpi(|v|)\d v$ is easily deduced from our analysis.

We give full details for the Example \ref{exe:maxw} in the Introduction which is the most studied model in the framework we are dealing with here (see \cite{aoki,bernou1,kim}.

We recall that, here
$$\bm{k}(x,v,v')=\gamma^{-1}(x)\mathcal{M}_{\theta(x)}(v)$$
with
$$\mathcal{M}_{\theta}(v)=(2\pi\theta)^{-d/2}\exp\left(-\frac{|v|^{2}}{2\theta}\right), \qquad x \in \partial \Omega, \:\:v \in \R^{d}.$$ 
and
$$\gamma(x)=\bm{\kappa}_{d}\sqrt{\theta(x)}\int_{\R^{d}}|w|\M_{1}(w)\d w, \qquad x \in \partial\Omega$$ 
for some positive constant $\bm{\kappa}_{d}$ depending only on the dimension ensuring in particular \eqref{eq:normalise}. We assume here that the temperature mapping $x \in \partial\Omega \mapsto \theta(x)$ is bounded away from zero and continuous and denote
$$\theta_{0}:=\inf_{x\in \partial\Omega}\theta(x) >0.$$
In this case, one sees that $\gamma(x) \geq c_{d}\sqrt{\theta_{0}}$ for some explicit $c_{d} >0$ and
$$\bm{k}(x,v,v') \leq c_{d}^{-1}\mathcal{M}_{\theta_{0}}(v).$$
With this, it is easy to deduce that the boundary operator $\H$ associated to $\bm{k}(\cdot,\cdot,\cdot)$ is dominated by a rank-one operator on $\lp$ and as such is a \emph{regular diffuse boundary operator} (see \cite[Remark 3.6]{LMR}). In particular, from \cite[Theorem 5.1]{LMR}, 
$$\H\mathsf{M}_{0}\H \in \mathscr{B}(\lp,\lm) \qquad \text{ is weakly compact}$$
which means that Assumption \ref{hypH} \textit{2)} is met. Moreover, since $\bm{k}(x,v,v') > 0$ for any $(x,v) \in \Gamma_{+}$ and any $v' \in \Gamma_{-}(x)$, we deduce from \cite[Remark 4.5]{LMR} that $\mathsf{M}_{0}\H$ is irreducible, i.e. Assumption \ref{hypH} \textit{1)} is met. For this model, recalling here that the measure $\bm{m}_{0}$ over $\R^{+}$ is given by
$$\bm{m}_{0}(\d \varrho)=|\S^{d-1}|\varrho^{d-1}\d \varrho$$
we see easily that $\bm{k}(x,v,v')$ satisfy Assumptions \ref{hypr0} (with $\varpi \equiv 1$). Indeed, notice that $\bm{k}(x,v,v')$ is independent of $v'$ and depends only on $|v|$ and we denote simply
$$\bm{k}(x,v,v')=\bm{G}(x,|v|).$$
Then, \eqref{eq:limkvarpi} simply reads
$$\lim_{\varrho\to\infty}\varrho^{d+2}\bm{G}\left(y,\left|\varrho\frac{x-y}{|x-y|}\right|\right)=\lim_{\varrho\to\infty}\varrho^{d+2}\bm{G}\left(y,\varrho\right)=0, \qquad \forall (x,v)  \in \Gamma_{-}, (y,w) \in \Gamma_{+}$$
which obviously hold true since $\bm{G}(y,\varrho)=\frac{1}{\gamma(y)(2\pi\theta(y))^{\frac{d}{2}}}\exp\left(-\frac{\varrho^{2}}{2\theta(y)}\right).$ In the same way, \eqref{eq:inftyintK} reads simply
$$\underset{\sigma\in \S^{d-1}}{\sup_{(y,w) \in \Gamma_{+}}}
\int_{0}^{\infty}\varrho^{d+1}\left[\frac{\varrho^{2}}{\theta(y)}\bm{G}(y,\varrho) + \bm{G}(y,\varrho)\right]\d\varrho < \infty;$$
since $\nabla_{2}\bm{k}(y,v,w)=-\frac{v}{\theta(y)\gamma(y)}\mathcal{M}_{\theta(y)}(v)=-\frac{v}{\theta(y)}\bm{G}(y,|v|)$. Since $\gamma(y)$ and $\theta(\cdot)$ are bounded from below, the result follows easily. Finally, \eqref{eq:inftyderK} is obviously satisfied since $\nabla_{3}\bm{k}(x,v,w)=0$. Therefore, Assumption \ref{hypr0} is met and one deduces from Proposition \ref{lem:norm2} that, if $\partial\Omega$ is of class $\mathcal{C}^{1,\alpha}$ with $\alpha >\frac{1}{2}$, then \eqref{eq:power} holds true with $\mathsf{p}=3$ and Assumption \ref{hypH} \textit{4)} is met. Let us now determine $N_{\H}$ for which Assumption \ref{hypH} is met. Direct computations show that, for any integer $k$, 
\begin{equation}\label{eq:1H}\H \in \mathscr{B}(\lp,\Y_{k+1}^{-}) \Longleftrightarrow  \int_{0}^{\infty}|v|^{-k}\mathcal{M}_{\theta(x)} (v)\d v < \infty \qquad \forall x \in \partial\Omega \Longleftrightarrow  k < d,\end{equation}
which means that $N_{\H}=d-1$ (since $N_{\H}$ needs to be an integer). Therefore, for this model, we can reformulate Theorem \ref{theo:maindec} as follows, if  $\partial\Omega$ is of class $\mathcal{C}^{1,\alpha}$ with $\alpha > \frac{1}{2}$, then the following holds: for any $f \in \X_{d}$ 
$$\left\|U_{\H}(t)f-\varrho_{f}\Psi_{H}\right\|_{\X_{0}}=\mathbf{o}\left((1+t)^{d-1}\right).$$
Of course, one can be more explicit about the form of $\mathbf{o}\left((1+t)^{d-1}\right)$ and, in this case, one sees also that \eqref{eq:decay-power} holds true for any choice of $\mathsf{p}$ and $\beta >0$ since, according to Proposition \ref{lem:norm2}, for any $n > 2$
$$\left\|(\mathsf{M}_{\l}\H)^{n}\right\|_{\mathscr{B}(\lp)} \leq \frac{C_{n}}{|\l|^{\frac{n}{2}}} \qquad \forall |\l| >1.$$ 
Therefore, with the notations of Theorem \ref{theo:maindec}, the decay of 
$$\left\|\int_{-\infty}^{\infty}\exp\left(i\eta\,t\right)\mathsf{\Theta}_{f}(\eta)\d \eta\right\|_{\X_{0}}$$
can be made as close as desired from $\left(\omega_{f}\left(\frac{\pi}{t}\right)\right)$
where $\omega_{f}\::\:\R^{+} \to \R^{+}$ denotes the \emph{minimal modulus of continuity} of the uniformly continuous mapping $\mathsf{\Theta}_{f}$. Since moreover 
$$\H \in \mathscr{B}(\lp,\Y^{-}_{d+\alpha}) \qquad \forall \alpha \in (0,1)$$
thanks to \eqref{eq:1H}, one sees that, if the conjecture \eqref{eq:conj} holds true, then one would have
$$\left\|U_{\H}(t)f-\varrho_{f}\Psi_{H}\right\|_{\X_{0}}=\mathbf{O}\left(t^{-(d-\delta)}\right), \qquad \forall f \in \X_{d}, \qquad \forall \delta >0.$$
thanks to Remark \ref{nb:conj}.


 \medskip

We see that, here above, we fully exploit the fact that the model in Example \ref{exe:maxw}, the kernel was radially symmetric with respect to $v$ and did not depend on $w$. Such properties are still shared by the Example \ref{exe:gener} given in the Introduction and henceforth, it can be treated exactly along the same lines as those described here (under \emph{ad hoc} explicit condition ensuring Assumption \ref{hypr0} to hold true). For such a model, one sees also that $N_{\H}$ is the maximal $n \in \N$ for which  $\gamma(n,d) < \infty$ where, for all $s \geq0$,
$$\gamma(s,d):=\sup_{x \in \partial\Omega}\gamma^{-1}(x)\int_{v \cdot n(x)<0}|v|^{-s-1}\G(x,v)|v \cdot n(x)|\bm{m}(\d v)\in (0,\infty].$$
Clearly, the precise value of $N_{\H}$ depends on the explicit expression of $\bm{G}$.
 \appendix

\section{Proof of several technical results}\label{app:technA}

\subsection{Fine properties of $\mathsf{G}_{\l},$ $\mathsf{M}_{\l}\H$ and $\mathsf{\Xi}_{\l}\H$} We collect here all the technical details useful for the proof of Proposition \ref{propo:convK}. The main step is the following which corresponds to Prop. \ref{propo:convK} for $k=0$:
\begin{lemme}\label{lem:Meis} For any $f \in \X_{0}$, the limit
\begin{equation}\label{eq:unifGeis}
\lim_{\e\to0^{+}}\left\|\mathsf{G}_{\e+i\eta}f-\mathsf{G}_{i\eta}f\right\|_{\lp}=0\end{equation}
uniformly with respect to $\eta \in \R$. 
For any $\eta \in \R$, it holds
\begin{equation}\label{eq:Meis}
\left\|\mathsf{M}_{\varepsilon+i\eta}-\mathsf{M}_{i\eta}\right\|_{\mathscr{B}(\Y^{-}_{1},\lp)} \leq \varepsilon\,D, \qquad \left\|\mathsf{\Xi}_{\e+i\eta}-\mathsf{\Xi}_{i\eta}\right\|_{\mathscr{B}(\Y_{1}^{-},\X_{0})} \leq \e\,D\end{equation}
where $D$ is the diameter of $\Omega$. Consequently, 
\begin{equation}\label{eq:MeisH}
\left\|\mathsf{M}_{\varepsilon+i\eta}\H-\mathsf{M}_{i\eta}\H\right\|_{\mathscr{B}(\lp)} \leq  \varepsilon\,D\,\|\H\|_{\mathscr{B}(\lp,\Y_{1}^{-})} \qquad \forall \eta \in \R\end{equation}
and, $\left\|\mathsf{\Xi}_{\e+i\eta}\H-\mathsf{\Xi}_{i\eta}\H\right\|_{\mathscr{B}(\lp,\X_{0})} \leq \e\,D \,\|\H\|_{\mathscr{B}(\lp,\Y_{1}^{-})}$ for any $\eta \in \R.$
\end{lemme}

\begin{proof} Let us prove first \eqref{eq:unifGeis}. Given $f \in \X_{0}$ and $(x,v) \in \Omega \times V$, 
\begin{multline*}
\left|\mathsf{G}_{\e+i\eta}f(x,v)-\mathsf{G}_{i\eta}f(x,v)\right|=\left|\int_{0}^{\tau_{-}(x,v)}\left(e^{-\e\,t}-1\right)e^{-i\eta\,t}f(x-tv,v)\d t\right|\\
\leq \int_{0}^{\tau_{-}(x,v)}\left(1-e^{-\e\,t}\right)|f(x-tv,v)|\d t,\end{multline*}
so that
$$\sup_{\eta\in \R}\left\|\mathsf{G}_{\e+i\eta}f-\mathsf{G}_{i\eta}f\right\|_{\lp} \leq \int_{\Gamma_{+}}\d\mu_{+}(x,v)\int_{0}^{\tau_{-}(x,v)}\left(1-e^{-\e\,t}\right)|f(x-tv,v)|\d t.$$
Since $1-e^{-\e\,t} \leq 1$ for any $\e >0,$ $t\geq0$, the dominated convergence theorem combined with \eqref{Eq:G0} gives the result. Let now consider \eqref{eq:Meis} and \eqref{eq:MeisH}. We give the proof for $\mathsf{M}_{i\eta}$, the proof for $\mathsf{\Xi}_{i\eta}$ being exactly the same. Let $\eta \in \R$ be fixed. Let $\varphi \in \Y_{1}^{-}$ be given and $\varepsilon>0$. One has
\begin{multline*}
\left\|\mathsf{M}_{\e+i\eta}\varphi-\mathsf{M}_{i\eta}\varphi\right\|_{\lp}
=\int_{\Gamma_{+}}\left|e^{-(\e+i\eta)\tau_{-}(x,v)}-e^{-i\eta\tau_{-}(x,v)}\right|
\,|\mathsf{M}_{0}\varphi(x,v)|\d\mu_{+}(x,v)\\
=\int_{\Gamma_{+}}\left|\exp(-\e\tau_{-}(x,v))-1\right|\,|\mathsf{M}_{0}\varphi(x,v)|\d\mu_{+}(x,v)\\
\leq C_{0}\e\int_{\Gamma_{+}}\tau_{-}(x,v)\,|\mathsf{M}_{0}\varphi(x,v)|\d\mu_{+}(x,v)\end{multline*}
where 
$C_{0}:=\sup_{s >0}\frac{|\exp(-s)-1|}{s}=1.$
Now, because $\tau_{-}(x,v) \leq D|v|^{-1}$  we get
$$\left\|\mathsf{M}_{\e+i\eta}\varphi-\mathsf{M}_{i\eta}\varphi\right\|_{\lp} \leq \e\,D\int_{\Gamma_{+}}|v|^{-1}\,|\mathsf{M}_{0}\varphi(x,v)|\d\mu_{+}(x,v)=D\,C_{0}\e\|\mathsf{M}_{0}\psi\|_{\lp}$$
where $\psi(x,v)=|v|^{-1}\varphi(x,v)$. Because $\|\mathsf{M}_{0}\psi\|_{\lp}=\|\psi\|_{\lm}=\|\varphi\|_{\Y^{-}_{1}}$ we obtain
$$\left\|\mathsf{M}_{\e+i\eta}\varphi-\mathsf{M}_{i\eta}\varphi\right\|_{\lp} \leq \e\,D\,\|\varphi\|_{\Y^{-}_{1}}$$
which proves \eqref{eq:Meis}. Now,  since $\mathrm{Range}(\H) \subset \Y^{-}_{1}$, one deduces \eqref{eq:MeisH} directly from \eqref{eq:Meis}.\end{proof}

\begin{proof}[Proof of Proposition \ref{propo:convK}] Inequalities \eqref{eq:Meisk} and \eqref{eq:MeisHk} are true for $k=0$ (recall then that $\Y_{0}^{+}=\lp$ thanks to Lemma \ref{lem:Meis}. The proof for general $k$ is exactly the same and is omitted here. Let us focus on \eqref{cor:uniformpower}. The proof is  done by induction on $j \in \N.$ For $j=1$, the result is true, see \eqref{eq:MeisH}. Noticing that, for any $j \in \N$
\begin{multline*}
\left\|\left(\mathsf{M}_{\e+i\eta}\H\right)^{j+1}-\left(\mathsf{M}_{i\eta}\H\right)^{j+1}\right\|_{\mathscr{B}(\lp)}
\leq \left\|\left(\mathsf{M}_{\e+i\eta}\H\right)^{j}-\left(\mathsf{M}_{i\eta}\H\right)^{j}\right\|_{\mathscr{B}(\lp)}\|\mathsf{M}_{\e+i\eta}\H\|_{\mathscr{B}(\lp)} \\
+ \left\|\left(\mathsf{M}_{i\eta}\H\right)^{j}\right\|_{\mathscr{B}(\lp)}\,\left\| \mathsf{M}_{\e+i\eta}\H-\mathsf{M}_{i\eta}\H\right\|_{\mathscr{B}(\lp)}
\end{multline*}
we easily get the result since $\|\mathsf{M}_{i\eta}\H\|_{\mathscr{B}(\lp)}\leq1$. It remains only to prove \eqref{eq:RiemLeb}, i.e.
$$\lim_{|\eta|\to\infty}\sup_{\e \in [0,1]}\left\|\mathsf{G}_{\e+i\eta}f\right\|_{\lp}=0.$$
The proof resorts from Riemann-Lebesgue Theorem, the only slightly delicate point being to make the Riemann-Lebesgue argument uniform with respect to $\e$. We write
\begin{multline*}
\mathsf{G}_{\e+i\eta}f(x,v)=\int_{\R}e^{-i\eta s}h_{x,v}^{\e}(s)\d s, \\
\qquad h_{x,v}^{\e}(s)=\ind_{[0,t_{-}(x,v)]}(s)f(x-sv,v)e^{-\e s}, \qquad s \in \R, (x,v) \in \Gamma_{+}.\end{multline*}
Notice that, since
$$\int_{\Gamma_{+}}\d\mu_{+}(x,v)\int_{\R}\left|h^{\e}_{x,v}(s)\right|\d s \leq \int_{\Gamma_{+}}\d\mu_{+}(x,v)\int_{\R}\left|h^{0}_{x,v}(s)\right|\d s=\|f\|_{\X_{0}} <\infty$$
we deduce from Fubini's Theorem that, for any $\e \geq0$ and $\mu_{+}$-a. e. $(x,v) \in \Gamma_{+}$, the mapping 
$$s \in \R \longmapsto h_{x,v}^{\e}(s)$$
belongs to $L^{1}(\R)$. Thus, according to Riemman-Lebesgue Theorem,
$$\lim_{|\eta|\to\infty}\mathsf{G}_{\e+i\eta}f(x,v)=0 \qquad \text{ for $\mu_{+}$-a. e. $(x,v) \in \Gamma_{+}$}$$
Actually, this convergence can be made uniform with respect to $\e$. Indeed, recalling the classical proof of the Riemann-Lebesgue Theorem, we write
$$\int_{\R}e^{-i\eta s}h^{\e}_{x,v}(s)\d s=-\int_{\R}e^{-i\eta\left(s+\frac{\pi\eta}{|\eta|^{2}}\right)}h_{x,v}^{\e}(s)\d s=-\int_{\R}e^{-i\eta s}h_{x,v}^{\e}\left(s-\frac{\pi\eta}{|\eta|^{2}}\right)\d s$$
where we used that $e^{-i\pi}=-1$. Thus,
$$\int_{\R}e^{-i\eta s}h^{\e}_{x,v}(s)\d s=\frac{1}{2}\int_{\R}e^{-i\eta s}\left[h_{x,v}^{\e}(s)-h_{x,v}^{\e}\left(s-\frac{\pi\,\eta}{|\eta|^{2}}\right)\right]\d s$$
and
$$|\mathsf{G}_{\e+i\eta}f(x,v)| \leq \frac{1}{2}\left\|h_{x,v}^{\e}(\cdot)-h_{x,v}^{\e}\left(\cdot-\frac{\pi\eta}{|\eta|^{2}}\right)\right\|_{L^{1}(\R)}.$$
Now, writing $h_{x,v}^{\e}(s)=e^{-\e s}h_{x,v}^{0}(s)$
we see that
\begin{multline*}
\left\|h_{x,v}^{\e}(\cdot)-h_{x,v}^{\e}\left(\cdot-\frac{\pi\eta}{|\eta|^{2}}\right)\right\|_{L^{1}(\R)}
\leq \int_{\R}e^{-\e s}\left|h_{x,v}^{0}(s)-h_{x,v}^{0}\left(s-\frac{\pi\eta}{|\eta|^{2}}\right)\right|\d s\\
+ \int_{\R}\left|e^{-\e s}-e^{-\e\left(s-\frac{\pi\eta}{|\eta|^{2}}\right)}\right|\,\left|h_{x,v}^{0}\left(s-\frac{\pi\eta}{|\eta|^{2}}\right)\right|\d s\\
\leq \left\|h_{x,v}^{0}(\cdot)-h_{x,v}^{0}\left(\cdot-\frac{\pi\eta}{|\eta|^{2}}\right)\right\|_{L^{1}(\R)}
+\left|1-e^{\e\frac{\pi\eta}{|\eta|^{2}}}\right|\,\|h_{x,v}^{0}\|_{L^{1}(\R)}\,.\end{multline*}
The first term is independent of $\e$ and goes to zero as $|\eta| \to \infty$ owing to the continuity of translation. For the second term, there is $C_{R} >0$ such that 
$$\sup_{\e \in [0,1]}\left|1-e^{\e\frac{\pi\eta}{|\eta|^{2}}}\right| \leq \frac{C_{R}}{|\eta|} \qquad \forall |\eta| >R.$$
This proves that
$$\lim_{|\eta|\to\infty}\sup_{\e \in [0,1]}\left|\mathsf{G}_{\e+i\eta}f(x,v)\right|=0 \qquad \text{ for $\mu_{+}$-a. e. $(x,v) \in \Gamma_{+}$}.$$
Now, since 
$$\sup_{\eta\in \R}\sup_{\e\in [0,1]}\left|\mathsf{G}_{\e+i\eta}f(x,v)\right| \leq \|h^{0}_{x,v}(\cdot)\|_{L^{1}(\R)}$$
and $\ds\int_{\Gamma_{+}}\|h^{0}_{x,v}(\cdot)\|_{L^{1}(\R)}\d \mu_{+}(x,v)=\|f\|_{\X_{0}} <\infty$, we deduce the result from the dominated convergence theorem.\end{proof}

\subsection{Differentiability properties of $\mathsf{M}_{\l}\H,$ $\mathsf{\Xi}_{\l}\H$ and $\mathsf{G}_{\l}$.}
We give here the full proof of Proposition \ref{prop:DerivG}
\begin{proof}[Proof of Proposition \ref{prop:DerivG}] We prove the various points of the Proposition. \\
\noindent \textit{(1)} For $\l \in \C_{+}$ and $f \in \X_{1}$, one has 
$$\dfrac{\d}{\d\l}\mathsf{G}_{\l}f(x,v)=-\int_{0}^{\tau_{-}(x,v)}tf(x-tv,v)e^{-\l t}\d t, \qquad \text{ for a. e.} (x,v)\in \Gamma_{+}.$$
Notice that, for any $f \in \X_{1}$ and any $(x,v) \in \Gamma_{+}$
$$\int_{0}^{\tau_{-}(x,v)}tf(x-tv,v)\d t=\int_{0}^{\tau_{-}(x,v)}t_{+}(x-tv,v)f(x-tv,v)\d t=\mathsf{G}_{0}(t_{+}f)(x,v)$$
since $t_{+}(x-tv,v)=t$. In particular, $\mathsf{G}_{0}(t_{+}f) \in \lp$ since $f \in \X_{1}$ and  we can invoke the dominated convergence theorem to get the conclusion. The result for higher-order derivatives proceed along the same lines. Let us now prove \eqref{lem:G1}. For $f \in \X_{k}$, it holds for $\mu$-a. e. $(x,v) \in \Gamma_{+}$
$$\dfrac{\d^{j}}{\d \lambda^{j}}\mathsf{\mathsf{G_{\lambda}}}f(x,v)=(-1)^{j}\int_{0}^{\tau_{-}(x,v)}s^{j}f(x-sv,v)\exp(-\lambda s)\d s.$$
Introducing $\varphi(x,v)=|f(x,v)|\,t_{+}(x,v)^{j}$, $(x,v) \in \Omega\times \R^{d}$, we get easily that
$$\left|\dfrac{\d^{j}}{\d \lambda^{j}}\mathsf{\mathsf{G_{\lambda}}}f(x,v)\right| \leq \int_{0}^{\tau_{-}(x,v)}\varphi(x-sv,v)\d s=\mathsf{G}_{0}\varphi(x,v).$$
Then, according to \eqref{Eq:G0},
$$\left\|\dfrac{\d^{j}}{\d \lambda^{j}}\mathsf{\mathsf{G_{\lambda}}}f\right\|_{\lp} \leq \|\mathsf{G}_{0}\varphi\|_{\lp} \leq \|\varphi\|_{\X_{0}}$$
For $j \leq k$, it is clear that $\|\varphi\|_{\X_{0}} \leq D^{j}\|f\|_{\X_{j}}\leq D^{j}\|f\|_{\X_{k}}$ and the conclusion follows.\\

\noindent \textit{(2)} For $\varphi \in \Y_{k+1}^{-}$, $\e >0$, $\eta \in\R$ one checks easily that
\begin{multline*}
\dfrac{\d^{k}}{\d\eta^{k}}\mathsf{M}_{\varepsilon+i\eta}\varphi(x,v)-\dfrac{\d^{k}}{\d\eta^{k}}\mathsf{M}_{i\eta}\varphi(x,v)=(-i)^{k}\tau_{-}(x,v)^{k}\\
\left(\exp\left(-(\e+i\eta)\tau_{-}(x,v)\right)-\exp\left(-i\eta\tau_{-}(x,v)\right)\right)\mathsf{M}_{0}\varphi(x,v), 
\end{multline*}
for any $(x,v) \in \Gamma_{-}.$ Therefore
$$
\left|\dfrac{\d^{k}}{\d\eta^{k}}\mathsf{M}_{\varepsilon+i\eta}\varphi(x,v)-\dfrac{\d^{k}}{\d\eta^{k}}\mathsf{M}_{i\eta}\varphi(x,v)\right|
=\tau_{-}(x,v)^{k}|\mathsf{M}_{0}\varphi(x,v)|\left|\exp(-\e(\tau_{-}(x,v))-1\right|
$$
and, reasoning as in Lemma \ref{lem:Meis}, we get
$$\left|\dfrac{\d^{k}}{\d\eta^{k}}\mathsf{M}_{\varepsilon+i\eta}\varphi(x,v)-\dfrac{\d^{k}}{\d\eta^{k}}\mathsf{M}_{i\eta}\varphi(x,v)\right| \leq \e\,D^{j+1}|v|^{-k-1}\left|\mathsf{M}_{0}\varphi(x,v)\right|$$
and the result follows.\end{proof}

We give the proof of Corollary \ref{cor:MHclass}
\begin{proof}[Proof of Corollary \ref{cor:MHclass}] Since the mapping $\l \in \C_{+} \mapsto \mathsf{M}_{\l}\H \in \mathscr{B}(\lp)$ is holomorphic, for any $\e >0$, the mapping
$$\eta \in\R \mapsto \frac{\d^{j}}{\d \eta^{j}}\mathsf{M}_{\e+i\eta}\H \in \mathscr{B}(\lp)$$
is continuous for any $0 \leq j \leq k$. Thanks to \eqref{prop:derMeis}, we can let $\e \to 0$ and conclude that the derivatives exist and are continuous on $\R$. Now, for any $\l \in \overline{\C}_{+}$, one has
$$\dfrac{\d}{\d\l}\mathsf{M}_{\l}\varphi(x,v)=-\tau_{-}(x,v)\exp(-\lambda\tau_{-}(x,v))\varphi(x-\tau_{-}(x,v)v,v)=-\tau_{-}\mathsf{M}_{\l}\varphi$$
so that, by the dominated convergence theorem, 
$\lim_{\l\to0}\frac{\d}{\d\l}\mathsf{M}_{\l}\varphi=-\tau_{-}\mathsf{M}_{0}\varphi$
provided $\varphi \in \Y_{1}^{-}.$ The conclusion follows easily. The proof for $\mathsf{\Xi}_{\l}\H$ proceeds along the same line. For $\varphi \in \Y_{k+1}^{-}$, $\e >0$, $\eta \in\R$ one checks easily that
\begin{multline*}
\dfrac{\d^{k}}{\d\eta^{k}}\mathsf{\Xi}_{\varepsilon+i\eta}\varphi(x,v)-\dfrac{\d^{k}}{\d\eta^{k}}\mathsf{\Xi}_{i\eta}\varphi(x,v)=(-i)^{k}t_{-}(x,v)^{k}\\
\left(\exp\left(-(\e+i\eta)t_{-}(x,v)\right)-\exp\left(-i\eta t_{-}(x,v)\right)\right)\mathsf{\Xi}_{0}\varphi(x,v), 
\end{multline*}
for any $(x,v) \in \Omega\times V.$ So, as in \eqref{prop:derMeis},
$$\left\|\frac{\d^{k}}{\d \eta^{k}}\mathsf{\Xi}_{\varepsilon+i\eta}\H-\frac{\d^{k}}{\d \eta^{k}}\mathsf{\Xi}_{i\eta}\H\right\|_{\mathscr{B}(\lp,\X_{0})} \leq \varepsilon\,D\,\|\H\|_{\mathscr{B}(\lp,\Y_{k+1}^{-})} \qquad \forall \eta \in \R, \qquad \varepsilon >0.$$
We deduce the result as in the previous point.\end{proof}

With the notations of the above proof, we have also the following technical Lemma regarding derivatives of $\mathsf{L}_{N}(i\eta)$ which was used in the proof of Proposition \ref{prop:reguPsif}.
\begin{lemme}\label{lem:estJ} For any $j \in \{1,\ldots,N_{\H}\}$, there exists $\bar{C}_{j} >0$ such that
\begin{equation}\label{eq:estJ}
\left\|\mathsf{L}^{(j)}_{N}(\eta)\right\|_{\mathscr{B}(\lp)} \leq \bar{C}_{j}\left(N+1\right)^{j}\,\|\mathsf{L}_{\floor{\frac{N}{2^{j}}}}(i\eta)\|_{\mathscr{B}(\lp)} \qquad \forall N \geq 2^{j}.\end{equation}
\end{lemme}
\begin{proof} The proof is based upon elementary but tedious computations. For simplicity of notations, we will simply here denote $\|\cdot\|$ for the norm $\|\cdot\|_{\mathscr{B}(\lp)}$ . We notice first that,  since $\mathsf{L}_{N}(i\eta)=\left(\mathsf{L}_{1}(i\eta)\right)^{N}$, one has for the first derivative:
$$\mathsf{L}_{N}^{(1)}(i\eta)=\sum_{r=0}^{N}\mathsf{L}_{r}(i\eta)\mathsf{L}_{1}^{(1)}(i\eta)\mathsf{L}_{N-r}(i\eta)$$
We also denote
Since $\|\mathsf{L}_{1}(i\eta)\| \leq 1$ and 
$$\|\mathsf{L}_{1}^{(1)}(i\eta)\|=\left\|\dfrac{\d}{\d\eta}\mathsf{M}_{i\eta}\H\right\|=\|\tau_{-}\mathsf{L}_{1}(i\eta)\|  \leq \|\mathsf{M}_{0}\|_{\mathscr{B}(\Y_{1}^{-},\Y^{+}_{1})}\|\H\|_{\mathscr{B}(\lp,\Y_{1}^{-})}:=C_{1}$$
and
$$\left\|\mathsf{L}_{N}^{(1)}(i\eta)\right\| \leq C_{1}\sum_{r=0}^{N}\left\|\mathsf{L}_{r}(i\eta)\right\|\,\|\mathsf{L}_{N-r}(i\eta)\| \leq 2C_{1}\sum_{r=0}^{\floor{\frac{N}{2}}}
\left\|\mathsf{L}_{r}(i\eta)\right\|\,\left\|\mathsf{L}_{N-r}(i\eta)\right\|.$$
Since $N-r \geq \floor{\frac{N}{2}}$ for any $0\leq r \leq \floor{\frac{N}{2}}$, we get
\begin{equation}\label{eq:j=1}
\left\|\mathsf{L}_{N}^{(1)}(i\eta)\right\| \leq 2C_{1}\|\mathsf{L}_{\floor{\frac{N}{2}}}(i\eta)\|\sum_{r=0}^{\floor{\frac{N}{2}}}\left\|\mathsf{L}_{r}(i\eta)\right\|, \qquad N \geq 2\mathsf{mp}\end{equation}
which results in 
$$\left\|\mathsf{L}_{N}^{(1)}(i\eta)\right\| \leq C_{1}(N+1)\|\mathsf{L}_{\floor{\frac{N}{2}}}(i\eta)\|$$
and proves the result for $j=1$. Now, for $j=2$, one has
\begin{multline*}
\mathsf{L}^{(2)}_{N}(i\eta)=\sum_{r=0}^{N}\mathsf{L}^{(1)}_{r}(i\eta)\mathsf{L}_{1}^{(1)}(i\eta)\mathsf{L}_{N-r}(i\eta)
+\sum_{r=0}^{N}\mathsf{L}_{r}(i\eta)\mathsf{L}_{1}^{(2)}(i\eta)\mathsf{L}_{N-r}(i\eta)\\
+\sum_{r=0}^{N}\mathsf{L}_{r}(i\eta)\mathsf{L}_{1}^{(1)}(i\eta)\mathsf{L}_{N-r}^{(1)}(i\eta).\end{multline*}
One has again, 
$$\|\mathsf{L}_{1}^{(1)}(i\eta)\| \leq C_{1}, \qquad \|\mathsf{L}_{1}^{(2)}(i\eta)\| \leq C_{2}=\|\mathsf{M}_{0}\|_{\mathscr{B}(\Y^{-}_{2},\Y^{+}_{2})}\|\H\|_{\mathscr{B}(\lp,\Y^{-}_{2})}$$
so that, as before
\begin{equation*}\begin{split}
\left\|\mathsf{L}_{N}^{(2)}(i\eta)\right\| &\leq 2C_{1}\sum_{r=0}^{N}\|\mathsf{L}_{r}(i\eta)\|\,\|\mathsf{L}^{(1)}_{N-r}(i\eta)\|+C_{2}\sum_{r=0}^{N}\|\mathsf{L}_{r}(i\eta)\|\,\|\mathsf{L}_{N-r}(i\eta)\|\\
&\leq 4C_{1}\sum_{r=0}^{\floor{\frac{N}{2}}}\|\mathsf{L}_{r}(i\eta)\|\,\|\mathsf{L}_{N-r}^{(1)}(i\eta)\|
+2C_{2}\sum_{r=0}^{\floor{\frac{N}{2}}}\|\mathsf{L}_{r}(i\eta)\|\,\|\mathsf{L}_{N-r}(i\eta)\|\end{split}\end{equation*}
The last sum is bounded like in the previous step while, for the first sum, we apply \eqref{eq:j=1} to $N-r$ so that
$$\|\mathsf{L}_{N-r}^{(1)}(i\eta)\| \leq 2C_{1}\|\mathsf{L}_{\floor{\frac{N-r}{2}}}(i\eta)\|\sum_{r_{1}=0}^{\floor{\frac{N-r}{2}}}\left\|\mathsf{L}_{r_{1}}(i\eta)\right\|\leq 2C_{1}\|\mathsf{L}_{\floor{\frac{N}{4}}}(i\eta)\|\sum_{r_{1}=0}^{\floor{\frac{N-r}{2}}}\left\|\mathsf{L}_{r_{1}}(i\eta)\right\|$$
and
$$2C_{1}\sum_{r=0}^{\floor{\frac{N}{2}}}\|\mathsf{L}_{r}(i\eta)\|\,\|\mathsf{L}_{N-r}^{(1)}(i\eta)\| \leq (2C_{1})^{2}\|\mathsf{L}_{\floor{\frac{N}{2}}}(i\eta)\|\sum_{r=0}^{\floor{\frac{N}{2}}}\sum_{r_{1}=0}^{\floor{\frac{N-r}{2}}}\|\mathsf{L}_{r_{1}}(i\eta)\|$$
so that
\begin{equation*}\label{eq:j=2}
\left\|\mathsf{L}_{N}^{(2)}(i\eta)\right\| \leq (2C_{1})^{2}\|\mathsf{L}_{\floor{\frac{N}{4}}}(i\eta)\|\sum_{r=0}^{\floor{\frac{N}{2}}}\sum_{r_{1}=0}^{\floor{\frac{N-r}{2}}}\|\mathsf{L}_{r_{1}}(i\eta)\|
+2C_{2}\|\mathsf{L}_{\floor{\frac{N}{4}}}(i\eta)\|\sum_{r=0}^{\floor{\frac{N}{2}}}\|\mathsf{L}_{r}(i\eta)\|.\end{equation*}
This clearly gives the rough estimate (using $\floor{\frac{N-r}{2}}+1 \leq \floor{\frac{N}{2}}+1 \leq \frac{N+1}{2}$), 
\begin{equation}
\label{eq:j=2}
\left\|\mathsf{L}_{N}^{(2)}(i\eta)\right\| \leq \left\|\mathsf{L}_{\floor{\frac{N}{4}}}(i\eta)\right\|\left(\left(C_{1}(N+1)\right)^{2}+C_{2}(N+1)\right)\end{equation}
and proves the result for $j=2$. By a tedious but simple induction argument, we deduce then the result for any $j \in \{0,\ldots,N_{\H}-1\}$. Recall that, for $j \in \N$,
$$\|\mathsf{L}_{1}^{(j)}(i\eta)\|=\|\tau_{-}^{j}\mathsf{M}_{i\eta}\H\|\leq \|\mathsf{M}_{0}\|_{\mathscr{B}(\Y_{j}^{-},\Y_{j}^{+})}\|\H\|_{\mathscr{B}(\lp,\Y_{j}^{+})}:=C_{j}$$
which is finite as long as $j \leq N_{\H}.$
\end{proof}

\subsection{Additional properties of the Dyson-Phillips iterates}\label{sec:appDP}
For technical reasons we need to   introduce a slightly different expression for the iterates $\bm{U}_{k}(t)$ defined in Section \ref{sec:DP} where we allow various boundary operator to enter the construction. More precisely, let us consider a sequence $\left(\H_{n}\right)_{n\in\N}$ of boundary operators 
$$\left(\H_{n}\right)_{n} \subset \mathscr{B}(\lp,\lm), \qquad \|\H_{n}\|_{\mathscr{B}(\lp,\lm)} \leq 1, \qquad \forall n \in \N.$$
We mimic then the construction of \cite{luisa} and define, for any $t\geq0$, $\bm{V}_{0}(t)=\bm{U}_{0}(t)=U_{0}(t)$ and
\begin{defi}\label{defi:Vk}
Let $ t \geq 0$, $k \geq 1$ and $f \in \D_{0}$ be given. For $(x,v) \in \overline{\Omega} \times V$ with $t_{-}(x,v) < t$, there exists a unique $y \in \partial\Omega$ with $(y,v) \in \Gamma_{-}$ and a unique $0 < s < \min(t,\tau_{+}(y,v))$ such that $x=y+sv$ 
and then one sets
$$[\bm{V}_{k}(t)f](x,v)=\left[\H_{k}\B^{+}\bm{V}_{k-1}(t-s)f\right](y,v),$$
We set $[\bm{V}_{k}(t)f](x,v)=0$ if $t_{-}(x,v) \geq t$ and  $\bm{U}_{k}(0)f=0$.
\end{defi}

We establish here the properties of such a sequence of Dyson-Phillips  operators. We first recall the following, taken from \cite[Proposition 3, Corollary 2]{luisa} (see also \cite[Proposition 3.6]{ALMJM}):
\begin{propo}\label{prop:B+U0}
For any $f \in \D_{0}$, any $t \geq 0$, $U_{0}(t)f \in \D(\T_{0})$ and the traces $\B^{\pm}U_{0}(t)f \in L^{1}_{\pm}$ and the mappings $t \geq 0 \mapsto \B^{\pm}U_{0}(t)f \in L^{1}_{\pm}$ 
are continuous with 
$$\int_{0}^{t}\|\B^{+}U_{0}(s)f\|_{\lp} \d s=\|f\|_{\X_{0}}-\|U_{0}(t)f\|_{\X_{0}}, \qquad \forall t \geq 0.$$
\end{propo}
One then proves by induction,  \emph{exactly} as in \cite[Theorem 3.2]{luisa} (see also \cite[Theorem 3.9]{ALMJM}), the following
\begin{theo}\label{theo:VKT} For any $k \geq 1$, $f \in \D_{0}$ one has $\bm{U}_{k}(t)f \in \X_{0}$ for any $t \geq 0$ with
\begin{equation}\label{eq:Vk1}
\|\bm{V}_{k}(t)f\|_{\X_{0}} \leq \,\|f\|_{\X_{0}}.\end{equation}
In particular, $\bm{U}_{k}(t)$ can be extended to be a bounded linear operator, still denoted $\bm{U}_{k}(t) \in\mathscr{B}(\X_{0})$ with
$$\|\bm{V}_{k}(t)\|_{\mathscr{B}(\X_{0})} \leq 1\qquad \forall t \geq 0, k \geq 1.$$
Moreover, the following holds for any $k \geq 1$
\begin{enumerate}
\item $(\bm{V}_{k}(t))_{t \geq 0}$ is a strongly continuous family of $\mathscr{B}(\X_{0})$.
\item For any $f \in \D_{0}$, one has $\bm{V}_{k}(t)f \in \D(\T_{\mathrm{max}})$ for all $t \geq 0$ with 
$$\T_{\mathrm{max}}\bm{V}_{k}(t)f=\bm{V}_{k}(t)\T_{\mathrm{max}}f=\bm{V}_{k}(t)\T_{0}f.$$
\item For any $f \in \D_{0}$ and any $t \geq 0$, the traces $\B^{\pm}\bm{V}_{k}(t)f \in L^{1}_{\pm}$ and the mappings $t \geq 0\mapsto \B^{\pm}\bm{V}_{k}(t)f \in L^{1}_{\pm}$ are continuous. 
\item For any $f \in \D_{0}$, it holds
\begin{equation}\label{eq:B+Vn}
\int_{0}^{t}\|\B^{+}\bm{V}_{k}(s)f\|_{\lp}\d s \leq  \,\int_{0}^{t}\|\H_{k}\B^{+}\bm{V}_{k-1}(s)f\|_{\lp} \d s, \qquad \forall t \geq 0.\end{equation}
%
\end{enumerate}
\end{theo} 

One can actually sharpen estimate \eqref{eq:Vk1} 
\begin{propo}\label{prop:Vnt}
For any $n \geq 1$, $f \in \D_{0}$, one has $\bm{V}_{n}(t)f \in \X_{0}$ for any $t\geq 0$ with
\begin{equation}\label{eq:Vnt}
\left\|\bm{V}_{n}(t)f\right\|_{\X_{0}} \leq \prod_{k=1}^{n}\left\|\H_{k}\right\|_{\mathscr{B}(\lp,\lm)}\,\|f\|_{\X_{0}}.\end{equation}
In particular, $\bm{V}_{n}(t)$ can be extended to a bounded linear operator, still denoted $\bm{V}_{n}(t) \in \mathscr{B}(\X_{0})$ with
$$\left\|\bm{V}_{n}(t)\right\|_{\mathscr{B}(\X_{0})} \leq \prod_{j=1}^{n}\left\|\H_{k}\right\|_{\mathscr{B}(\lp,\lm)}, \qquad n \geq 1,\qquad t\geq0.$$
\end{propo}
 \begin{proof} The proof is made by induction.  Let $f \in \D_{0}$ and $t\geq0$ be fixed. For $n=1$, one deduce from \eqref{10.47} that
$$\|\bm{V}_{1}(t)f\|_{\X_{0}}=\int_{\Gamma_{-}}\d\mu_{-}(z,v)\int_{0}^{\tau_{+}(z,v)}\left|[\bm{V}_{1}(t)f](z+sv,v)\right|\d s.$$
From the definition of $\bm{V}_{1}(t)$, for $\mu_{-}$-a. e. $(z,v) \in \Gamma_{-}$ and $s \in (0,\tau_{+}(z,v))$,  one has
$$\left[\bm{V}_{1}(t)f\right](z+sv,v)=\begin{cases}\left[\H_{1}\B^{+}U_{0}(t-s)f\right](z,v) \qquad &\text{ if } t  > s\\
0 \qquad &\text{ if } t \leq s.
\,,\end{cases}$$
so that
\begin{equation*}\begin{split}
\|\bm{V}_{1}(t)f\|_{\X_{0}}&=\int_{\Gamma_{-}}\d\mu_{-}(z,v)\int_{0}^{\min(t,\tau_{+}(z,v))}\left|\left[\H_{1}(\B^{+}U_{0}(t-s)f)\right](z,v)\right|\d s\\
&\leq \int_{0}^{t} \|\H_{1}(\B^{+}U_{0}(t-s)f)\|_{\lm}\d s 
\end{split}
\end{equation*}
Therefore, 
$$\|\bm{V}_{1}(t)f\|_{\X_{0}} \leq \|\H_{1}\|_{\mathscr{B}(\lp,\lm)}\int_{0}^{t}\|\B^{+}U_{0}(t-s)f\|_{\lp}\d s$$
and one deduces from Prop. \ref{prop:B+U0} that
$$\|\bm{V}_{1}(t)f\|_{\X_{0}}\leq \|\H_{1}\|_{\mathscr{B}(\lp,\lm)}\left(\|f\|_{\X_{0}}-\|U_{0}(t)f\|_{\X_{0}}\right).$$
This proves \eqref{eq:Vnt} for $n=1$. Assume then the result to be true for $n \geq 1$ and let us prove for $n+1$. Using \eqref{10.47} one has, a before,
\begin{equation*}\begin{split}
\left\|\bm{V}_{n+1}(t)f\right\|_{\X_{0}}&=\int_{\Gamma_{-}}\d\mu_{-}(z,v)\int_{0}^{\tau_{+}(z,v)}\left|\left[\bm{V}_{n+1}(t)f\right](z+sv,v)\right|\d s\\
&=\int_{\Gamma_{-}}\d\mu_{-}(z,v)\int_{0}^{\min(t,\tau_{+}(z,v))}\left|\left[\H_{n+1}\B^{+}\bm{V}_{n}(t-s)f\right](z,v)\right|\d s\\
&\leq \int_{0}^{t}\left\|\H_{n+1}\left(\B^{+}\bm{V}_{n}(t-s)f\right)\right\|_{\lm}\d s\\
&\leq \left\|\H_{n+1}\right\|_{\mathscr{B}(\lp,\lm)}\int_{0}^{t}\left\|\B^{+}\bm{V}_{n}(t-s)f\right\|_{\lp}\d s.\end{split}\end{equation*}
Using then \eqref{eq:B+Vn}, we deduce that
\begin{equation*}\begin{split}\left\|\bm{V}_{n+1}(t)f\right\|_{\X_{0}} 
&\leq \left\|\H_{n+1}\right\|_{\mathscr{B}(\lp,\lm)}\int_{0}^{t}\|\H_{n}\B^{+}\bm{V}_{n-1}(s)f\|_{\lm} \d s\\
&\leq \left\|\H_{n+1}\right\|_{\mathscr{B}(\lp,\lm)}\,\|\H_{n}\|_{\mathscr{B}(\lp,\lm)}\int_{0}^{t}\|\B^{+}\bm{V}_{n-1}(s)f\|_{\lp}\d s.\end{split}\end{equation*}
Using repeatedly  \eqref{eq:B+Vn} we deduce
$$\int_{0}^{t}\|\B^{+}\bm{V}_{n-1}(s)f\|_{\lp} \d s \leq \prod_{k=1}^{n-1}\|\H_{k}\|_{\mathscr{B}(\lp,\lm)}\int_{0}^{t}\|\B^{+}U_{0}(s)f\|_{\lp}\d s$$
and Prop. \ref{prop:B+U0} yields 
$$\int_{0}^{t}\|\B^{+}\bm{V}_{n-1}(s)f\|_{\lp} \d s \leq \prod_{k=1}^{n-1}\|\H_{k}\|_{\mathscr{B}(\lp,\lm)}\|f\|_{\X_{0}}$$
and the result follows.
\end{proof}


\begin{thebibliography}{GVdMP}

\bibitem{aoki}
\textsc{K. Aoki, F. Golse,} On the speed of approach to equilibrium for a collisionless gas, \textit{Kinet. Relat. Models} {\bf 4} (2011), 87--107.

\bibitem{arendt}
\textsc{W. Arendt, Ch. J.~K. Batty, M. Hieber, F. Neubrander,} \textbf{Vector-valued Laplace transforms and Cauchy problems},
\textit{Monographs in Mathematics}, \textbf{96}, Birkh\"auser Verlag, Basel, 2001. 


 
 
 \bibitem{luisa}
\textsc{L. Arlotti}, Explicit transport semigroup associated to abstract boundary conditions, \textit{Discrete Contin. Dyn. Syst. A} (2011), \emph{Dynamical systems, differential equations and applications. 8th AIMS Conference. Suppl. Vol. I,} 102--111.

\bibitem{ALMJM}
\textsc{L. Arlotti, B. Lods}, 
An $L^{p}$-approach to the well-posedness of transport equations associated to a regular field: Part II, \textit{Mediterr. J. Math.} {\bf 16} (2019), Paper No. 145, 30 pp.



\bibitem{mjm1}
\textsc{L. Arlotti, J. Banasiak, B. Lods}, A new approach to transport equations associated to a regular field: trace results and well--posedness, \textit{Mediterr. J. Math.}, {\bf 6} (2009), 367--402.


\bibitem{bernou1}
\textsc{A. Bernou}, A semigroup approach to the convergence rate of a collisionless gas,
\textit{Kinet. Relat. Models} {\bf 13} (2020), 1071--1106.

\bibitem{bernou2}
\textsc{A. Bernou, N. Fournier}, A coupling approach for the convergence to equilibrium for a collisionless gas, \texttt{https://arxiv.org/abs/1910.02739}, 2019.




\bibitem{brezis}
\textsc{H. Br\'ezis}, \textbf{\textit{Functional analysis, Sobolev spaces and partial differential equations,}} Universitext. Springer, New York, 2011.


 

\bibitem{briant}
\textsc{M. Briant, Y. Guo,} Asymptotic stability of the Boltzmann equation with Maxwell boundary conditions, \textit{J. Differential Equations} {\bf 261} (2016), 7000--7079.

\bibitem{CIP}
\textsc{C. Cercignani, R. Illner, M. Pulvirenti,} \textbf{\textit{The mathematical theory of dilute gases,}}
Applied Mathematical Sciences, 106. Springer-Verlag, New York, 1994.



\bibitem{ces1}
\textsc{M. Cessenat}, Th\'eor\`emes de traces $L_p$ pour les espaces
de fonctions de la neutronique. \textit{C. R. Acad. Sci. Paris.},
Ser. I {\bf 299} (1984), 831--834.
\bibitem{ces2}
\textsc{M. Cessenat}, Th\'eor\`emes de traces pour les espaces de
fonctions de la neutronique. \textit{C. R. Acad. Sci. Paris.}, Ser.
I {\bf 300} (1985), 89--92.

 \bibitem{chacon}
\textsc{R. V. Chacon, U. Krengel}, Linear modulus of linear operator, 
\textit{Proc. Amer. Math. Soc.}, {\bf 15} (1964), 553--559. 



\bibitem{chill}
\textsc{R. Chill, D. Seifert,} Quantified versions of Ingham's theorem, \textit{Bull. Lond. Math. Soc.} {\bf 48} (2016), 519--532.




\bibitem{EGKM}
\textsc{R. Esposito, Y. Guo, C. Kim, R. Marra,} Non-isothermal boundary in the Boltzmann theory and Fourier law,
\textit{Comm. Math. Phys.} {\bf 323} (2013), 177--239.

\bibitem{grafakos}
\newblock \textsc{L. Grafakos,} 
\newblock \textit{\textbf{Classical Fourier analysis,}} Third edition. 
\newblock Graduate Texts in Mathematics, 249. Springer, New York, 2014.


\bibitem{guo03}
\textsc{Y. Guo,} Decay and continuity of the Boltzmann equation in bounded domains, \textit{Arch. Ration. Mech. Anal.}, {\bf 197} (2010), 713--809. 


\bibitem{kim}
\textsc{J. Jin, C. Kim,} Damping of kinetic transport equation with diffuse boundary condition, preprint, 2020, \texttt{https://arxiv.org/abs/2011.11582}.


\bibitem{kato}
{\sc T. Kato}, \textbf{\textit{Perturbation theory for linear operators}}, Classics in Mathematics, Springer Verlag, 1980.

\bibitem{kuo}
\textsc{H. W. Kuo,} Equilibrating effect of Maxwell-type boundary condition in highly rarefied gas, \textit{J. Stat. Phys.} {\bf 161} (2015), 743--800.

\bibitem{liu}
\textsc{H. W. Kuo, T. P. Liu, L. C. Tsai,} Free molecular flow with boundary effect, \textit{Comm. Math. Phys.} {\bf 318}
(2013), 375--409.



\bibitem{LM-iso}
\textsc{B. Lods, M. Mokhtar-Kharroubi,} On eventual compactness of collisionless kinetic semigroups with non zero velocities, preprint, 2020.

\bibitem{LM-torus}
\textsc{B. Lods, M. Mokhtar-Kharroubi,} About the rate of convergence for collisional linear kinetic equations on the torus, a tauberian approach, in preparation, 2021.

\bibitem{LM-L1}
\textsc{B. Lods, M. Mokhtar-Kharroubi,} A tauberian approach to the long-time behaviour of the linear spatially homogeneous Boltzmann equation, work in preparation, 2021.
\bibitem{LMK-arxiv}
\textsc{B. Lods, M. Mokhtar-Kharroubi,} Quantitative tauberian approach to collisionless transport equations with diffuse boundary operators, preprint, 2020.  \texttt{https://arxiv.org/abs/2005.12583}

\bibitem{LMR}
\textsc{B. Lods, M. Mokhtar-Kharroubi, R. Rudnicki}, Invariant density and time asymptotics for collisionless kinetic equations with partly diffuse boundary operators, \textit{Ann. Inst. H. Poincar\'e Anal. Non Lin\'eaire} {\bf 37} (2020), 877--923.


\bibitem{marek}
\textsc{I. Marek,} Frobenius theory of positive operators: Comparison theorems and applications, \textit{SIAM J. Appl. Math.} {\bf 19} (1970), 607--628. 


\bibitem{MKR}
\textsc{M. Mokhtar-Kharroubi, R. Rudnicki}, 
On asymptotic stability and sweeping of collisionless kinetic equations. \textit{Acta Appl. Math.} \textbf{147} (2017), 19--38. 

 

\bibitem{mmkseifert}
\textsc{M. Mokhtar-Kharroubi, D. Seifert}, Rates of convergence to equilibrium for collisionless kinetic equations in slab geometry, \textit{J. Funct. Anal.} {\bf 275} (2018),  2404--2452. 
 
 
\bibitem{stroock}
\textsc{D. W. Stroock}, \textit{\textbf{Essentials of integration theory for analysis}}, Springer, 2011.


\bibitem{aoki1}
\textsc{T. Tsuji, K. Aoki, F Golse,} Relaxation of a free-molecular gas to equilibrium caused by interaction with vessel wall,
\textit{J. Stat. Phys.} {\bf 140} (2010), 518--543.
 
\bibitem{voigt84}
\textsc{J. Voigt}, Positivity in time dependent linear transport theory,
\textit{Acta Applicandae Mathematicae} \textbf{2} (1984), 311--331.
\bibitem{voigt}
\textsc{J. Voigt}, 	{\textbf{\textit{Functional analytic treatment of the
initial boundary value problem for collisionless gases}}},
Habilitationsschrift, M\"unchen, 1981.
\end{thebibliography}
\end{document}